\documentclass[preprint,11pt]{article}
\usepackage{amsthm}
\usepackage{jmlr2e}
\usepackage{amsmath,amssymb}
\usepackage{array}
\usepackage{cellspace}
\usepackage{multirow}
\usepackage{tabularx}
\usepackage{hyperref}
\usepackage{booktabs}

\usepackage{array}
\usepackage{xcolor}
\usepackage{blindtext}

\allowdisplaybreaks[2]

\newtheorem{assumption}{Assumption}
\usepackage{amsmath}
\usepackage{graphicx}
\usepackage{enumerate}
\usepackage{tikz}
\usepackage{url} 
\usepackage{amsthm,amsmath,amsfonts,amssymb,mathrsfs}
\usepackage{algorithm}
\usepackage{algpseudocode}
\usepackage{multirow}
\usepackage{epstopdf}
\usepackage{epsfig}
\usepackage{subcaption,float}
\usepackage{pgfplots}
\usepackage{dsfont,stmaryrd,bm}
\usepackage{booktabs}

\newcommand{\mbb}{\mathbb}
\newcommand{\mbf}{\mathbf}
\newcommand{\Var}{\text{Var}}
\newcommand{\mcal}{\mathcal}
\newcommand{\zy}[1]{{\color{red}\bf[ZY: #1]}}

\usepackage{threeparttable}
\usepackage{threeparttable}
\newcommand{\bg}{{\boldsymbol g}}
\newcommand{\bh}{{\boldsymbol h}}
\newcommand{\bdSigma}{{\boldsymbol \Sigma}}

\numberwithin{equation}{section}
\usepackage{hyperref}
\hypersetup{
    colorlinks=true,
    linkcolor=blue,
    filecolor=magenta,
    urlcolor=cyan,
}
\allowdisplaybreaks
\usepackage{xr}
\newcommand{\bed}{\begin{definition}}
\newcommand{\eed}{\end{definition}}

\newcommand{\bitem}{\begin{itemize}}
\newcommand{\eitem}{\end{itemize}}

\newcommand{\beqn}{\begin{equation}}
\newcommand{\eeqn}{\end{equation}}
\newcommand{\balign}{\begin{align}}
\newcommand{\ealign}{\end{align}}

\newcommand{\beq}{\begin{equation}}
\newcommand{\eeq}{\end{equation}}

\newcommand{\diag}{\mathrm{diag}}

\usepackage{lastpage}
\jmlrheading{23}{2022}{1-\pageref{LastPage}}{1/21; Revised 5/22}{9/22}{21-0000}{Author One and Author Two}

\ShortHeadings{Unmixing Nonparametric Densities}{Jianqing Fan, Zheng Tracy Ke, Zhaoyang Shi}
\firstpageno{1}

\begin{document}

\title{Optimal Demixing of Nonparametric Densities}

\author{\name Jianqing Fan \email jqfan@princeton.edu \\
       \addr Department of Operations Research and Financial Engineering\\
Princeton University\\
       \AND
       \name Zheng Tracy Ke \email zke@fas.harvard.edu \\
       \addr Department of Statistics\\
Harvard University
        \AND
       \name Zhaoyang Shi \email zshi@fas.harvard.edu \\
       \addr Department of Statistics\\
Harvard University}

\editor{My editor}

\maketitle

\begin{abstract}
Motivated by applications in statistics and machine learning, we consider a problem of unmixing convex combinations of nonparametric densities.   Suppose we observe $n$ groups of samples, where the $i$th group consists of $N_i$ independent samples  from a $d$-variate density $f_i(x)=\sum_{k=1}^K \pi_i(k)g_k(x)$. Here, each $g_k(x)$ is a nonparametric density, and each $\pi_i$ is a $K$-dimensional mixed membership vector. We aim to estimate $g_1(x), \ldots,g_K(x)$.  This problem generalizes topic modeling from discrete to continuous variables and finds its applications in LLMs with word embeddings.


In this paper, we propose an estimator for the above problem, which modifies the classical kernel density estimator by assigning group-specific weights that are computed by topic modeling on histogram vectors and de-biased by U-statistics.  
For any $\beta>0$, assuming that each $g_k(x)$ is in the Nikol'ski class with a smooth parameter $\beta$, 
we show that the sum of integrated squared errors of the constructed estimators 
has a convergence rate that depends on $n$, $K$, $d$, and the per-group sample size $N$. 
We also provide a matching lower bound, which suggests that our estimator is rate-optimal. 
\end{abstract}


\tableofcontents

\section{Introduction} \label{sec:Intro}
Linear unmixing is a technique widely used in hyperspectral imaging \citep{bioucas2012hyperspectral}, chemical engineering \citep{kwan2006novel,ayhan2015use}, and biomedical data analysis \citep{dey2017visualizing, wang2022non}. Given $n$ data points, it aims to find $K$ base vectors (`end members') such that each data point is a convex combination of base vectors, where $K$ is often much smaller than $n$. In this paper, we consider a linear unmixing problem on the space of nonparametric densities. Suppose $f_1(x), \ldots, f_n(x)$ and $g_1(x),\ldots,g_K(x)$ are densities on $\mathbb{R}^d$ satisfying that
\beq \label{model-1}
f_i(x) = \sum_{k=1}^K \pi_i(k)g_k(x), \qquad\mbox{for $1\leq i\leq n$}, 
\eeq
where each $\pi_i$ takes value in the standard probability simplex ${\cal S}_0\subset\mathbb{R}^K$ (i.e., the entries of $\pi_i$ are nonzero and sum to $1$). 
We further assume that $N_i$ samples are drawn from each $f_i(x)$:
\beq \label{model-2}
X_{i1}, X_{i2}, \ldots, X_{iN_i}\; \overset{iid}{\sim}\; f_i(x). 
\eeq
The goal is to use $\{X_{ij}\}_{1\leq i\leq n,1\leq j\leq N_i}$ to estimate the base densities $g_1(x), \ldots,g_K(x)$.  


When $K=1$, all $X_{ij}$'s are independent samples from a single density $g_1(x)$, so the problem reduces to classical nonparametric density estimation. When $K>1$, however, the setting is fundamentally different, with few results in the literature. In this paper, we develop both new methodology and minimax theory for the above problem.

Our study is motivated by several application examples: 
\begin{itemize}
\item {\it Topic modeling with word embeddings}: Traditional topic modeling is based on word count data. Modern large language models (LLMs) provide word embeddings that contain richer information than word counts.  Recently,  
\cite{austern2025poisson} introduce a new topic model that is equivalent to \eqref{model-1}-\eqref{model-2}, where $X_{ij}$ is the contextualized word embedding of the $j$th word in the $i$th document, and each $g_k(x)$ represents an abstract ``topic'', depicting the distribution of the word embedding in topic $k$. The estimated $g_1(x),\ldots, g_K(x)$ are useful for uncovering the underlying topic structure in a corpus. 
\item {\it Archetypal analysis on distributional data}: Archetypal analysis \citep{cutler1994archetypal} aims to express $n$ data points as convex combinations of $K$ latent archetypes. In the original setting, `data' and `archetypes' are feature vectors, but in certain applications they can be distributional data \citep{bauckhage2014kernel,wu2017prototypal}. 
Existing methods typically apply the kernel trick, but the formulation in \eqref{model-1}-\eqref{model-2} provides an alternative solution to archetypal analysis on distributional data. 
\item {\it Decontamination of mutual contamination models}: \cite{katz2019decontamination} consider mutual contamination models, in which a learner observes random samples from different convex combinations of a set of unknown base distributions. The goal is to infer the base distributions (i.e., decontamination). Our model 
in \eqref{model-1}-\eqref{model-2} 
can be viewed a special mutual contamination model, where each base distribution has a continuous density. 
The decontamination problem has several applications, such as in multi-source domain adaption \citep{zhan2024domain}
\end{itemize}

\subsection{Why classical approaches do not work}
Kernel density estimator (KDE) is a text-book method for nonparametric density estimation. It assumes that i.i.d. samples are observed from the density to be estimated. 
Let ${\cal K}(\cdot)$ be a kernel function such that $\int {\cal K}(x)dx=1$, and write ${\cal K}_h(x):= \frac{1}{h^d}{\cal K}\bigl(\frac{1}{h}x \bigr)$ for brevity. 
In light of \eqref{model-2}, the KDE for $f_i(x)$ is 
\beq \label{KDE-fi}
\hat{f}_i^{\text{KDE}}(x) =\frac{1}{N_i}\sum_{j=1}^{N_i} {\cal K}_h(X_{ij}-x), \qquad 1\leq i\leq n.   
\eeq 
However, how to use $\hat{f}_1^{\text{KDE}}(x), \ldots, \hat{f}_n^{\text{KDE}}(x)$ to estimate $g_1(x), \ldots, g_K(x)$ is unclear. 
Moreover, the rate of estimating each $f_i(x)$ only depends on $N_i$, but we expect a much faster rate of convergence on estimating $g_k(x)$, which should depend on $\sum_{i=1}^n N_i$.

In the mixture density estimation (MDE) problem, we observe i.i.d. samples from a density 
$f(x)=\sum_{k=1}^K \beta_k g(x; \lambda_k)$, where $g(x;\lambda)$ is a parametric density family, and $\beta_1,\beta_2,\ldots,\beta_K$ are the mixing weights. This model can be recast as a latent variable model 
and solved by the expectation-maximization (EM) algorithm. 
In a similar spirit, we introduce a parametric version of the model in \eqref{model-1}-\eqref{model-2} by 
\beq \label{NMDE-to-latentV}
\pi_1, \ldots,\pi_n \, \overset{iid}{\sim}\, \mathrm{Dirichlet}(\alpha), \quad\;\; X_{i1}, \ldots, X_{iN_i} | \pi_i \; \overset{iid}{\sim}\;  f_i(x)=\sum_{k=1}^K \pi_i(k) g(\cdot; \lambda_k), 
\eeq
and the parameters can be estimated by the variational EM algorithm  \citep{blei2004variational}.
Unfortunately, 
this requires each $g_k(x)$ to have a parametric form.  

There exist approaches that approximate a nonparametric density by parametric ones, primarily for the case of estimating a single density $g(x)$. The Dirichlet Process Mixture Model (DPMM) \citep{escobar1995bayesian}  expresses a nonparametric density as an infinite mixture of parametric densities, with mixing weights generated from a `stick-breaking construction' process. 
The variational auto-encoder (VAE) expresses a nonparametric density as an infinite mixture of normal densities \citep{demange2023variational}, parametrized by neural networks. 
However, if we use these ideas to extend the model in \eqref{NMDE-to-latentV} to a nonparametric version, several issues arise: How to design a valid estimation procedure? 
Is the obtained $\hat{g}_k(x)$ smooth? What rate of convergence can be guaranteed? There are no clear answers.   

In the machine learning literature, \cite{katz2019decontamination} propose an interesting method for estimating $g_1(x), \ldots, g_K(x)$ when $n=K$. Let $\hat{F}_i(x)$ be the empirical distribution associated with $X_{i1}, \ldots, X_{iN_i}$. This method computes the ``residual'' of each $\hat{F}_i$ after removing a maximal linear effect of $\{\hat{F}_j: j\neq i\}$. These ``residuals'' are then used to construct estimators of base distributions. However, the estimated distributions typically do not admit smooth densities; moreover, this method  relies on strong conditions on the mixing weights $\pi_i$, which may be hard to satisfy in practice. 

\subsection{The induced topic models} \label{subsec:NMDE-to-TM}
Interestingly, the model in \eqref{model-1}-\eqref{model-2} has a connection to the topic model \citep{blei2003latent}, which is widely used in text analysis. In detail, fix a partition of $\mathbb{R}^d$ into $M$ bins, 
\beq \label{def-bins}
\mathbb{R}^d=\cup_{m=1}^M {\cal B}_m,
\eeq
and construct a `multivariate histogram' vector $Y_i\in\mathbb{R}^M$ for the $i$th group of samples by
\beq\label{def-Y}
Y_i(m)= |\{X_{ij}:1\leq j\leq N_i\}\cap {\cal B}_m|, \qquad 1\leq m\leq M. 
\eeq
Following \eqref{model-2}, it is easy to see that $Y_i \sim\mathrm{Multinomial}(N_i, \Omega^{\cal B}_i)$, where $\Omega^{\cal B}_i\in\mathbb{R}^M$ is a vector such that $\Omega^{\cal B}_i(m) = \int_{x\in {\cal B}_m}f_i(x)dx$. Combining this with the expression of $f_i(x)$ in \eqref{model-1}, we immediately have the following result: 
\beq \label{NMDE-to-TM}
Y_i\sim \mathrm{Multinimial}\biggl(N_i,\;\; \sum_{k=1}^K \pi_i(k)g^{\cal B}_k\biggr), \quad \mbox{where } g^{\cal B}_k(m)= \int_{x\in {\cal B}_m} g_k(x)dx. 
\eeq
This is in fact a topic model with $n$ documents, $K$ topics, and a size-$M$ vocabulary, where  $g_1^{\cal B}, \ldots, g_K^{B}\in\mathbb{R}^M$ are the ``topic vectors'' and $\pi_1, \ldots,\pi_n\in\mathbb{R}^K$ are the ``topic weight vectors.''   

Inspired by this connection, \cite{austern2025poisson} propose a two-step method for estimating the densities $g_1(x), \ldots, g_K(x)$:
\begin{itemize}
\item Apply topic modeling on the matrix $Y=[Y_1, Y_2,\ldots, Y_n]$ to obtain $\hat{g}_1^{\cal B}, \ldots, \hat{g}_K^{\cal B}$.
\item For each $1\leq k\leq K$, construct $\hat{g}_k(x)$ from the entries of $\hat{g}_k^{\cal B}$ by kernel smoothing. 
\end{itemize}
This provides the first method for the nonparametric linear unmixing problem in \eqref{model-1}-\eqref{model-2} with a provable rate of convergence. 
%
%
%
Unfortunately, the rate does not match the information-theoretic lower bound when $g_1(x), \ldots, g_K(x)$ have high smoothness (see Section~\ref{subsec:UB}). 

In this paper, we introduce a new estimator, which also leverages the induced topic model in \eqref{NMDE-to-TM}, but in a fundamentally different manner. We show that the new estimator addresses the above limitation and is minimax optimal.

\subsection{A new estimator} \label{subsec:our-method}
Let $Y$ be the same as defined in \eqref{def-Y}. Suppose we have applied an existing topic modeling algorithm on $Y$, and let $\widehat{G}=[\hat{g}_1^{\cal B}, \ldots, \hat{g}_K^{\cal B}]$ be the estimated topic matrix. 
Our estimator is inspired by several key insights:
\begin{itemize}
	\item First, we consider an idealized case where $\pi_1, \ldots, \pi_n$ are known. In this case, we construct an ideal estimator for $g_k(x)$, which is a weighted KDE by leveraging these known $\pi_i$'s.
	\item Next, we develop an estimator of $\pi_1, \ldots, \pi_n$ based on 
	the estimated topic matrix $\widehat{G}$ and  we plug $\hat{\pi}_1, \ldots, \hat{\pi}_n$ into the ideal estimator. This is a natural plug-in estimator for $g_k(x)$, but unfortunately it has a large bias. 
	\item Finally, we de-bias the above plug-in estimator by replacing some quadratic terms by corresponding U-statistics. 
\end{itemize}
The detailed derivation and rationale are deferred to Section~\ref{sec:Method}, but let us describe the estimator first.
With the notation in \eqref{def-bins}-\eqref{def-Y},  we define a matrix $T\in\mathbb{R}^{M\times M}$ with elements
\beq \label{def-T-U}
T_{mm'} =  \sum_{i=1}^n \frac{1}{N_i(N_i-1)}  \sum_{1\leq j\neq j'\leq N_i} U_{ijm}U_{ij'm'}, \quad \mbox{where}\;\; U_{ijm}=1\{X_{ij}\in {\cal B}_m\}. 
\eeq
For any $x\in\mathbb{R}^d$, we define a matrix $S(x)\in\mathbb{R}^{M\times n}$ by 
\beq \label{def-S-K}
S_{mi}(x) = \frac{1}{N_i(N_i-1)}\sum_{1\leq j\neq j'\leq N_i} K_{ij}(x) U_{ij'm}, \quad\mbox{where}\;\;  K_{ij}(x) = {\cal K}_h(x-X_{ij}),
\eeq
where ${\cal K}_h$ is defined in \eqref{KDE-fi}.
Our proposed estimator has the form: 
\beq \label{our-estimator}
\widehat{\boldsymbol g}(x) = \widehat{G}'\widehat{G}(\widehat{G}'T\widehat{G})^{-1}\widehat{G}'S(x){\bf 1}_n,  
\eeq
where for each $1\leq k\leq K$, $\hat{g}_k(x)$ is the $k$th entry of $\widehat{\boldsymbol g}(x)$. 

Here, each $\hat{g}_k(x)$ is a weighted U-statistic, with weights depending on the estimated topic matrix $\widehat{G}$. It has an entirely different form from the estimator in \cite{austern2025poisson}. 
As a consequence, the choice of the bin size $M$ is also fundamentally different.
They choose $M$ as a power of $Nn$, whereas we simply set $M\asymp K\log(Nn)$; 
see Section~\ref{sec:Method}. 

%

\subsection{The minimax optimal rate}
For any fixed $\beta>0$, let $\Theta^*_\beta$ denote the Nikol’ski class \citep{nikol2012approximation} of smooth probability densities with a smoothness parameter $\beta$. 
Write for brevity $\bm{g}(x)=(g_1(x), \ldots, g_K(x))'$ and $\widehat{\bm{g}}(x)=(\hat{g}_1(x), \ldots, \hat{g}_K(x))'$. We denote $\bm{g}(x)\in\Theta^*_\beta$ when each individual $g_k(x)$ belongs to the Nikol’ski class.  One of our main results is the minimax rate of convergence for the integrated squared error: Suppose $N_i\asymp N$ for $1\leq i\leq n$. As $n N\to \infty$, if $K=O((Nn)^{c^*_{\beta,d}})$, where $c^*_{\beta,d}\in (0,1/4)$ is a constant that only depends on $\beta$, then up to a permutation of $\hat{g}_1(x),\ldots,\hat{g}_K(x)$, 
 \beq \label{minimax}
\inf_{\widehat{\bm{g}}(x)} \sup_{\bm{g}(x)\in\Theta^*_\beta} \mathbb{E}\Biggl[\sum_{k=1}^K \int [\hat{g}_k(x)-g_k(x)]^2dx \Biggr]\;\; \asymp\;\; K\left(\frac{K}{N n}\right)^{\frac{2\beta}{2\beta+d}}. 
\eeq

To establish \eqref{minimax}, we need a minimax upper bound and a minimax lower bound. 
The minimax upper bound is obtained by analyzing the error rate of our estimator. In the proof, we have developed a large-deviation inequality for a specific class of incomplete U-statistics and an integrated variance bound for incomplete U-processes. The minimax lower bound is proved by choosing a least-favorable configuration to uncover the role of a diverging $K$.

%
%

\subsection{Summary, organization and notations}

We consider unmixing convex combinations of nonparametric densities, a problem with applications in topic modeling, archetypal analysis and multi-source domain adaption. Compared to classical nonparametric density estimation, this problem is more challenging and much less studied. Our contributions include:
\begin{itemize}
\item We propose a new estimator, which are significantly different from existing methods (e.g., \cite{katz2019decontamination, austern2025poisson}).
\item We derive an explicit rate of convergence for our estimator. Our results allow both $N$ and $K$ to grow with $n$ and cover a full range of smoothness on the base densities. 
\item We provide a matching lower bound and show that our estimator is optimal. In contrast, existing estimators either have no explicit error rate or suffer a non-optimal rate.  
\item As a component of our analysis, we also provide a new theoretical result on estimating a topic model under a growing $K$, which is of independent interest. 
\end{itemize}

The remainder of this paper is organized as follows: 
In Section~\ref{sec:Method}, we derive our estimator and explain the insights. Section~\ref{sec:TM} discusses the topic modeling on $Y$, focusing an existing algorithm Topic-SCORE. 
Our main theoretical results are presented in Section~\ref{sec:main}. 
Simulation studies are contained in Section~\ref{sec:Simu}.  
Sections~\ref{sec:UBproof}-\ref{sec:LBproof} prove the minimax upper bound and the minimax lower bound, respectively. Section~\ref{sec:Discuss} concludes the paper with discussions. Proofs of secondary lemmas are relegated to the supplementary material.

\section{Topic-Weighted Kernel Density Estimator} \label{sec:Method}

In this section, we explain the rationale of our proposed estimator \eqref{our-estimator}.  We show that it is equivalent to a de-biased weighted kernel density estimator with estimated weights from topic modeling.   

\subsection{An oracle estimator when $\Pi$ is known} \label{subsec:oracle}

We first consider an oracle setting where the mixed membership matrix $\Pi=[\pi_1,\ldots,\pi_n]'\in\mbb{R}^{n\times K}$ is given.  Let ${\boldsymbol g}(x)=(g_1(x),\ldots,g_K(x))'\in\mathbb{R}^K$ and ${\boldsymbol f}(x)=(f_1(x),\ldots,f_n(x))'\in\mathbb{R}^n$. It follows from the model in \eqref{model-1} that
\[
{\boldsymbol f}(x) = \Pi {\boldsymbol g}(x) \qquad \mbox{or} \qquad  {\boldsymbol g}(x) =  (\Pi'\Pi)^{-1}\Pi'  {\boldsymbol f}(x)
\]
The unknown density vector ${\boldsymbol f}(x)$ can be estimated by $\widehat{\boldsymbol f}^{\text{KDE}}(x)=(\hat{f}^{\text{KDE}}_1(x),\ldots,\hat{f}^{\text{KDE}}_n(x))'$, where $\hat{f}_i^{\text{KDE}}(x)$ is given by \eqref{KDE-fi}.  This leads naturally to 
\[
\widehat{\boldsymbol g}^{\text{oracle}}(x) = (\Pi'\Pi)^{-1}\Pi' \widehat{\boldsymbol f}^{\text{KDE}}(x).  
\]
The $k^{th}$ component of the above estimator can be written as
\beq \label{oracle-equivalence} 
\hat{g}_k^{\text{oracle}}(x) = \frac{1}{n}\sum_{i=1}^n b_i(k)\hat{f}_i^{\text{KDE}}(x), 
\eeq
where $b_i(k)$ is given by
\beq \label{B}
    B=(n^{-1}\Pi'\Pi)^{-1}\Pi':=[b_1,b_2, \ldots,b_n]\in\mbb{R}^{K\times n}.
\eeq
It is the weight on the $i$th group of samples when they are used to estimate $g_k(x)$.
Putting more explicitly, we propose a weighted kernel density estimator (KDE):
\beq  \label{oracle}
\hat{g}_k^{\text{oracle}}(x)= \frac{1}{n}\sum_{i=1}^n b_i(k)\frac{1}{N_i} \sum_{j=1}^{N_i}{\cal K}_h(x-X_{ij}), \qquad 1\leq k\leq K,
\eeq




The rate of convergence of this oracle estimator can be fairly easily derived (see the supplement), and it matches the lower bound to be introduced in Section~\ref{sec:main}. We thereby use this estimator as a good starting point.

\subsection{Estimation of $\Pi$ and the plug-in estimator} \label{subsec:estimate-Pi}
To extend the oracle estimator to the real case, we need an estimate of $\Pi$. In light of the induced topic model in \eqref{NMDE-to-TM}, we can apply topic modeling to $Y$. 
However, if we directly apply an existing topic modeling algorithm on estimating $\Pi$, the rate of convergence is not fast enough. Assuming $K$ is finite, the minimax optimal rate is \citep{klopp2023assigning,wu2023sparse}
\[
\max_{1\leq i\leq n}\|\hat\pi_i-\pi_i\|_1=O_{\mathbb{P}}\bigl(N^{-1/2}\bigr).
\]
This rate does not even decrease with $n$; consequently, the plug-in error caused by estimating $\Pi$ is too large. In order to get the desirable rate, we must account for the dependence among $\hat{\pi}_i$'s. This is made possible if 
\beq \label{hat-pi-requirement}
\mbox{each $\hat{\pi}_i$ can be well-approximated by a simple function of $X$}. 
\eeq

To find an estimator to satisfy \eqref{hat-pi-requirement}, we notice that, while the rate on estimating $\pi_i$ is slow, the rate on estimating the topic matrix $G$ can be made very fast. In fact, the minimax optimal rate \citep{ke2024using,bing2020fast} is 
\beq \label{TM-minimax-rate}
\sum_{k=1}^K\|\hat{g}_k-g_k\|_1 =\widetilde{O}_{\mathbb{P}}\left( K\sqrt{\frac{KM}{Nn}}\right).  
\eeq
Here, $M$ is the bin size that a user can choose. If we choose $M$ appropriately small, then the rate on estimating $G$ can be much faster than the optimal rate in \eqref{minimax}, 
so that we can almost treat $G$ as known. 

The above observation inspires us to consider a regression estimator of $\Pi$. In detail, let $L\in\mathbb{R}^{n\times n}$ be a diagonal matrix whose $i$th diagonal entry is equal to $N_i$. It follows from \eqref{NMDE-to-TM} that $\mathbb{E}Y = G\Pi'L$. We thereby estimate $\Pi$ by
\begin{align}\label{Pi-estimate}
\widehat{\Pi}= L^{-1}Y'\widehat{G}(\widehat{G}'\widehat{G})^{-1}. 
\end{align}
Since the plug-in error rate in $\widehat{G}$ is much faster than the desirable rate,  we can almost treat $G$ as known and approximate the $\widehat{\Pi}$ in \eqref{Pi-estimate} by 
\beq \label{Pi-proxy}
\widehat{\Pi}\quad \approx\quad \widehat{\Pi}^*:= L^{-1}Y'G(G'G)^{-1}. 
\eeq
The right hand side of \eqref{Pi-proxy} is a linear function of $Y$, and $Y$ is constructed from $X_{ij}$'s explicitly. Hence, the estimator $\widehat{\Pi}$ satisfies \eqref{hat-pi-requirement} as we hope. 

Note that this estimator does not yield a faster rate on estimating $\Pi$ than existing estimators (e.g., \cite{klopp2023assigning,wu2023sparse}). The advantage lies in avoiding a major analytical hurdle when we use $\widehat{\Pi}$ to construct other estimators: Since the dependence among $\hat{\pi}_1, \ldots, \hat{\pi}_n$ can be quantified explicitly as in \eqref{Pi-proxy}, we obtain a much tighter control on the plug-in error. 
On the other hand, this advantage is tied to our problem: As the choice of bins vary, we have many induced topic models, and all of them share the same $\Pi$. This gives us the freedom to pick a proper induced topic model so that  \eqref{Pi-proxy} holds.



We now plug the $\widehat{\Pi}$ in \eqref{Pi-estimate} into the oracle estimator in \eqref{B}-\eqref{oracle}. Write for brevity
\beq \label{def-hat-Q}
\widehat{Q}:=  \widehat{G}(\widehat{G}'\widehat{G})^{-1},\qquad \widetilde{W}:=  n (\widehat{Q}'YL^{-2} Y'\widehat{Q})^{-1} \widehat{Q}'.
\eeq
We use $\tilde{w}_i'\in\mathbb{R}^K$ to denote the $i$th row of $\widetilde{W}$. 
It is not hard to see that the plug-in estimator has the following form:  
\beq \label{naiveEstimate}
\hat{g}_k^{\text{plug-in}}(x):= \frac{1}{n}\sum_{i=1}^n \frac{1}{N_i^2}\sum_{j=1}^N \sum_{m=1}^M \tilde{w}_{m}(k) {\cal K}_h(x-X_{ij})Y_{i}(m),\quad 1\leq k\leq K.
\eeq
We write $\widehat{\boldsymbol g}^{\text{plug-in}}(x)=(\hat{g}^{\text{plug-in}}_1(x),\ldots, \hat{g}^{\text{plug-in}}_K(x))'$ for brevity.

\subsection{De-biasing the plug-in estimator} \label{subsec:de-biasing}
Unfortunately, the plug-in estimator in \eqref{naiveEstimate} has a non-negligible bias. 
In this subsection, we explain where the bias comes from and introduce a de-biasing technique.  
%
%
%
The following lemma is proved in the supplemental material: 
\begin{lemma} \label{lem:naiveEstimate}
Let $\mathbb{R}^d =\cup_{m=1}^M {\cal B}_m$ be the bins and $\widehat{Q}$ be as in \eqref{def-hat-Q}. Write  $U_{ijm}=1\{X_{ij}\in {\cal B}_m\}$, for all $1\leq i\leq n$, $1\leq j\leq N_i$ and $1\leq m\leq M$. Then, the plug-in estimator in \eqref{naiveEstimate} satisfies that $\widehat{\boldsymbol g}^{\text{plug-in}}(x)=\widetilde{\boldsymbol \Sigma}^{-1}\widetilde{\boldsymbol h}(x)$, where $\widetilde{\boldsymbol h}(x)\in\mathbb{R}^K$ and $\widetilde{\boldsymbol \Sigma} \in\mathbb{R}^{K\times K}$ are defined by
\begin{align*}
\widetilde{\boldsymbol h}_k(x) & = \frac{1}{n}\sum_{i=1}^n \frac{1}{N_i^2}\sum_{m=1}^M\sum_{j,j'=1}^{N_i} \widehat{Q}_{mk} {\cal K}_h(x-X_{ij})U_{ij'm},\quad 1\leq k\leq K,\cr
\widetilde{\boldsymbol \Sigma}_{k\ell} &= \frac{1}{n}\sum_{i=1}^n \frac{1}{N_i^2}\sum_{m, m'=1}^M \sum_{j,j'=1}^{N_i} \widehat{Q}_{mk}\widehat{Q}_{m'\ell}U_{ijm}U_{ij'm'}, \quad 1\leq k,\ell\leq K. 
\end{align*}
\end{lemma}

Lemma~\ref{lem:naiveEstimate} gives an equivalent form of the plug-in estimator. We first study the matrix $\widetilde{\boldsymbol \Sigma}$. Let $\boldsymbol{U}_{ij}\in\mathbb{R}^M$ be the vector consisting of $\{U_{ijm}: 1\leq m\leq M\}$. For each $i$, $\boldsymbol{U}_{i1}, \ldots, \boldsymbol{U}_{iN_i}$ are i.i.d. from $\mathrm{Multinomial}(1, \xi_i)$, with $\xi_i := \mathbb{E}[\boldsymbol{U}_{i1}]$. It is seen that
\[
\widetilde{\boldsymbol \Sigma} = \widehat{Q}'\biggl[ \frac{1}{n}\sum_{i=1}^n \widehat{\xi_i\xi_i'}\biggr] \widehat{Q}, \qquad\mbox{where}\quad \widehat{\xi_i\xi_i'}:=\frac{1}{N^2_i}\sum_{j,j'=1}^{N_i}\boldsymbol{U}_{ij}\boldsymbol{U}'_{ij'}. 
\]
However, $\mathbb{E}[\boldsymbol{U}_{ij}\boldsymbol{U}'_{ij'}]\neq \xi_i\xi_i'$ when $j=j'$. Therefore, the estimator $\widehat{\xi_i\xi_i'}$ is biased. To fix this issue, we consider the following U-statistics, which employ a diagonal deletion strategy by excluding $j=j'$ in the sum:
\[
\widehat{\xi_i\xi_i'}^*:=\frac{1}{N_i(N_i-1)}\sum_{1\leq j\neq j'\leq N_i}\boldsymbol{U}_{ij}\boldsymbol{U}'_{ij'}.
\]
This is an unbiased estimator of $\xi_i\xi_i'$. 
Replacing $\widehat{\xi_i\xi_i'}$ by this estimator gives a modification of $\widetilde{\boldsymbol \Sigma}$ in \eqref{def-h-Sigma} below. We also modify $\widetilde{\boldsymbol h}_k(x)$ by a similar diagonal deletion strategy.  
%
%
Let $\widehat{\boldsymbol h}(x)\in\mathbb{R}^K$ and $\widehat{\boldsymbol \Sigma} \in\mathbb{R}^{K\times K}$ be such that 
\begin{align} \label{def-h-Sigma}
\widehat{\boldsymbol h}_k(x) & = \frac{1}{n}\sum_{i=1}^n \frac{1}{N_i(N_i-1)} \sum_{m=1}^M\sum_{1\leq j\neq j'\leq N_i} \widehat{Q}_{mk} {\cal K}_h(x-X_{ij})U_{ij'm},\quad 1\leq k\leq K,\cr
\widehat{\boldsymbol \Sigma}_{k\ell} &= \frac{1}{n}\sum_{i=1}^n \frac{1}{N_i(N_i-1)} \sum_{m, m'=1}^M \sum_{1\leq j\neq j'\leq N_i} \widehat{Q}_{mk}\widehat{Q}_{m'\ell}U_{ijm}U_{ij'm'}, \;\; 1\leq k,\ell\leq K. 
\end{align}
Our final estimator of $\boldsymbol{g}(x)$ is the diagnonal deletion version:
\beq \label{g-estimate-final}
\widehat{{\boldsymbol g}}(x)= (\hat{g}_1(x), \ldots, \hat{g}_K(x))':= \widehat{\boldsymbol \Sigma}^{-1}\widehat{\boldsymbol h}(x). 
\eeq 

Recall that 
$\widehat{Q}=  \widehat{G}(\widehat{G}'\widehat{G})^{-1}$ and $Q = G(G'G)^{-1}$. 
Using the matrices $S$ and $T$ defined in \eqref{def-T-U}-\eqref{def-S-K}, we can re-express $\widehat{\bh}(x)$ and $\widehat{\bdSigma}$ in \eqref{def-h-Sigma} by
\begin{align} \label{h-Sigma-matrix-form}
& \widehat{\boldsymbol h}(x) = n^{-1}\widehat{Q}'S(x) {\bf 1}_n = n^{-1} (\widehat{G}'\widehat{G})^{-1}\widehat{G}'S(x) {\bf 1}_n,\cr
& \widehat{\boldsymbol \Sigma} = n^{-1}\widehat{Q}'T\widehat{Q} = n^{-1} (\widehat{G}'\widehat{G})^{-1} \widehat{G}'T\widehat{G}(\widehat{G}'\widehat{G})^{-1}.
\end{align}
It follows immediately that 
\beq \label{our-estimator-repeat}
\widehat{\boldsymbol g}(x) = \widehat{G}'\widehat{G}(\widehat{G}'T\widehat{G})^{-1}\widehat{G}'S(x) {\bf 1}_n. 
\eeq
This simplifies the expression in \eqref{g-estimate-final}, which is also the one we introduce in \eqref{our-estimator}.

\subsection{The choice of tuning parameters and bins}  \label{subsec:tuning}

Our method requires two tuning parameters, the bandwidth $h$ and the bin size $M$. In addition, given $M$, we need to choose the bins ${\cal B}_1, \ldots, {\cal B}_M$. 
In this subsection, we discuss how to determine them in practice.

First, based on our theory in Section~\ref{sec:main}, the optimal  $M$ is at the order of $K\log(Nn)$, which does not depend on unknown parameters. In Section~\ref{sec:Simu}, we also observe that the performance of our estimator is insensitive to $M$ in a wide range (see Figure~\ref{fig:handM}). 
We thus fix $M$ as
\beq \label{chooseM}
M = 2 \lfloor K\log(Nn)\rfloor. 
\eeq

Next, we discuss how to choose the bins ${\cal B}_1, \ldots, {\cal B}_M$. 
Let $\bar{g}(x)=\frac{1}{K}\sum_{k=1}^K g_k(x)$.
For our theoretical analysis to carry through,  we hope: 
\beq \label{cond-bin}
\int_{x\in {\cal B}_m} \bar{g}(x)dx \asymp 1/M, \qquad\mbox{for all bins }{\cal B}_1, {\cal B}_2, \ldots, {\cal B}_M. 
\eeq
 When $\bar{g}(x)$ has a bounded support and is uniformly lower bounded by a constant, we can use a simple equi-size cubic bins, which guarantees \eqref{cond-bin}. 
In the more general case where $\bar{g}(x)$ has unbounded support, we can show that as long as the $K$ entries of $\alpha=\Pi'\diag(N_1,\ldots, N_n){\bf 1}_n$ are of the same order, 
\eqref{cond-bin} is equivalent to
\beq \label{cond-bin-2}
\mathbb{E}[\#\{X_{ij}\in {\cal B}_m\}] \asymp \bar{N}n/M. 
\eeq
This inspires a practical choice of bins: we make sure that the total number of samples falling into each bin is of the order $\bar{N}n/M$. For example, when $d=1$, we can sort 
$X_{ij}$'s and let $\delta_m$ be the $m/M$ quantile. The bins are 
\beq \label{bins-d=1}
(-\infty, \delta_1] \cup (\delta_1, \delta_2] \cup \ldots \cup (\delta_{M-2}, \delta_{M-1}] \cup  (\delta_{M-1}, \infty).  
\eeq

Finally, we describe how to select $h$. The optimal $h$ depends on the unknown smoothness of base densities. Therefore, we face the same challenge as in classical nonparametric density estimation. 
Fortunately, there is a careful study of bandwidth selection in that literature. One approach is to minimize the asymptotic mean integrated error squared error (AMISE), which can be estimated by leave-one-out cross validation (abbreviated as CV; e.g., see \cite{park1993cross}) or plug-in estimators (e.g., see \cite{hall1991optimal}). 
Another approach is the Lepski's method \citep{lepskii1992asymptotically}, which takes an ordered sequence of $h$ and finds $\hat{h}$ such that the resulting estimator does not differ too much from all estimators from smaller bandwidths. We can possibly extend these approaches to our setting. For convenience, we only consider the CV approach. Following \citet[Page 135-137]{wasserman2006all}, we define AMSE as 
\[
\text{AMISE}(h; \widehat{\bm{g}}) = \int \|\widehat{\bm{g}}(x)\|^2dx - 2\int \bigl\langle\bm{g}(x), \widehat{\bm{g}}(x)\bigr\rangle dx:= \int \|\widehat{\bm{g}}(x)\|^2dx - J(h; \widehat{\bm{g}}). 
\]
Only the second term  $\bm{g}(x)=B\bm{f}(x)$ is unknown, where $B =(n^{-1}\Pi'\Pi)^{-1}\Pi'=[b_1,b_2, \ldots,b_n]$ is as in \eqref{B}. Let $\{X^*_{ij}\}$ be an independent copy of $\{X_{ij}\}$. It is seen that 
\[
J(h; \widehat{\bm{g}}) = \frac{2}{n}\sum_{i=1}^n  \int  b_i' \widehat{\bm{g}}(x) f_i(x)dx = \frac{2}{n}\sum_{i=1}^n  \Biggl\{
\frac{1}{N_i}\sum_{j=1}^{N_i}
b_i'\mathbb{E}\bigl[\widehat{\bm{g}}(X^*_{ij})\bigr]\Biggr\}. 
\]
Borrowing the CV idea, we leave-out each $X_{ij}$ and use the remaining samples to construct $\widehat{\bm{g}}_{(-ij)}$. This also yields $\widehat{\Pi}_{-(ij)}$ as an intermediate quantity, which we plug into \eqref{Pi-estimate} to obtain $\widehat{B}_{-(ij)}=[\hat{b}_{1(-ij)}, \ldots, \hat{b}_{n(-ij)}]$.  
We select $h$ by minimizing the following criteria: 
\beq \label{CV}
\widehat{\text{AMISE}}(h; \widehat{\bm{g}})
=
\int \|\widehat{\bm{g}}(x)\|^2 dx - \frac{2}{n} \sum_{i=1}^n 
\Biggl[ \frac{1}{N_i} \sum_{j=1}^{N_i} \hat{b}_{i (-ij)}'\widehat{\bm{g}}_{(-ij)}(X_{ij})\Biggr]. 
\eeq

This procedure requires conducting topic modeling for $\sum_{i=1}^nN_i$ times. To save the computational cost, we propose a modified version that only runs topic modeling once. 
Note that $B=\boldsymbol{\Sigma}^{-1}\Pi'$. We plug in $\widehat{\boldsymbol{\Sigma}}$ in \eqref{h-Sigma-matrix-form} and $\widehat{\Pi}$ in \eqref{Pi-estimate} to get $\widehat{B}=n(\widehat{G}'\widehat{G})(\widehat{G}'T\widehat{G})^{-1}\widehat{G}'YL^{-1}$.
Here, only $\widehat{G}$ comes from topic modeling. Similarly, in the expression of $\widehat{\bm{g}}$ in \eqref{our-estimator-repeat}, only $\widehat{G}$ comes from topic modeling. Therefore, we use all data to obtain $\widehat{G}$ but perform leave-one-out in the remaining quantities. This gives the following leave-one-out estimators:
\begin{align} \label{CV-new}
& \widetilde{B}_{(-ij)} =n(\widehat{G}'\widehat{G})\bigl[\widehat{G}'T_{(-ij)}\widehat{G}\bigr]^{-1}\widehat{G}'(YL^{-1})_{-(ij)}, \cr
& \widetilde{\bm{g}}_{-(ij)}(x) = \widehat{G}'\widehat{G}\bigl[\widehat{G}'T_{(-ij)}\widehat{G}\bigr]^{-1}\widehat{G}'S_{(-ij)}(x) {\bf 1}_n,
\end{align}
where $T_{(-ij)}$, $S_{(-ij)}(x)$ and $(YL^{-1})_{-(ij)}$ are the corresponding leave-one-out versions of $T$, $S(x)$ and $YL^{-1}$, respectively. 
Let $\widetilde{\text{AMISE}}(h; \widehat{\bm{g}})$ be a counterpart of \eqref{CV}, where we modify  the second term using the quantities in \eqref{CV-new}. 
We select $h$ by minimizing $\widetilde{\text{AMISE}}(h; \widehat{\bm{g}})$. 
This procedure only requires conducting topic modeling once.


\subsection{Comparison with existing estimators} \label{subsec:comparison}
There are few existing methods for our problem. \cite{katz2019decontamination} consider a similar model where the base distributions do not necessarily have smooth densities. They define a quantity 
\[
\begin{array}{rl}
\kappa^*(f_i| \{f_j, j\neq i\})= \max\bigl\{\kappa\in [0,1]:  & \mbox{there exists $R_i$ such that $f_i=\kappa R_i + \sum_{j\neq i}\nu_j f_j$ for}\\
& \mbox{some $\nu_j\geq 0$ with $\sum_{j\neq i}\nu_j=1-\kappa$}.\bigr\}
\end{array}
\]
The resulting $R_i$ is also a distribution and called a ``residual'' of $f_i$. \cite{katz2019decontamination} show that $R_i$ can be estimated from the empirical CDFs $\hat{F}_1, \ldots, \hat{F}_n$. They also show that when $n=K$ and when $\Pi$ satisfies certain conditions, 
the base distributions can be recovered by recursively computing such ``residuals.'' 
The key insight is to exploit a simplex geometry in the distribution space. 

However, this approach imposes restrictive conditions on $\Pi$, which essentially requires that each $\pi_i$ is close to being degenerate. In contrast, our method bypasses this limitation by considering the induced topic model. While topic modeling also relies on a simplex geometry, the geometry arises in the standard Euclidean space, which can be effectively exploited using existing topic modeling algorithms, hence requiring very mild conditions on $\Pi$. 

In theory, \cite{katz2019decontamination} only establish a consistency guarantee but do not provide any explicit rate of convergence. In contrast, we not only derive the explicit rate of convergence for our estimator but also show that it is minimax optimal.

Our method is connected to the one in \cite{austern2025poisson} in leveraging the induced topic model. 
However, the two methods differ fundamentally in how to use the estimated topic model. 
 \cite{austern2025poisson} treat the topic matrix $G$ as a ``discretization'' of the continuous densities $g_1(x), \ldots, g_K(x)$. 
 The bin size $M$ leads to a bias-variance trade-off: on one hand, $M$ needs to be properly large so that $G$ contains sufficient information of $g_1(x), \ldots, g_K(x)$; 
on the other hand, since $M$ affects the estimation error on $G$, it cannot be too large. 
Unfortunately, in some parameter regime (e.g., when $\beta>1$), there does not exist a sweet spot of $M$. This explains why their method cannot be optimal across the whole parameter regime.  

In contrast, our estimator builds on the oracle estimator in \eqref{oracle}, which is already minimax optimal across the whole parameter regime. We only need to develop a $\widehat{\Pi}$ such that the plug-in error is negligible. 
This seems challenging, as $\Pi$ is a large matrix. Fortunately, we leverage the induced topic model to design a $\widehat{\Pi}$ in which the dependence among its entries can be well characterized. This allows us to have a sharp control of the plug-in error and eventually prove that our estimator has the same error rate as the oracle estimator, thus being optimal.

\section{Topic Modeling: Background and Algorithms} \label{sec:TM}
One step in our estimator is to conduct topic modeling on the matrix $Y$. In this section, we review the topic model and existing topic modeling algorithms. In particular, we describe Topic-SCORE \citep{ke2024using}, the algorithm we use in our numerical experiments and theoretical analysis.

\subsection{Topic modeling algorithms and their error rates} \label{subsec:TM-review}
Suppose we observe $n$ documents written on a vocabulary of $M$ words. Let $Y_i\in\mathbb{R}^M$ be the word count vector for the $i$th document, where $Y_i(m)$ is the count of the $m$th word, and $N_i=\sum_{m=1}^M Y_i(m)$ is the total length of this document. Let $g_1, g_2,\ldots, g_K\in\mathbb{R}^M$ be $K$ topic vectors, each being a probability mass function (PMF) and representing a discrete distribution on the vocabulary. 
Each document $i$ is associated with a topic weight vector $\pi_i\in\mathbb{R}^K$, where $\pi_i(k)$ is the fractional weight this document puts on the $k$th topic, for $1\leq k\leq K$.  The topic model assumes that $Y_1, Y_2, \ldots, Y_n$ are independent, and 
\beq \label{topic-model}
Y_i\sim \mathrm{Multinomial}(N_i, \Omega_i), \qquad \Omega_i = \sum_{k=1}^K \pi_i(k)g_k. 
\eeq
The goal of topic modeling is to use the data matrix $Y=[Y_1, \ldots,Y_n]$ to estimate the parameters $G=[g_1, \ldots,g_K]$ and $\Pi=[\pi_1,\ldots, \pi_n]'$. 

There are many existing methods for this problem. They can be roughly divided into two categories depending on the key assumption: In the first category, it is assumed that $\pi_i$'s are i.i.d. from a Dirichlet prior. 
A popular method in this category is the latent Dirichlet allocation \citep{blei2003latent}. 
In the second category, it is assumed that each topic has at least one anchor word, where a word $m$ is an anchor word if the $m$th row of $G$ has only one nonzero entry. 
Representative methods include the NMF approach \citep{arora2012learning}, the Topic-SCORE algorithm \citep{ke2024using}, and the LOVE algorithm \citep{bing2020fast}. 

We can plug in any of these algorithms into our estimator. However, to achieve the best performance, we prefer a topic modeling algorithm that has a fast rate of convergence. Define the $\ell^1$-error as   (subject to a column permutation of $\widehat{G}$)
$${\cal L}(\widehat{G}, G)=\sum_{1\leq m\leq M}\sum_{1\leq k\leq K}|\widehat{G}_{mk}-G_{mk}|.
$$
Assuming $K$ is fixed, \cite{ke2024entry} summarize the rates of convergence of ${\cal L}(\widehat{G}, G)$ for different algorithms and find that only Topic-SCORE \citep{ke2024using} and LOVE \citep{bing2020fast} attain the optima rate $\sqrt{M/(\bar{N}n)}$, up to logarithmic factors.  
In comparison, LOVE requires that $M\ll \bar{N}$, while Topic-SCORE does not have this requirement.  
When $K$ grows with $n$, LOVE has the following rate of convergence, 
\beq \label{TM-UB}
{\cal L}(\widehat{G}, G)\leq C(\bar{N}n)^{-1/2}\sqrt{K^{3}M\log(\bar{N}n)}, 
\eeq
which is again under the assumption of $M\ll \bar{N}$. The rate of convergence of Topic-SCORE for growing $K$ has not been given in the literature. In Section~\ref{subsec:rate-TM}, we close this gap and show that Topic-SCORE also attains the rate in \eqref{TM-UB}; and this is true for both the case of $M\ll \bar{N}$ and $M=O(\bar{N})$. The rate in \eqref{TM-UB} is minimax optimal subject to a logarithm factor, according to the lower bound in \cite{bing2020fast}. 
Based on these observations, we recommend using Topic-SCORE.

\subsection{The Topic-SCORE algorithm}
Topic-SCORE \citep{ke2024using} belongs to the category of topic modeling algorithms that assume the anchor word condition. A word $1\leq m\leq M$ is called an anchor word of topic $k$ if and only if
\beq \label{anchor-word}
g_k(m)\neq 0, \quad \mbox{and}\quad g_\ell(m)=0 \mbox{ for all }\ell\neq k. 
\eeq
The anchor word condition requires that each topic has at least one anchor word.

We now briefly describe the key idea behind Topic-SCORE.  Let $L=\diag(N_1, N_2, \ldots, N_n)$ and $\widetilde{Y}= YL^{-1}$. 
Let $\hat{\xi}_1, \ldots, \hat{\xi}_K\in\mathbb{R}^M$ be the first $K$ left singular vectors of $\widetilde{Y}$ and define 
\[
\widehat{R}=[\diag(\hat{\xi}_1)]^{-1}[\hat{\xi}_2, \ldots, \hat{\xi}_K] \; \in\; \mathbb{R}^{M\times (K-1)}.
\]
\cite{ke2024using} show that there exists a low-dimensional $K$-vertex simplex associated with the rows of $\widehat{R}$. As long as each topic has at least one anchor word, this simplex can be estimated from $\widehat{R}$. Once this simplex is available, it can be used to obtain an explicit formula for converting the singular vectors to valid estimates of the topic vectors. 
%
%
The full algorithm is as follows: 
\begin{enumerate}
\item {\it (SVD)}. Let $D=\diag(\widetilde{Y}{\bf 1}_n)$. Obtain the first $K$ left singular vectors $D^{-1/2}\widetilde{Y}$, and denote them by $\hat{\xi}_1, \ldots, \hat{\xi}_K\in\mathbb{R}^M$ . Construct a matrix $\widehat{R}=[\mathrm{diag}(\hat{\xi}_1)]^{-1}[\hat{\xi}_2, \ldots, \hat{\xi}_K]$. 
\item {\it (Simplex vertex hunting)}. Apply a vertex hunting algorithm \citep[Section 3.4]{ke2023special} to the rows of $\widehat{R}$ 
to obtain $\hat{v}_1, \hat{v}_2, \ldots, \hat{v}_K$. 
Write $\widehat{V}=[\hat{v}_1, \hat{v}_2, \ldots, \hat{v}_K]$. 

\item {\it (Topic matrix estimation)}. For each $1\leq m\leq M$, obtain $\hat{u}_m$ by solving the equations: $
{\bf 1}_K'\hat{u}_m=1$, and $\widehat{V}\hat{u}_m = \hat{r}_m$. 
If $\hat{u}_m$  contains negative entries, set these entries to zero and renormalize the whole vector to have a unit $\ell^1$-norm. For each $1\leq k\leq K$, estimate $g_k$ by normalizing the $k$th column of $D^{1/2}\diag(\hat{\xi}_1)[\hat{u}_1, \hat{u}_2, \ldots, \hat{u}_M]'$ to have a unit $\ell^1$-norm. 
\end{enumerate}
In Section~\ref{subsec:rate-TM}, we provide a tight error bound for this algorithm, which extends the existing bounds from fixed $K$ \citep{ke2024entry} to growing $K$. This result is not only useful for the analysis of our estimator but also of independent interest.

\section{Main Results} \label{sec:main}
In this section, we present our main theoretical results. For technical reasons, we introduce a slightly modified version of the estimator in \eqref{our-estimator}:
\beq \label{our-estimator2}
\widehat{\boldsymbol g}^+(x) = \widehat{G}'\widehat{G}\bigl(\widehat{G}'T\widehat{G} + \epsilon I_K \bigr)^{-1}\widehat{G}'S(x) {\bf 1}_n, 
\eeq
where
\[
\epsilon=\epsilon_n=  \frac{Kn}{M^2}\cdot 1\biggl\{\lambda_{\min}(\widehat{G}'T\widehat{G}) <  \frac{Kn}{M^2\log^2(Nn)}\biggr\}.  
\]
In fact, $\lambda_{\min}(\widehat{G}'T\widehat{G})$ is often much larger than $\frac{Kn}{M^2\log^2(Nn)}$, so that $\widehat{\boldsymbol g}^+(x)$ and $\widehat{\boldsymbol g}(x)$ are identical except for a negligible probability. However, when we bound the expected error of $\widehat{\boldsymbol g}(x)$, we must study the behavior of our estimator on this small-probability event, which is technically challenging. 
To avoid this hurdle, we study the above modified version. 
 
In Sections~\ref{subsec:reg-cond}-\ref{subsec:LB}, we focus on $d=1$, for which the regularity conditions and main results are both easy to present and interpret. The extension to $d>1$ is deferred to Section~\ref{subsec:d>1}. 

\subsection{Regularity conditions} \label{subsec:reg-cond}

When $d=1$, each $g_k(x)$ is a continuous density on $\mathbb{R}$. 
We assume they are smooth. 
Some commonly used smooth classes include the H\"{o}lder class, the Sobolev class, and the Nikol’ski class. 
The H\"{o}lder class focuses on a form of local smoothness, whereas the other two provide a form of integrated smoothness, with the Nikol’ski class broader than the Sobolev class.
Since we aim to control the integrated estimation error, we use the Nikol'ski class, with a smoothness parameter $\beta$: 

\begin{assumption}[Nikol'ski  Smoothness]\label{assump1}
Fix $\beta > 0$ and $L_0>0$. Let $\lfloor\beta\rfloor$ denote the largest integer that is strictly smaller than $\beta$. For each $1\leq k\leq K$, the  $\lfloor\beta\rfloor$th order derivative of $g_k$ exists and satisfies:
  \[
  \left( \int \left| g_k^{(\lfloor\beta\rfloor)}(x + t) - g_k^{(\lfloor\beta\rfloor)}(x) \right|^2 dx \right)^{1/2} \le L_0 |t|^{\beta - \lfloor\beta\rfloor},\qquad\mbox{for all $t\in\mathbb{R}$}.
  \]
\end{assumption}
Under this assumption, we use an order-$\beta$ kernel. 
Besides requiring that $\int {\cal K}(u)du=1$ as in \eqref{KDE-fi}, we also require:
\beq \label{cond-kernel}
\int u^j {\cal K}(u)du=0, \; 1\leq j\leq \lfloor \beta\rfloor, \quad  \int |u|^\beta |{\cal K}(u)|du +  \int |\mcal{K}(u)|du + \sup_u|\mcal{K}(u)| <\infty. 
\eeq

We also need an assumption to guarantee that the model is identifiable:
\begin{assumption}\label{assump-anchor}
For $1\leq k\leq K$, let $S_k=\{x\in\mathbb{R}: g_k(x)>0, g_\ell(x)=0 \mbox{ for all $\ell\neq k$}\}$. Suppose the Lesbegue measure of each $S_k$ is lower bounded by a constant $c_3>0$.  
\end{assumption}

Each $S_k$ is called the anchor region of $g_k(x)$. It is analogous to the anchor word definition (see \eqref{anchor-word}) in a classical topic model. 
Under this assumption, $g_1(x), \ldots, g_K(x)$ are identifiable \citep{austern2025poisson}. 
While other identifiability conditions may be used, a key advantage of adopting Assumption~\ref{assump-anchor} is that it naturally implies the anchor-word condition in the induced topic model. This allows us to apply existing anchor-word-based topic modeling algorithms, such as Topic-SCORE.

Note that unlike the classical nonparametric density estimation setting, we have $K$ different densities that are mixed together. We need additional conditions to ensure that ``unmixing'' them is possible. 
Intuitively, we require that (a) the $K$ densities are not too ``similar'', and (b) the ``effective sample size'' for estimating each individual $g_k(x)$ is properly large.  
We formally state these conditions as follows:

\begin{assumption} \label{assump2}
Let $\bar{g}(x)=\frac{1}{K}\sum_{k=1}^K g_k(x)dx$ and $\Sigma_{\boldsymbol g} = \frac{1}{K}\int {\boldsymbol g}(x) {\boldsymbol g}(x)' dx$. There exists a constant $c_1\in (0,1)$ such that 
\[
c_1\leq \lambda_{\min}(\Sigma_{\boldsymbol g}) \leq \lambda_{\max}(\Sigma_{\boldsymbol g})\leq c_1^{-1}, \qquad\mbox{and}\qquad \max_x \bar{g}(x) \leq c_1^{-1}. 
\]
Let $\Sigma_\Pi = (K/n)\Pi'\Pi$ and $\eta = \Pi'{\bf 1}_n$. Denote by $\eta_{\max}$ and $\eta_{\min}$ the maximum and minimum entries of $\eta$. There exists a constant $c_2\in (0,1)$ such that 
\[
c_2 <\lambda_{\min}(\Sigma_\Pi)\leq \lambda_{\max}(\Sigma_\Pi)\leq c_2^{-1}, \qquad\mbox{and}\qquad c_2K^{-1}n\leq \eta_{\min}\leq \eta_{\max}\leq c^{-1}_2K^{-1}n. 
\]
\end{assumption}

The $K\times K$ matrix $\Sigma_{\boldsymbol g}$ captures the `similarity' of $K$ densities. 
If $\Sigma_{\boldsymbol g}$ is ill-conditioned, some of the $g_k(x)$ are too similar to each other, making the unmixing task too difficult. 
The $K\times K$ matrix $\Sigma_\Pi$ and the vector ${\bf 1}_n'\Pi$ both describe the `sample size balance' on $K$ densities. When $\Sigma_\Pi$ is ill-conditioned or when ${\bf 1}_n'\Pi$ has severely unbalanced entries, there exists at least one $g_k(x)$ so that the $n$ samples' total weights on this density are significantly smaller than $n/K$; in such a case, this density cannot be estimated accurately enough. With the above being said, Assumption~\ref{assump2} is almost necessary for the success of the unmixing task.

\subsection{A perturbation bound for our estimator} \label{subsec:main-bound}
Similarly as in Section~\ref{subsec:TM-review}, let ${\cal L}(\widehat{G}, G)=\sum_{1\leq m\leq M}\sum_{1\leq k\leq K}|\widehat{G}_{mk}-G_{mk}|$ be the $\ell^1$-estimation error in the topic modeling step, whose order will be $\delta_n$ below. 
The following theorem is our main technical result and will be proved in Section~\ref{sec:UBproof}. \footnote{The estimation error on $g_1(x), \ldots, g_K(x)$ is measured up to a permutation of the $K$ estimated densities. For notational convenience, throughout this section, we omit this permutation in our theorem statements.}

\begin{thm} \label{thm:main}
Fix $d=1$. Consider the model \eqref{model-1}-\eqref{model-2}, where $K^2\log(Nn)=O(Nn)$ and Assumption~\ref{assump1} holds. Let $\Sigma_G = (M/K)G'G$ and $\Sigma_\Pi = (K/n)\Pi'\Pi$. Let $\delta_n$ be a sequence such that $M\delta_n^2=o(K)$. Suppose:
 \begin{itemize}
 \item[(i)] $\lambda_{\min}(\Sigma_G)\asymp \lambda_{\max}(\Sigma_G)\asymp 1$, and $\|G{\bf 1}_K\|_\infty = O(M^{-1}K)$. \footnote{This condition on $\Sigma_G$ is implied by the condition on $\Sigma_{\boldsymbol g}$ in Assumption~\ref{assump2}, as long as $M\to\infty$ and the bins are properly chosen (see Lemma~\ref{lem:key} in the supplement). On the other hand, even when the condition on $\Sigma_{\boldsymbol g}$ is not satisfied, the condition on $\Sigma_G$ may still hold. To maintain the broadness of this theorem, we choose to present the condition on $\Sigma_G$ directly.}
 \item[(ii)] $\lambda_{\min}(\Sigma_\Pi)\asymp \lambda_{\max}(\Sigma_\Pi)\asymp 1$, and $\|{\bf 1}_n'\Pi\|_\infty =O(K^{-1}n)$. 
 \item[(iii)] $N_i\asymp N$ for all $1\leq i\leq n$. 
 \item[(iv)] $h \to 0$ and $Nnh \to \infty$.
 \item[(v)] There is a constant $C>0$ such that  ${\cal L}(\widehat{G}, G)\leq C\delta_n$ with probability $1-o((Nn)^{-5})$.
 \end{itemize}
Let $\widehat{\boldsymbol g}^+(x)$ be the estimator in \eqref{our-estimator2}. Suppose $K\leq M\leq [Nn/\log^2(Nn)]^{1/2}$ and $ (Nn)^{-1}K\ll h\ll \log^{-1}(Nn)$. 
    Then, there exists a constant $C_0>0$ such that 
\beq \label{UB-main}
\mbb{E}\left[\int \|\widehat{\bm{g}}^+(x)-\bm{g}(x)\|^2dx\right]\le C_0 \cdot K\left( h^{2\beta}  + \frac{K}{Nnh}+\frac{M}{K}\delta_n^2  +  \frac{K^2}{Nn} \right). 
\eeq
\end{thm}

In \eqref{UB-main}, the first two terms are analogous to the `bias' and `variance' in classical nonparametric density estimation, which can be optimized by choosing a proper bandwidth $h$:
\begin{cor} \label{cor:rate1}
Under the conditions of Theorem~\ref{thm:main}, if $h\asymp [K/(Nn)]^{\frac{1}{2\beta+1}}$, then
\[
\mbb{E}\left[\int \|\widehat{\bm{g}}^+(x)-\bm{g}(x)\|^2dx\right]\le C_0 \cdot K \left[\Bigl(\frac{K}{Nn}\Bigr)^{\frac{2\beta}{2\beta+1}}+\frac{M}{K}\delta_n^2  +  \frac{K^2}{Nn} \right]. 
\]
\end{cor}

The last two terms in \eqref{UB-main} still persist here. They are the extra price paid for `unmixing'  densities. Specifically, the third term comes from the error of estimating the topic matrix $G$, and the last term is the additional error of using $\widehat{G}$ to estimate $\Pi$. 



\subsection{The error for estimating $G$} \label{subsec:rate-TM}
We now study the topic modeling error and derive the order of $\delta_n$. 
As explained in Section~\ref{sec:TM}, we apply the Topic-SCORE algorithm. The following theorem is proved in the supplement: 


\begin{thm}\label{thm:TM_errorbound}
    Consider the topic model in \eqref{topic-model}. Define $\Sigma_\Pi = (K/n)\Pi'\Pi$, and $\Sigma_G = (M/K)G'G$. 
    We assume:
    \begin{itemize}
        \item[(i')]  $\lambda_{\min}(\Sigma_G)\asymp \lambda_{\max}(\Sigma_G)\asymp 1$, and the smallest entry of $G{\bf 1}_K$ is of the order $M^{-1}K$. 
        \item[(ii')] $\lambda_{\min}(\Sigma_\Pi)\asymp \lambda_{\max}(\Sigma_\Pi)\asymp 1$, and all of the $K$ entries of $\Pi'{\bf 1}_n$ are of the order $K^{-1}n$. 
    \end{itemize}
    Additionally, suppose all $N_i$'s are of the order $N$, and each topic has at least one anchor word (see \eqref{anchor-word}). 
    Let $\widehat{G}=[\hat{g}_1, \ldots, \hat{g}_K]$ be the output of Topic-SCORE. 
    Suppose $\min\{M,N\}\ge \log^3(Nn)$, $\log(N)=O(\log(n))$, $K\le M$, and $K^3M\log^2(Nn)\leq  Nn$. 
There exists a constant $C_{\text{TM}}>0$ such that with probability $1- o((Nn)^{-5})$, 
\[
\Vert e_m'(\widehat{G} - G)\Vert_1 \leq \|e_m'G\|_1\cdot C_{\text{TM}} \sqrt{\frac{KM\log (Nn)}{Nn}};
\]
Furthermore, with probability $1-o((Nn)^{-5})$, the $\ell^1$-error satisfies that 
\begin{align}\label{tscore-op-norm}
\mathcal L (\widehat G, G)  
\leq C  \sqrt{\frac{K^3 M\log (Nn)}{Nn}}.
\end{align}
\end{thm}


Theorem~\ref{thm:TM_errorbound} extends the analysis of Topic-SCORE from finite $K$  \citep{ke2024using, ke2024entry} to growing $K$. A lower bound for the rate of convergence is presented in \cite{bing2020fast}, which is  $\sqrt{K^3M/(nN)}$. 
Therefore, our result shows that Topic-SCORE is rate-optimal (up to a logarithmic factor) in the growing-$K$ setting. This is also the first method that achieves the optimal rate in both regimes of $M=O(N)$ and $M\gg N$ (in contrast, the method in \cite{bing2020fast} only applies to the regime of $M=O(N)$).



\subsection{The rate of convergence of our estimator} \label{subsec:UB}
We insert the topic modeling error bound from Section~\ref{subsec:rate-TM} into the error bound from Section~\ref{subsec:main-bound}. It follows that 
\[
\mbb{E}\left[\int \|\widehat{\bm{g}}^+(x)-\bm{g}(x)\|^2dx\right]\le C_0 \cdot K \left[\Bigl(\frac{K}{Nn}\Bigr)^{\frac{2\beta}{2\beta+1}}+\frac{K^2M^2\log(Nn)}{Nn}  +  \frac{K^2}{Nn} \right]. 
\]
This bound holds under the conditions of Theorem~\ref{thm:main} and Theorem~\ref{thm:TM_errorbound}. 
These conditions are implied by Assumptions~\ref{assump1}-\ref{assump2}, 
if $M\to\infty$ and the bins are properly chosen (see Section~\ref{subsec:tuning} and Lemma~\ref{lem:key} in the supplement). 
Moreover, when $K$ satisfies some conditions, the second and third terms are dominated by the first term. Combining these observations gives the following theorem: 

\begin{thm}\label{thm:upperbound}
Consider the model \eqref{model-1}-\eqref{model-2} with $d=1$, where Assumptions \ref{assump1}-\ref{assump2} hold, and $N_i\asymp N$ for all $i$. Suppose $
K=O\bigl(\max\bigl\{(Nn)^{\frac{1}{4}}[\log(Nn)]^{-1},(Nn)^{\frac{1}{4\beta+3}}\log^{-\frac{4\beta+2}{4\beta+3}}(Nn)\bigr\}\bigr)$, $N\geq \log(n)$, and $\log(N)=O(\log(n))$. Let $\widehat{\boldsymbol g}^+(x)$ be the estimator in \eqref{our-estimator2}, where the kernel ${\cal K}(\cdot)$ satisfies \eqref{cond-kernel}, 
the plug-in $\widehat{G}$ is obtained from Topic-SCORE, 
and the bin size $M$ and the bandwidth $h$ satisfy that $M\asymp K\log(Nn)$ and $h\asymp [K/(Nn)]^{\frac{1}{2\beta+1}}$. Then, 
\[
\mbb{E}\left[\int \|\widehat{\bm{g}}(x)^+-\bm{g}(x)\|^2dx\right]\le C \cdot K\left( \frac{K}{Nn} \right)^{\frac{2\beta}{2\beta+1}}. 
\]
\end{thm}


It is interesting to compare this rate with those of existing methods. The method in \cite{katz2019decontamination} is only shown to be consistent under a relatively weak guarantee, without an explicit error rate. When $K$ is fixed, the method in \cite{austern2025poisson} has an error rate of $(Nn)^{-\theta^*(\beta)}$, where \footnote{Precisely, the paper considers a H\"{o}lder class with smoothness parameter $\beta$ (a more specialized class when the support is bounded), but the rate is transferable.}
\[
\theta^*(\beta) = \begin{cases}
2\beta/(2\beta+1), & \mbox{when }\beta\leq 1,\cr
2/3, & \mbox{when }1<\beta\leq 2, \cr
2\beta/(5\beta+2), & \mbox{when }\beta>2. 
\end{cases}
\]
Their rate is strictly slower than ours for $
\beta>1$. 
Moreover, they don't provide any result for the growing-$K$ setting. 


\subsection{A matching lower bound} \label{subsec:LB}
To show the optimality of the error rate established above, we also provide a matching lower bound:
\begin{thm}\label{thm:lowerbound}
Let $\Theta_{\beta}$ be the collection of all $({\boldsymbol g}(x), \Pi)$ such that Assumptions~\ref{assump1}-\ref{assump2} are satisfied. 
    Under the model in \eqref{model-1}-\eqref{model-2}, there exists a constant $C_0'>0$ such that
    \[
        \inf_{\widehat{\bm{g}}(x)}\sup_{(\bm{g}(x), \Pi)\in\Theta_\beta}\mbb{E}\left[\int \|\widehat{\bm{g}}(x)- \bm{g}(x)\|^2dx\right]\ge C_0'K\left(\frac{K}{N n}\right)^{\frac{2\beta}{2\beta+1}}.
\]
\end{thm}

Comparing Theorem~\ref{thm:lowerbound} with Theorem~\ref{thm:upperbound}, we obtain the optimality of our estimator. 

The proof of Theorem~\ref{thm:lowerbound} is more subtle than that of the lower bound for standard nonparametric density estimation, because we need the rate to depend on $K$. We address this by using a proper least-favorable configuration. See Section~\ref{sec:LBproof} for details.  


\subsection{Extension to a general $d$} \label{subsec:d>1}
Our results can be extended to any fixed dimension $d$. In this case, each $g_k(x)$ is a multivariate density. 
We assume that $g_1(x), \ldots, g_K(x)$ belong to the anisotropic Nikol'ski class \citep{lepski2013multivariate} defined as follows.
{\renewcommand{\theassumption}{1$'$}
\setcounter{assumption}{0}
\begin{assumption}[Multivariate Nikol’ski Smoothness]\label{assump1prime}
Fix $d \ge 1$, $\bm{\beta} = (\beta_1,\dots,\beta_d) \in (0,\infty)^d$ 
and $L > 0$. For each $j \in \{1,\dots,d\}$, let 
$r_j = \lfloor \beta_j \rfloor$ denote the largest integer strictly smaller than $\beta_j$, 
and write $\bm{r} = (r_1,\dots,r_d)$. For each $1 \le k \le K$, the mixed partial derivative 
$\partial^{\bm{r}} g_k 
= \partial_1^{\,r_1} \cdots \partial_d^{\,r_d} g_k$ 
exists and satisfies
\[
\left(
\int_{\mathbb{R}^d}
\left|
\partial^{\bm{r}} g_k(x+t)
-
\partial^{\bm{r}} g_k(x)
\right|^2dx 
\right)^{1/2}
\le
L
\sum_{j=1}^d |t_j|^{\beta_j - r_j},
\qquad
\text{for all } t \in \mathbb{R}^d.
\]
\end{assumption}}
Our estimator requires having a multivariate kernel ${\cal K}(x)$ to obtain $K_{ij}(x) = {\cal K}_h(x-X_{ij})$ (see \eqref{def-S-K}), where ${\cal K}_h(x)=\frac{1}{h^d}{\cal K}(\frac{x}{h})$ for a bandwidth $h>0$.  
However, under Assumption~\ref{assump1prime}, we should use different bandwidths for different coordinates.  Let $\bm{h}=(h_1, h_2, \ldots, h_d)'$ be a bandwidth matrix. 
We define 
\beq \label{cond-kernel-d}
{\cal K}_{\bm{h}}(x) = \prod_{j=1}^d \frac{1}{h_j}{\cal K}_j\bigl(\frac{x_j}{h_j}\bigr), \quad\mbox{where ${\cal K}_j$ satisfies \eqref{cond-kernel} for $\beta=\beta_j$.}
\eeq
In the supplementary material, we present a counterpart of Theorem~\ref{thm:main} when Assumption~\ref{assump1} is replaced by Assumption~\ref{assump1prime}. The error bound has the following form: 
\[
\mbb{E}\left[\int_{\mbb{R}^d} \|\widehat{\bm{g}}^+(\bm{x})-\bm{g}(\bm{x})\|^2d\bm{x}\right]\le C_0 \cdot K\left( \sum_{j=1}^d h_j^{2\beta_j}  + \frac{K}{Nn\prod_{j=1}^d h_j}+\frac{M}{K}\delta_n^2  +  \frac{K^2}{Nn} \right). 
\]
Compared to \eqref{UB-main}, only the first two terms are different. The last two terms, which correspond to the extra price paid for ``unmixing''  densities, are the same as before. 
We recall that these two terms arise from the error of estimating $G$ and the additional error of using $\widehat{G}$ to estimate $\Pi$. These errors are only related to the induced topic model, which are independent of the dimension $d$ (especially, the bin size $M$ does not depend on $d$).


Similarly as in the case of $d=1$, we choose $\bm{h}$ to optimize the first two terms; in addition, when $K$ grows with $Nn$ at a speed not too fast, the last two terms will be negligible. We thus have the following theorem:  
\begin{thm}\label{thm:upperboundind}
Fix $d\geq 1$ and consider the model \eqref{model-1}-\eqref{model-2}, where Assumptions~\ref{assump1prime} and Assumptions~\ref{assump-anchor}-\ref{assump2} hold, and $N_i\asymp N$ for all $1\leq i\leq n$. 
Define the harmonic mean $(\beta^*)^{-1}=d^{-1}\sum_{j=1}^d\beta_j^{-1}$. 
Suppose $
K=O\bigl(\max\bigl\{(Nn)^{\frac{1}{4}}[\log(Nn)]^{-1},(Nn)^{\frac{d}{4\beta^*+3d}}\log^{-\frac{4\beta^*+2d}{4\beta^*+3d}}(Nn)\bigr\}\bigr)$, $N\geq \log(n)$, and $\log(N)=O(\log(n))$. 
Let $\widehat{\boldsymbol g}^+(\bm{x})$ be the estimator in \eqref{our-estimator2}, where $\mcal{K}_{\bm{h}}(x)$ has the product form  in \eqref{cond-kernel-d}, the plug-in $\widehat{G}$ is obtained by Topic-SCORE, 
and the bin size $M$ and the bandwidth vector $\bm{h}$ satisfy that $M\asymp K\log(Nn)$ and $h_j\asymp [K/(Nn)]^{\frac{1}{2\beta_j+d(\beta^*)^{-1}\beta_j}}$ for all $1\le j\le d$. Then, 
\[
\mbb{E}\left[\int_{\mbb{R}^d} \|\widehat{\bm{g}}(\bm{x})^+-\bm{g}(\bm{x})\|^2d\bm{x}\right]\le C \cdot K\left( \frac{K}{Nn} \right)^{\frac{2\beta^*}{2\beta^*+d}}. 
\]
\end{thm}


A special case is that $g_1(x), \ldots, g_K(x)$ belong to the isotropic Nikol'ski class, corresponding to $\beta_j=\beta$ for all $1\leq j\leq d$. The rate becomes $K[K/(Nn)]^{\frac{2\beta}{2\beta+d}}$. 
A matching lower bound can also be proved. It is a straightforward extension of Theorem~\ref{thm:lowerbound} (see Section~\ref{sec:LBproof}). Combining these observations, we conclude the minimax optimal of our estimator under isotropic Nikol'ski smoothness.

\section{Simulations} \label{sec:Simu}
We investigate the performance of our estimator on synthetic data. 
For a given $K$, we construct $g_1(x), \ldots, g_K(x)$ as follows. 
For any $a, \beta>0$, let $\Phi_{\beta}(x; a)$ be a bump function  
defined as follows: 
\beq \label{bumps}
\Phi_{\beta}(x;a)
=
\bigl( 1 - |x/a|^{\beta}\bigr) \, 1\{ |x| < a\}. 
\eeq
It can be verified that $\Phi_{\beta}(x;a)$ is in the  Nikol’ski  class with a smoothness parameter $\beta$. 
For $k=1,\dots,K$, let $a_k = 11 + 2k$,
and define the unnormalized density as 
\beq \label{g-in-simu}
\tilde g_k(x)
=3\,\Phi_\beta(x;12)+
2\,\Phi_\beta(x-a_k;1)+
2\,\Phi_\beta(x+a_k;1).
\eeq
Let $
g_k(x) = \left(\int_{\mathbb{R}} \tilde g_k(u)du\right)^{-1}\tilde g_k(x)$. By construction, all densities are in the  Nikol’ski  class with smoothness of $\beta$, and they share a common central support $[-12,12]$. Additionally, each $g_k(x)$ has an anchor region as defined in Assumption \ref{assump-anchor}. The membership vectors $\{\pi_i\}_{i=1}^n$ are generated from a multivariate logistic-normal model: 
\[
\pi_i = \mathrm{softmax}(Z_i / \tau), \qquad Z_i \sim \mathcal{N}(0, I_K),
\] 
where the temperature parameter is $\tau=0.5$.  
Finally, we generate $X_{ij}$ from the model \eqref{model-1}-\eqref{model-2}, letting $N_i=N$ for all $1\leq i\leq n$. 
We measure the performance of estimators by the mean
integrated squared error (MISE), defined by $\int\|\widehat{\bm{g}}(x)-\bm{g}(x)\|^2dx$. Here, the integral is numerically approximated by a fine evaluation grid with 400 points.  

\begin{figure}[tbp]
\centering

\includegraphics[width=0.25\textwidth]{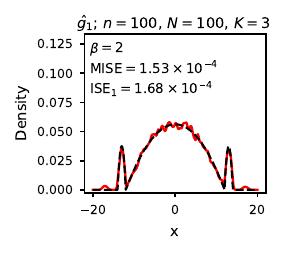}\hfill
\includegraphics[width=0.25\textwidth]{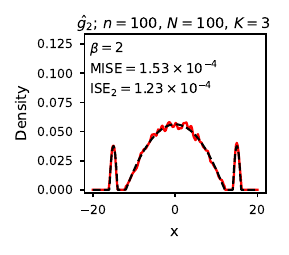}\hfill
\includegraphics[width=0.25\textwidth]{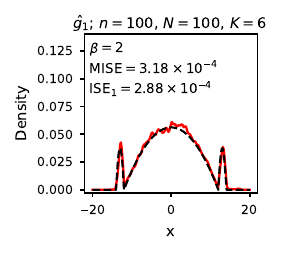}\hfill
\includegraphics[width=0.25\textwidth]{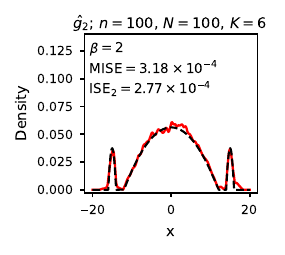}

\vspace{0.3em}

\includegraphics[width=0.25\textwidth]{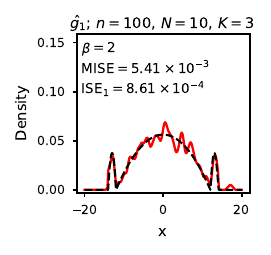}\hfill
\includegraphics[width=0.25\textwidth]{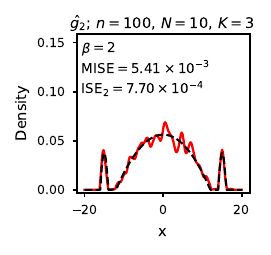}\hfill
\includegraphics[width=0.25\textwidth]{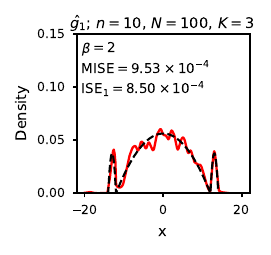}\hfill
\includegraphics[width=0.25\textwidth]{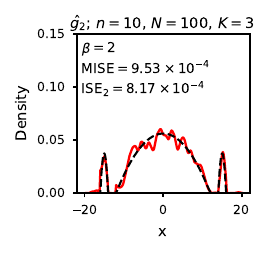}
\caption{MISE and estimated densities for different $n,N$ and $K$ (Experiment~1).  
Black:  true density; red: estimated density from one realization.  Only the estimates for the first two components are presented; other looks similar.}
\label{fig:MISE}
\end{figure}

{\bf Experiment 1}. 
In this experiment, we study the performance of our method under different values of $(n, N, K)$.    
In the baseline setting, $(n, N, K)=(100, 100, 3)$. We then consider three other settings by (i) fixing $(n, N)$ as in the baseline and changing $K$ to $6$, (ii) fixing $(n, K)$ as in the baseline and changing $N$ to $10$, and (iii) fixing $(N, K)$ as in the baseline and changing $n$ to $10$. In each setting, the true densities $g_1(x), \ldots, g_K(x)$ are constructed as in \eqref{g-in-simu}, where the bump function is $\Phi_{\beta}(x; a)$ in \eqref{bumps} with $\beta = 2$. 
In our method, we use a Gaussian kernel. We set $M=2\lfloor K\log(Nn)\rfloor$ as in \eqref{chooseM} and select the bins as in \eqref{bins-d=1}. The bandwidth is selected by minimizing the the data-driven criteria, $\widetilde{\text{AMISE}}(h; \widehat{\bm{g}})$ (see \eqref{CV} and Section~\ref{subsec:tuning}). 
Recall that the optimal bandwidth is $h\asymp [K/(Nn)]^{1/(2\beta+1)}$. 
We thus re-parametrize $h=h_t = [K/(Nn)]^t$ and minimize over $t$ on a grid of 20 equally spaced points on $(0,1)$.  
This grid choice does not use any knowledge of $\beta$ and covers a wide range of $h$.

In Figure~\ref{fig:MISE}, for each of the four settings, we plot $\hat{g}_k(x)$ (red solid line) in one realization and the true $g_k(x)$ (black dashed line); to save space, we only show these plots for the first two densities, which results in a total of 8 plots. We also report  the integrated squared error (ISE$_k$) for individual $\hat{g}_k(x)$ and the MISE (equal to $\frac{1}{K}\sum_{k=1}^K$ISE$_k$), averaged on 100 repetitions. 
%
The results suggest that the performance decreases as $K$ increases and as $N$ or $n$ decreases. This is consistent with our theory.

{\bf Experiment 2}. 
In this experiment, we study the sensitivity of our estimator to the tuning parameters $h$ and $M$. We consider the baseline setting in Experiment~1, where $(n, N, K)=(100, 100, 3)$ and $\beta = 2$.
When examining the effect of the bandwidth \(h\), the number of bins \(M\) is fixed at $M=2\lfloor K\log(Nn)\rfloor$ as in Experiment 1. We then let the bandwidth $h$ vary over a grid of the form $h_\ell = [K/(Nn)]^{\alpha_\ell}, \ 1 \le \ell \le L$, where $\{\alpha_\ell\}_{\ell=1}^L$ are equally spaced points in $(0,1)$ with $L = 20$. Conversely, when studying the effect of \(M\), the bandwidth \(h\) is fixed at its optimal value $h=0.31$ (obtained from Experiment 1) and we also let $M$ vary over a logarithmically spaced grid centered at $M_0$: $M = M_0 \cdot 10^{t},\ t \in [-0.2,\,0.8]$, where 10 points are uniformly spaced in $t$. 
This two-step procedure isolates the influence of each tuning parameter while holding the other at the optimal level.  The results are shown in Figure~\ref{fig:handM}. 
As $h$ increases, the performance of our estimator initially improves and then deteriorates, leading to a sweet spot. This is consistent with our theory, as $h$ controls the bias-variance trade-off. 
The performance is insensitive to $M$ in a wide range. This is also as anticipated, because by the design of our estimator, $M$ only affects the topic modeling error, which is not the  leading error term.  

\begin{figure}[tbp]
\centering
\includegraphics[width=0.5\textwidth]{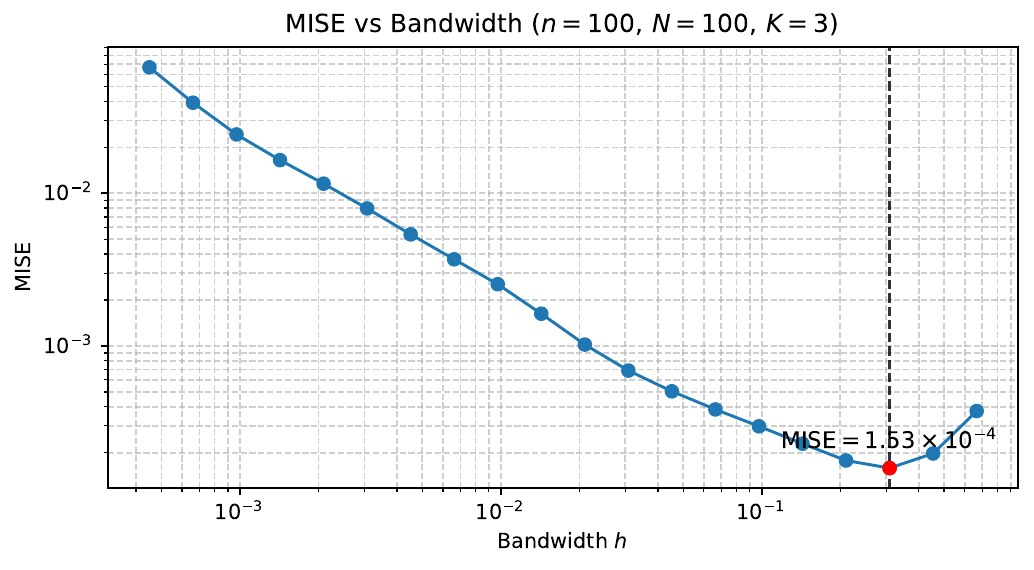}\hfill
\includegraphics[width=0.5\textwidth]{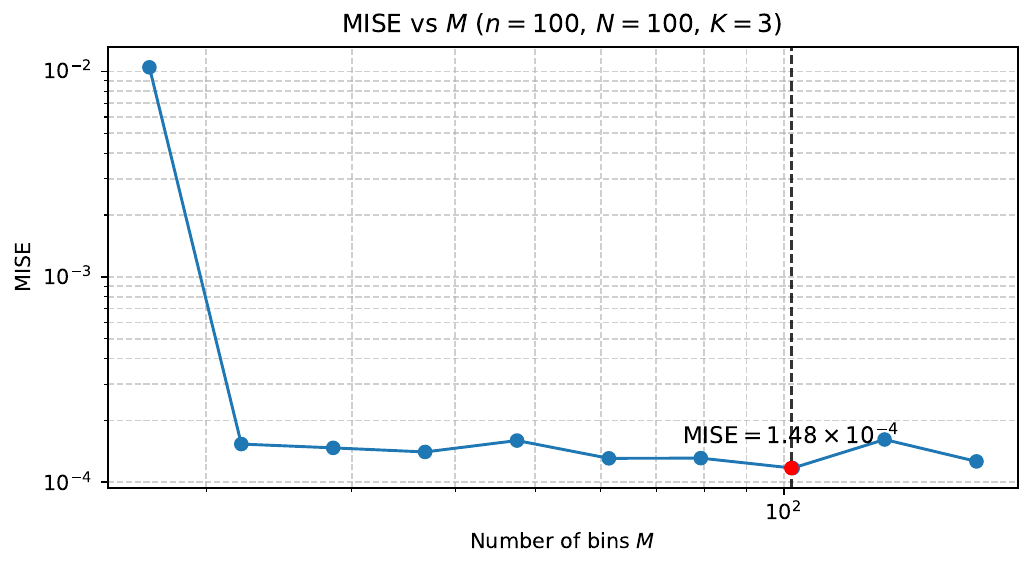}
\caption{Effect of the bandwidth \(h\) and the number of bins \(M\) on MISE for $n=100$, $N=100$ and $K=3$ (Experiment~2).}
\label{fig:handM}
\end{figure}

{\bf Experiment 3}. 
In this experiment, we compare our estimator with the one in \cite{austern2025poisson}. 
Let $(n, N, K)=(100, 100, 3)$. We construct $g_1(x), \ldots, g_K(x)$ as in \eqref{bumps}-\eqref{g-in-simu}. 
We take two values of $\beta$: $\beta=2$ (high smoothness) and $\beta=0.6$ (moderate smoothness). 

Both methods have two tuning parameters, the bandwidth $h$ and the bin size $M$. 
For a fair comparison, we use the ideal tuning parameters for both methods, which are determined by minimizing the true MISE. In detail, for the estimator in \cite{austern2025poisson}, 
the optimal $M$ is of the order $M_0' = \lfloor (Nn/(\log n))^{\frac{1}{2\theta_1(\beta)+1}}\rfloor$, with $\theta_1(\beta)=(\beta\wedge 1)\cdot1(\beta<2)+\frac{2\beta \cdot 1(\beta\ge 2)}{2+\beta}$.  
We consider a grid,  $M = M_0' \cdot 10^{t}$, for 10 points of $t$ that are uniformly spaced in $[-0.2,\,0.8]$. 
Similarly, let $h_0' = (Nn)^{-1/(2\beta+1)}$ and consider a grid, $h = h_0' \cdot 10^{t}$, for 20 points of $t$ that are uniformly spaced in $[-0.7,\,0.7]$. Combining them gives $10\times 20=200$ pairs of $(M, h)$. For all pairs, $M$ and $h$ are at their optimal orders, and the differences across pairs lie in the constants. 
We select the $(M^*, h^*)$ that minimizes the true MISE over 100 repetitions. 
For our estimator, the optimal $M$ is of the order $M_0=\lfloor K\log(Nn)\rfloor$ and the optimal $h$ is of the order $h_0= (K/(Nn))^{1/(2\beta+1)}$. Since the performance is insensitive to $M$, we fix $M=2M_0$. Meanwhile, we consider a grid, $h=h_0 \cdot 10^{t}$, for 20 equally spaced $t$ in $[-0.7, \, 0.7]$. We select the $h^{*}$ that minimizes the true MISE over 100 repetitions.  
By using such ideal tuning parameters, we can compare the best performances of these methods, without worrying about the effect of data-driven tuning parameter selection.

In Figure \ref{fig:comparsion}, the first and second plots correspond to the case of $\beta = 2$, 
and the third and the fourth plots correspond to $\beta = 0.6$. 
The estimated densities by our method (from one realization) are plotted in red, and the ones by the method in \cite{austern2025poisson} are plotted in blue. 
As discussed right below Theorem \ref{thm:upperbound}, two estimators have similar performances when $0<\beta<1$, while our method has a strictly better performance when $\beta>1$. 
The numerical results here support our theoretical arguments. 
\begin{figure}[htbp]
\centering

\includegraphics[width=0.24\textwidth]{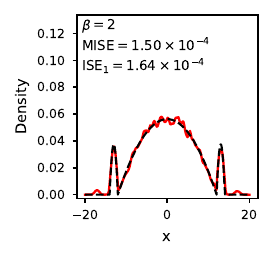}\hfill
\includegraphics[width=0.24\textwidth]{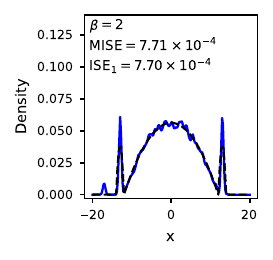}\hfill
\includegraphics[width=0.24\textwidth]{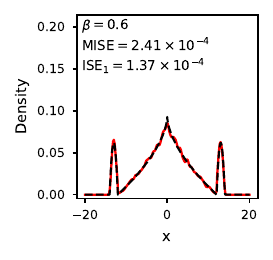}\hfill
\includegraphics[width=0.24\textwidth]{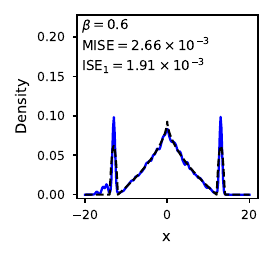}
\caption{Comparison of our estimator, which corresponds to the first and third plots,  with the one in \cite{austern2025poisson}, which corresponds to the second and fourth plots (Experiment~3).}
\label{fig:comparsion}
\end{figure}

\section{Proof of the Minimax Upper Bound}\label{sec:UBproof}



In this section, we prove Theorem~\ref{thm:main}. Specifically, Section~\ref{subsec:UBproof-preliminary} contains preliminary lemmas, 
Section~\ref{subsec:UBproof-decompose} presents a decomposition of the error into three terms, and Sections~\ref{subsec:UBproof-bias}-\ref{subsec:stochastic-error-2} analyze each of the three terms separately. 
Proofs of secondary lemmas are relegated to the supplementary material.

\subsection{Preliminaries}\label{subsec:UBproof-preliminary}
Our analysis frequently uses properties of a certain type of incomplete $U$-statistics. 
Suppose we have random variables:
\beq \label{x-rv}
\mbox{$ \{x_{ij}\}_{1\le i\le n,1\le j\le N_i}$ are independent across $i$; for any fixed $i$, they are i.i.d. across $j$}. 
\eeq

We first consider an incomplete U-process defined by these random variables and an arbitrary trivariate function $h(x; y,z)$:
\beq \label{incomplete-U-process}
    H_{n,N}(x)=\frac{1}{nN(N-1)}\sum_{i=1}^n\sum_{1\le j\neq j'\le N}h(x;x_{ij},x_{ij'}).
\eeq
\begin{lemma}[Integrated variance bound for incomplete $U$-processes]\label{lem:uprocess-bound}
Consider the incomplete $U$-process in \eqref{incomplete-U-process}. 
Suppose there exists $b_i>0$ such that   $\mbb{E}[\int h^2(x;x_{i1},x_{i2})dx]\le b_i^2$, for each $1\leq i\leq n$.  Let $\theta_i(x)=\mbb{E}[h(x;x_{i1},x_{i2})]$ and $\theta(x)=n^{-1}\sum_{i=1}^{n}\theta_i(x)$. Then,  
    \[
        \mbb{E}\biggl[\int (H_{n,N}(x)-\theta(x))^2dx\biggr]\le \frac{4(\sigma^2+b^2)}{nN}, 
    \]
where $\sigma^2=n^{-1}\sum_{i=1}^{n}\mbb{E}\left[\int (h(x;x_{i1},x_{i2})-\theta_i(x))^2dx\right]$ and $b^2=n^{-1}\sum_{i=1}^n b_i^2$. 
\end{lemma}

We then consider an incomplete U-statistics defined by $x_{ij}$'s and an arbitrary bi-viarate function $h^*(x_1, x_2)$:
\beq \label{eq:u-stat}
    H^*_{n,N}=\frac{1}{nN(N-1)}\sum_{i=1}^n\sum_{1\le j\neq j'\le N}h^*(x_{ij},x_{ij'}). 
\eeq
\begin{lemma}[Bernstein inequality for incomplete $U$-statistics]\label{lem:ustats-bound}
Consider the incomplete $U$-statistic in \eqref{eq:u-stat}. Suppose there exists $b>0$ such that $\max_{1\leq i\leq n, 1\leq j\neq j'\leq N}|h^*(x_{ij},x_{ij'})|\le b$. 
Let  $\theta=n^{-1}\sum_{i=1}^{n} \mbb{E}h^*(x_{i1},x_{i2})$ and $\sigma^2=n^{-1}\sum_{i=1}^{n}\Var(h^*(x_{i1},x_{i2}))$. For all $\delta\in (0,1)$, with probability $1-\delta$,
    \[
        |H^*_{n,N}-\theta|\le 2\sigma\sqrt{\frac{2\log(2/\delta)}{nN}}+\frac{16b\log(2/\delta)}{3nN}. 
    \] 
\end{lemma}

Both lemmas are proved in the supplementary material. The proofs use some known techniques for analyzing U-statistics, specially tailored to our setting where the random variables satisfy \eqref{x-rv}. For instance, the proof of Lemma~\ref{lem:ustats-bound} relies on decoupling the randomness in the tuple $(x_{ij},x_{ij'},x_{i\ell},x_{i\ell'})$ for $j\neq j'$ and $\ell\neq \ell'$.

\subsection{A decomposition of the error} \label{subsec:UBproof-decompose}
For any $\bar{\epsilon}>0$ and matrices $\bar{G}\in\mathbb{R}^{M\times K}$, $\bar{S}\in\mathbb{R}^{n\times M}$, and $\bar{T}\in\mathbb{R}^{M\times M}$, we define a mapping:
\beq \label{g(GST)}
{\boldsymbol g}(\bar{G}, \bar{S}, \bar{T}, \bar{\epsilon}) := \bar{G}'\bar{G} \bigl(\bar{G}'\bar{T}\bar{G} + \bar{\epsilon}I_K\bigr)^{-1}\bar{G}'\bar{S}{\bf 1}_n. 
\eeq
Our estimator in \eqref{our-estimator2} is defined as
\[
\widehat{\boldsymbol g}^+(x)  = {\boldsymbol g}\bigl(\widehat{G}, S(x), T, \epsilon \bigr).
\] 
To facilitate the analysis, we further introduce two intermediate quantities: 
\beq \label{all-estimators}
{\boldsymbol g}^*(x) = {\boldsymbol g}\bigl(G,\, \mathbb{E}[S(x)],\, \mathbb{E}[T], 0\bigr),\qquad \widehat{\boldsymbol g}^*(x) = {\boldsymbol g}\bigl(G, S(x), T, \epsilon \bigr). 
\eeq
The first estimator ${\boldsymbol g}^*(x)$ is the population counterpart of our estimator. The second one $\widehat{\boldsymbol g}^*(x)$ is an ideal version of our estimator assuming that the true topic matrix $G$ is given. 

Now, we decompose the error $\widehat{\boldsymbol g}(x)-{\boldsymbol g}(x)$ into three terms: 
\beq \label{error-decomposition}
\widehat{\boldsymbol g}^+(x)-{\boldsymbol g}(x) = \underbrace{{\boldsymbol g}^*(x)-{\boldsymbol g}(x)}_{\text{bias}} + \underbrace{\widehat{\boldsymbol g}^*(x)-{\boldsymbol g}^*(x)}_{\text{variance}} + \underbrace{\widehat{\boldsymbol g}^+(x)-\widehat{\boldsymbol g}^*(x)}_{\text{secondary error}}. 
\eeq
The first term is the non-stochastic `bias' term, 
the second term is the main stochastic error term, arising from the randomness in $(S, T)$, and the third term is the secondary stochastic error term, due to estimating $G$ by  $\widehat{G}$. We will control these three terms separately.

\subsection{The non-stochastic `bias' term} \label{subsec:UBproof-bias}
We bound the first term in \eqref{error-decomposition}. Recall that $T$ and $S(x)$ are defined in \eqref{def-T-U}-\eqref{def-S-K}. Let $K_{ij}(x)$ and $U_{ijm}$ be the same as in those definitions. 
We introduce a vector
\beq \label{def-k(x)}
\bm{k}(x) = \bigl(K_{11}(x), K_{21}(x), \ldots, K_{n1}(x)\bigr)'\in\mathbb{R}^n.  
\eeq
For any fixed $(i, m)$, $\mathbb{E}[U_{ijm}]$ and $\mathbb{E}[K_{ij}(x)]$ are the same across $j$. In addition, $(U_{ijm}, K_{ij}(x))$ is independent of $U_{ij'm'}$ when $j\neq j'$ (this is true even when $m=m'$).  We immediately have:
\[
\mathbb{E}[T_{mm'}] = \sum_{i=1}^n \mathbb{E}[U_{i1m}]\cdot \mathbb{E}[U_{i1m'}], \qquad \mathbb{E}[S_{mi}(x)] = \mathbb{E}[k_i(x)]\cdot\mathbb{E}[U_{i1m}]. 
\]  
Moreover, we note that $
\mathbb{E}[U_{i1m}]=\mathbb{P}(X_{ij}\in {\cal B}_m)=\sum_{k=1}^{K}\pi_i(k)G_{mk} = (G\Pi')_{mi}$. 
Combining it with the above equations gives
\[
\mathbb{E}[S(x)] =G\Pi' \diag\bigl(\mbb{E}[k(x)]\bigr), \qquad \mathbb{E}[T] = G\Pi'\Pi G'. 
\]
We plug $(\bar{G}, \bar{S}, \bar{T})=(G, \mathbb{E}[S(x)], \mathbb{E}[T])$  into \eqref{g(GST)}. It follows by direct calculations that 
\beq \label{g-star-expression}
\bg^*(x)= (\Pi'\Pi)^{-1}\Pi'\mathbb{E}[{\boldsymbol k}(x)]. 
\eeq
Meanwhile, since our model implies that ${\boldsymbol f}(x)=\Pi {\boldsymbol g}(x)$, we can write 
\beq \label{g-true-expression}
{\boldsymbol g}(x)=(\Pi'\Pi)^{-1}\Pi'{\boldsymbol f}(x).
\eeq
Comparing \eqref{g-true-expression} with \eqref{g-star-expression}, the `bias' term comes from approximating ${\boldsymbol f}(x)$ by $\mathbb{E}[{\boldsymbol k}(x)]$. Using properties of the kernel function ${\cal K}(\cdot)$, we can prove the following lemma:   

\begin{lemma} \label{lem:error1}
Under Assumptions~\ref{assump1}-\ref{assump2}, there exists a constant $C_1>0$ such that
\[
\int \|{\boldsymbol g}^*(x)-{\boldsymbol g}(x)\|^2dx\leq C_1 Kh^{2\beta}.
\]

\end{lemma}

\subsection{The main stochastic error} \label{subsec:stochastic-error-1}
We study the second term in \eqref{error-decomposition}. Define an event 
\beq \label{small-prob-event}
{\cal F} = \biggl\{ \lambda_{\min}(\widehat{G}'T\widehat{G}) \geq   \frac{\tilde{c} Kn}{M^2\log^2(n)} \biggr\}.
\eeq
Let $I_{\cal F}$ be a Bernoulli variable indicating that the event ${\cal F}$ happens.  
We write 
\begin{align*}
\|\widehat{\boldsymbol g}^*(x)-{\boldsymbol g}^*(x)\|^2 &= \|\widehat{\boldsymbol g}^*(x)-{\boldsymbol g}^*(x)\|^2 \cdot I_{\cal F} +  \|\widehat{\boldsymbol g}^*(x)-{\boldsymbol g}^*(x)\|^2\cdot I_{{\cal F}^c} \cr
&\leq \|\widehat{\boldsymbol g}^*(x)-{\boldsymbol g}^*(x)\|^2 \cdot I_{\cal F} + 2\|\widehat{\boldsymbol g}^*(x)\|^2 \cdot I_{{\cal F}^c} + 2\|{\boldsymbol g}^*(x)\|^2 \cdot I_{{\cal F}^c}.  
\end{align*}
It follows that
\begin{align} \label{hat-g-star-decompose}
\mathbb{E}\biggl[ \int\|\widehat{\boldsymbol g}^*(x) & -{\boldsymbol g}^*(x)\|^2dx  \biggr]  \leq 
\mathbb{E}\left[ \int\|\widehat{\boldsymbol g}^*(x)-{\boldsymbol g}^*(x)\|^2 dx \cdot I_{\cal F}\right] \cr
&\qquad +2 \mathbb{E}\left[ \int \|\widehat{\boldsymbol g}^*(x)\|^2dx \cdot I_{{\cal F}^c} \right] + 2\mathbb{P}({\cal F}^c)\cdot \int \|{\boldsymbol g}^*(x)\|^2dx. 
\end{align}
On the right hand side of \eqref{hat-g-star-decompose}, the first term is dominating. Therefore, we focus on controlling the first term. The other two terms will be considered in the end of this subsection. 


To study the first term in \eqref{hat-g-star-decompose}, we notice that on the event 
${\cal F}$, the quantity $\epsilon$ in \eqref{our-estimator2} is exactly zero, so that  
\beq \label{hat-g-star-2}
\widehat{\boldsymbol g}^{*}(x) = {\boldsymbol g}\bigl(G, S(x), T, 0 \bigr). 
\eeq
When \eqref{hat-g-star-2} holds, we can express $\widehat{\bg}^*(x)$ in terms of $\bg^*(x)$ and some error terms. Recall that $Q=G(G'G)^{-1}$. We further write $\Omega=Q(\Pi'\Pi)^{-1}$ for brevity. 
\begin{lemma} \label{lem:error2-decompose}
Let $\widehat{\bg}^*(x)$ and $\bg^*(x)$ be as defined in \eqref{all-estimators}. On the event ${\cal F}$, 
\beq \label{error2-decompose}
\widehat{\boldsymbol g}^*(x) = \bigl[I_K +\Omega'(T-\mathbb{E}T)Q\bigr]^{-1}\bigl[\bg^*(x) + \Omega'(S-\mathbb{E}S){\bf 1}_n\bigr]. 
\eeq 
\end{lemma}
When $\Omega'(S-\mathbb{E}S){\bf 1}_n$ and $\Omega'(T-\mathbb{E}T)Q$ are zero, the right hand side of \eqref{error2-decompose} is exactly equal to $\bg^*(x)$. 
Consequently, the difference between $\widehat{\bg}^*(x)$ and $\bg^*(x)$ can be controlled by
\[
\|\Omega'(S-\mathbb{E}S){\bf 1}_n\|\qquad \mbox{and}\qquad  \|\Omega'(T-\mathbb{E}T)Q\|.
\]
Below, we study these two error terms separately.  

First, we consider the $K$-dimensional random vector $\Omega'(S -\mathbb{E}S){\bf 1}_n$. 
For each $1\leq k\leq K$, define a tri-variate function 
\[
h_k(x; x_1, x_2)=\sum_{m=1}^M\Omega_{mk}{\cal K}_h(x-x_1)1\{x_2\in {\cal B}_m\}.
\]
Using the definition of $S=S(x)$ in \eqref{def-S-K}, we can immediately deduce that 
\beq \label{error21}
e_k'\Omega'(S -\mathbb{E}S){\bf 1}_n =u_k(x)-\mathbb{E}u_k(x), \quad u_k(x):= \sum_{i=1}^n\frac{\sum_{1\leq j\neq j'\leq N_i}h_k(x; X_{ij}, X_{ij'})}{N_i(N_i-1)}.  
\eeq
Here, the variables $\{X_{ij}\}_{1\le i\le n,1\le j\le N}$ are independent across $i$; and for any fixed $i$, they are i.i.d. across $j$. Hence, $u_k(x)$ is an incomplete $U$-process indexed by $x$. 
Applying Lemma~\ref{lem:uprocess-bound} to the specific form in \eqref{error21}, we can prove the following result:
\begin{lemma} \label{lem:error21}
Under Assumption~\ref{assump2}, 
there exists a constant $C_2>0$ such that 
\[
\mbb{E}\left[\int \bigl\|\Omega'[S(x)-\mathbb{E}S(x)]{\bf 1}_n\bigr\|^2dx\right] \le C_2 (Nnh)^{-1}K^2. 
\]
\end{lemma}

Next, we study $\Omega'(T-\mathbb{E}T)Q$. This is an asymmetric matrix. We first shift our attention to a symmetric matrix as follows: Note that  $\Omega= Q(\Pi'\Pi)^{-1}$. Since $\|(\Pi'\Pi)^{-1}\|\leq Cn^{-1}K$ by Assumption~\ref{assump2}(a), it follows that 
\beq  \label{error22-temp1}
\|\Omega'(T-\mathbb{E}T)Q\|\leq Cn^{-1}K\cdot \|Q'(T-\mathbb{E}T)Q\|. 
\eeq 
It remains to bound $\|Q'(T-\mathbb{E}T)Q\|$.

We apply the approach in \cite{vershynin2010introduction} to bound the spectral norm of the symmetric matrix $Q'(T-\mathbb{E}T)Q$. 
Let ${\cal N}_\epsilon$ denote an $\epsilon$-net of the unit sphere of $\mathbb{R}^K$.  By \cite[Lemma 5.4]{vershynin2010introduction}, 
\beq \label{error22-temp2}
\|W\|\leq (1-2\epsilon)^{-1}\max_{v\in{\cal N}_\epsilon}|v'Wv|, \qquad \mbox{for any symmetric matrix $W\in\mathbb{R}^{K\times K}$}. 
\eeq
To this end, the analysis reduces to bounding $|v'Q'(T-\mathbb{E}T)Qv|$ for a fixed unit-norm vector $v\in\mathbb{R}^K$.
By \eqref{def-T-U}, 
\[
T = \sum_{i=1}^n \frac{1}{N_i(N_i-1)} \sum_{1\leq j\neq j'\leq N} Q'U_{ij}U_{ij'}'Q.
\]
Let $s_{ij}=v'Q'U_{ij}$. 
We have
\beq \label{u-stat-T-ET}
v'Q'(T-\mathbb{E}T)Qv 
= \frac{1}{N(N-1)}\sum_{i=1}^n \sum_{1\leq j\neq j'\leq N}\bigl(s_{ij}s_{ij'}'  - \mathbb{E}[s_{ij}s_{ij'}']\bigr). 
\eeq
Since $\{s_{ij}\}_{1\leq i\leq n,1\leq j\leq N}$ are independent random variables, the above defines an incomplete $U$-statistic. We apply Lemma~\ref{lem:ustats-bound} to bound the right hand side of \eqref{u-stat-T-ET} and further combine this result with \eqref{error22-temp2}. It leads to the following lemma:
\begin{lemma} \label{lem:error22}
Under the conditions of Theorem~\ref{thm:main}, 
there exists a constant $C_3>0$ such that with probability $1-o((Nn)^{-5})$,  
\[
\|\Omega'(T-\mathbb{E}T)Q\|\le C_3(nN)^{-1/2}\sqrt{K(\log(Nn)+K)}.
\]
\end{lemma}

Finally, we combine Lemmas~\ref{lem:error21}-\ref{lem:error22} with the expression in  \eqref{error2-decompose}. By some elementary analysis, we can prove the following lemma. 



\begin{lemma}\label{lem:error23}
Under the conditions of Theorem~\ref{thm:main}, there is a constant $C_4>0$ such that
\begin{align*}
    \mbb{E}\left[\int \|\bm{\widehat{g}}^*(x)-\bm{g}^*(x)\|^2dx\cdot I_{\cal F}\right]\le C_4K\left(\frac{K^2}{Nn}+\frac{K}{Nnh}\right).
\end{align*}
\end{lemma}

Lemma~\ref{lem:error23} bounds the first term in \eqref{hat-g-star-decompose}. We still need to bound the other two terms. In our analysis, we first deduce a high-probability lower bound for $\lambda_{\min}(\widehat{G}'T\widehat{G})$. It implies that
\[
\mathbb{P}({\cal F}^c) = o((Nn)^{-5}). 
\] 
This permits us to control the third term in \eqref{hat-g-star-decompose}. 
Moreover, on the event ${\cal F}^c$, $\lambda_{\min}(\widehat{G}'T\widehat{G})$ is either larger than $\frac{Kn}{M^2\log^3(n)}$, or smaller. In the first case, $\epsilon=0$; and in the second case, $\epsilon=Kn/M^2$. In both cases, it holds that
\[
\lambda_{\min}\bigl( \widehat{G}'T\widehat{G} + \epsilon I_K\bigr)\geq \frac{Kn}{M^2\log^3(n)}. 
\]
Using this result, we can derive an upper bound for $\|\widehat{\boldsymbol g}^*(x)\|^2$ over the event ${\cal F}^c$. Denote this upper bound by $\zeta_n$. 
Then, the second term in \eqref{hat-g-star-decompose} is upper bounded by $\zeta_n^2\cdot \mathbb{P}({\cal F}^c)$, which can also be well-controlled. 
We formally prove the following lemma:

\begin{lemma}\label{lem:error2}
Under the conditions of Theorem~\ref{thm:main}, there is a constant $C'_4>0$ such that
\begin{align*}
    \mbb{E}\left[\int \|\bm{\widehat{g}}^*(x)-\bm{g}^*(x)\|^2dx\right]\le C'_4K\left(\frac{K^2}{Nn}+\frac{K}{Nnh}\right).
\end{align*}
\end{lemma}

\subsection{The secondary stochastic error, and proof of Theorem~\ref{thm:main}} \label{subsec:stochastic-error-2}
We now control the last term in \eqref{error-decomposition}. This error arises from not knowing $G$ and estimating it by $\widehat{G}$. 

Let ${\cal F}$ be the event in \eqref{small-prob-event}. 
Over the event ${\cal F}$, we re-write $\widehat{\bg}(x)$ as a more explicit expression of $\widehat{\bg}^*(x)$. 

\begin{lemma} \label{lem:error3-decompose}
Write for brevity $R = (Q'TQ)^{-1}(G'G)^{-1}$, $\Delta_1 = \widehat{G}'\widehat{G}- G'G$, and $\Delta_2=\widehat{G}'T\widehat{G}-G'TG$. Over the event ${\cal F}$, 
\beq \label{error3-decompose}
\widehat{\boldsymbol g}(x) = \bigl[ I_K + (R\Delta_2-\Delta_1)(G'G+\Delta_1)^{-1}\bigr]^{-1}\bigl[ \widehat{\bg}^*(x) + R(\widehat{G}-G)'S{\bf 1}_n\bigr]. 
\eeq 
\end{lemma}

We note that the matrix $R$ is relatively easy to study, using the result of Lemma~\ref{lem:error22}. What remains is to study the matrices $\Delta_1$ and $\Delta_2$ and the vector $R(\widehat{G}-G)'S{\bf 1}_n$. All three quantities are related to the deviation of $\widehat{G}$ from $G$ and $T$ from $\mbb{E}[T]$. We need the following lemma:
\begin{lemma} \label{lem:error3-prep}
Suppose the conditions of Theorem~\ref{thm:main} hold. Let ${\cal E}$ denote the event that $\|T-\mbb{E}[T]\|\leq M^{-1}n$ and $\|Q'(T-\mbb{E}[T])Q\|\leq n/(K\log^{1/2}(Nn))$, and let $I_{\cal E}$ be the Bernoulli variable indicating that this event happens. There exist positive constants  $C_5$-$C_7$ such that the following statements are true:
\begin{itemize}
\item $\|\Delta_1\|\le C_5\sqrt{\frac{K}{M}}\|\widehat{G}- G\|$, over the event ${\cal E}$. 
\item $\|R\Delta_2\|\leq C_6\sqrt{\frac{K}{M}}\|\widehat{G}- G\|$, over the event ${\cal E}$. 
\item $\mbb{E}\left[\int \|R(\widehat{G}-G)'S{\bf 1}_n\|^2dx\cdot I_{\cal E}\right]\le  C_7M\|\widehat{G}- G\|^2$.
\end{itemize}
\end{lemma} 

In order to apply Lemma~\ref{lem:error3-prep}, we must study the probability of the event $\mathcal{E}$. 
This is contained in  Lemma~\ref{lem:spectralnorms} of the supplement, where we bound $\|T-\mathbb{E}[T]\|$ and $\|Q'(T-\mathbb{E}[T])Q\|$ and show that the event $\mathcal{E}$ has an overwhelming probability.  

We now combine the statements in Lemma~\ref{lem:error3-prep} with the expression of $\widehat{\bg}(x)$ in Lemma~\ref{lem:error3-decompose}. 
Since this expression is only correct on the event ${\cal F}$ and the statements in Lemma~\ref{lem:error3-prep} are with respect to the event ${\cal E}$, we will consider an event 
\[
{\cal F}^* = {\cal F}\cap {\cal E}. 
\] 
We then decompose the error similarly as in  \eqref{hat-g-star-decompose}. We no longer provide these details but directly present the following lemma, whose proof is in the supplementary material. 
\begin{lemma} \label{lem:error3}
Under the conditions of Theorem~\ref{thm:main}, there is a constant $C_8>0$ such that
\[
    \mbb{E}\left[\int \|\widehat{\bm{g}}^*(x)-\widehat{\bm{g}}(x)\|^2dx\right]\le C_8\left[M\delta_n^2+\frac{KM^4\log^4(Nn)}{(Nn)^5h}\right].
\]
\end{lemma}

The conclusion of Theorem~\ref{thm:main} follows immediately by plugging Lemma~\ref{lem:error1}, Lemma~\ref{lem:error2} and Lemma~\ref{lem:error3} into the decomposition in \eqref{error-decomposition}.

\section{Proof sketch of the Minimax Lower Bound}\label{sec:LBproof}
In this section, we provide a sketch of the proof of Theorem~\ref{thm:lowerbound}, leaving the full version in the supplementary material. 
 
Let ${\cal F}_K^*$ be the space of $K$-densities, i.e., each element in ${\cal F}_K^*$ takes the form of $\bm{g}(x)=(g_1(x),\ldots,g_K(x))'$. For any $\bm{g}, \widetilde{\bm{g}}\in {\cal F}_K^*$, let  
\[
d(\bm{g}, \widetilde{\bm{g}}) := \left(\int \|\bm{g}(x)-\widetilde{\bm{g}}(x)\|^2dx\right)^{1/2}=\left(\sum_{k=1}^K\|g_k-\tilde{g}_K\|^2_{L^2}\right)^{1/2}, 
\]
where $\|\cdot\|_{L^2}$ is the standard $L^2$-norm. It is easy to see that $d(\cdot,\cdot)$ is a distance defined on $\Theta^*$ (i.e., the triangular inequality holds). Recall that $\Theta_{\beta}$ is the collection of all $({\boldsymbol g}, \Pi)$ such that Assumptions~\ref{assump1}-\ref{assump2} are satisfied. We fix $\Pi^*$ such that $\pi^*_i=e_k$ for $n_k$ out of $n$ groups, where $n_k\asymp n/K$ and $n_1+n_2+\cdots +n_K=n$. Let ${\cal F}_{K}(\beta, \Pi^*)$ be the collection of $\bm{g}$ such that $(\bm{g}, \Pi^*)\in\Theta_\beta$. To show the claim, it is sufficient to show that 
\beq \label{LB-illustration1}
\inf_{\widehat{g}} \sup_{\bm{g}\in{\cal F}_K(\beta, \Pi^*)} \mathbb{E}_{(\bm{g}, \Pi^*)} \bigl[d^2(\widehat{\bm{g}}, \bm{g})\bigr]\geq c\epsilon^2_n, \qquad \mbox{with}\quad \epsilon^2_n:= K\Bigl(\frac{K}{Nn}\Bigr)^{\frac{2\beta}{2\beta+1}}. 
\eeq
According to Theorem 2.7 of \cite{tsybakov2009introduction}, \eqref{LB-illustration1} holds if for some integer $J\geq 2$, we can construct $\bm{g}^{(0)}, \bm{g}^{(1)}, \ldots, \bm{g}^{(J)}$ such that 
\beq \label{LB-illustration2}
\min_{0\leq s\neq s'\leq J}d^2(\bm{g}^{(s)}, \bm{g}^{(s')})\geq c\epsilon_n^2, \qquad \frac{1}{J}\sum_{s=1}^J \mathrm{KL}(\mbb{P}_s,\mbb{P}_0) = O(\log(J)),
\eeq
where $\mbb{P}_s$ denotes the probability measure associated with the model \eqref{model-1}-\eqref{model-2} with parameters $\bm{g}^{(s)}$ and $\Pi^*$, and $\text{KL}(\cdot,\cdot)$ is the Kullback-Leibler divergence. 

We outline how to construct $\{\bm{g}^{(s)}\}_{s=0}^J$. In the first step, we fix a properly large constant $T>0$ and construct $\bm{g}^{(0)}$ such that (i) each $g_k^{(0)}$ is in the Nikol’ski class with smoothness $\beta$ (i.e., satisfying Assumption \ref{assump1}), (b) each $g_k^{(0)}$ has an anchor region (see Assumption~\ref{assump-anchor}) that is located in $(-\infty, -T)\cup (T, \infty)$, and (c) each $g_k^{(0)}$ is uniformly lower bounded by a constant in $[-T, T]$. This step uses certain bump functions as in \eqref{bumps}, and the construction details are in the supplement. 

In the second step, we perturb each $g_k^{(0)}$ in the region of $[-T, T]$. Let 
\[
B=\Big\lceil\Bigl(\frac{Nn}{K}\Bigr)^{\frac{1}{2\beta+1}}\Bigr\rceil, \qquad \delta = a_0\Bigl(\frac{K}{Nn}\Bigr)^{\frac{\beta}{2\beta+1}},\;\; \mbox{for a properly small constant $a_0>0$}. 
\]
We divide $[-T, T]$ into $B$ equal intervals, and let $x_1, x_2, \ldots, x_B$ be the centers of these intervals.
Let $\psi$ be a function that is in the Nikol’ski class with smoothness $\beta$ and satisfies:  
\[
\mathrm{Support}(\psi)=[-1,1], \quad \int_0^1 \psi(x)dx=0, \quad \int_0^1 \psi^2(x)dx=1, \quad \|\psi\|_{\infty}<\infty. 
\]
It can be shown that such $\psi$ exists. We let
\[
    \phi_b(x):=\delta\cdot \psi\left(Bx-x_b\right), \qquad 1\leq b \leq B. 
\]
It is seen that $\{\phi_b\}_{b=1}^B$ have disjoint support within $[-T,T]$, $\int \phi_b(u)du=0$, and 
\[
\int \phi^2_b(u)du\asymp \delta^2B^{-1}.
\] 
Let $J=\lceil 2^{KB/8}\rceil$. The Gilbert-Varshamov bound implies that there exist $(J+1)$ binary vectors $\nu^{(0)},\ldots,\nu^{(J)}\in \{0,1\}^{KB}$, such that $\nu^{(0)}={\bf 0}$ and $\|\nu^{(s)}-\nu^{(s')}\|_1\ge KB/8$ for all $0\leq s\neq s'\leq J$. We re-arrange each $\nu^{(s)}$ into an $K\times B$ matrix $\Omega^{(s)}$, and let $\omega_{k,b}^{(s)}$ denote its $(k, b)$th entry, for $1\leq k\leq K$ and $1\leq b\leq B$. It follows that
\beq\label{LB-illustration3}
|\Omega^{(s)}-\Omega^{(s')}|_1:= \sum_{k=1}^K \sum_{b=1}^B |\omega_{k,b}^{(s)}-\omega_{k,b}^{(s')}|\geq KB/8. 
\eeq 
We then construct $\bm{g}^{(1)},\ldots, \bm{g}^{(J)}$ by 
\beq \label{LB-illustration4}
    g_k^{(s)}(x)=g_k^{(0)}(x)+ \sum_{b=1}^B\omega_{k,b}^{(s)}\phi_b(x),\quad \text{for }\ 1\leq s\leq J \mbox{ and }1\leq k\leq K.
\eeq

To show that these constructed $\bm{g}^{(s)}$ satisfy \eqref{LB-illustration2}, we note that by our construction, each $\phi_b$ belongs to the Nikol’ski class with smoothness $\beta$, $\phi_b$'s have disjoint supports, $\int \phi_b(u)du=0$, and $\|\phi_b\|^2_{L^2}=\|\phi_1\|^2_{L^2}\asymp \delta^2 B^{-1}$. It follows that each $g_k^{(s)}$ is a density belonging to the Nikol’ski class, and by direct calculations, 
\[
d^2\bigl(\bm{g}^{(s)}, \bm{g}^{(s')}\bigr) = |\Omega^{(s)}-\Omega^{(s')}|\cdot \|\phi_1\|^2_{L^2} \geq C K\delta^2\geq c\epsilon_n^2.  
\]
This proves the first part of \eqref{LB-illustration2}. Furthermore, in the supplementary material, we show that 
\[
\max_{1\leq s\leq J} \mathrm{KL}(\mathbb{P}_s, \mathbb{P}_0) = O(Nn\delta^2). 
\]
Noticing that $Nn\delta^2 = K^{\frac{2\beta}{2\beta+1}}(Nn)^{\frac{1}{2\beta+1}}= KB$, we immediately obtain:
\[
    \frac{1}{J}\sum_{s=1}^J \mathrm{KL}(\mbb{P}_s,\mbb{P}_0)=O(Nn\delta^2)=O(KB)=O(\log J).
\]
This proves the second part of \eqref{LB-illustration2} and completes the proof.

\section{Discussion} \label{sec:Discuss}
We consider linearly unmixing $n$ convex combinations of nonparametric densities. 
We propose a novel estimator and show that it attains the optimal rate of convergence in the whole smooth regime, offering a much stronger theoretical guarantee than existing methods in the literature.  
One limitation of our work is that we need $K=O((Nn)^{c^*_{\beta,d}})$ for some constant $c^*_{\beta,d}\in (0,1/4)$, in order to achieve the optimal rate. 
When $K\gg (Nn)^{c^*_{\beta d}}$, we still obtain an upper bound (by letting $M=K\log(Nn)$ in the first equation of Section~\ref{subsec:UB}):
\[
\mbb{E}\left[\int \|\widehat{\bm{g}}^+(x)-\bm{g}(x)\|^2dx\right]\le C_0 \cdot K \left[\Bigl(\frac{K}{Nn}\Bigr)^{\frac{2\beta}{2\beta+1}}+\frac{K^4\log^2(Nn)}{Nn} \right]. 
\]
Compared to the lower bound, there is an extra term. Whether this term can be improved is unclear. The first term above arises from the oracle estimator that is given $\Pi$, and the second term comes from the error of estimating $\Pi$. 
We conjecture that when $K$ grows with $Nn$ at a fast speed, the plug-in error of $\Pi$ is fundamentally non-negligible. To support this conjecture, we may need tighter lower bound and/or tighter upper bound. We leave it to future work. 

In our framework, the dimension $d$ is fixed. When $d$ grows with $Nn$, we may assume that $X_{ij}$'s are mappings of latent low-dimensional variables (e.g., this may hold in the first application in Section~\ref{sec:Intro}, where $X_{ij}$'s represent the contextual word embeddings from a pre-trained language model). Suppose the latent variables $Z_{ij}$ are in dimension $r$. Let $\epsilon_{ij}$ be i.i.d. variables from ${\cal N}(0, \sigma^2I_d)$, and let $B$ be a $d\times r$ matrix such that $B'B=I_r$. A simple model is $X_{ij}=BZ_{ij}+\epsilon_{ij}$, where $Z_{ij}$ follow the model in \eqref{model-1}-\eqref{model-2}. 
In this set-up, we let $X^{\text{vect}}\in\mathbb{R}^{d\times \bar{N}n}$ be obtained by re-arranging all data points into a matrix. 
We can estimate $Z_{ij}$'s from the first $r$ right singular vectors of $X^{\text{vect}}$.  
Then, estimating $g_1(z), \ldots, g_K(z)$ based on $\hat{Z}_{ij}$'s is an error-in-variable version of our problem. 
It is possible to adapt our estimator to this setting. This is an interesting future direction. 

While we focus on the nonparametric setting, our methodology is also useful in the parametric setting where $g_k(x)=g(x; \lambda_k)$ (e.g., see \cite{zhang2009generalized, doss2023optimal} about related problems).  
The two steps of binning data into counts and applying topic modeling to estimate $\Pi$ are still applicable. 
In the last step, instead of using KDE, we can consider a plug-in maximum likelihood estimator (MLE) for $\lambda_1,\ldots, \lambda_K$, possibly with de-biasing. Compared to the variational EM, this approach will be computationally more efficient, due to leveraging existing topic modeling algorithms. 

Our model is closely related to mixed membership models in various applications \citep{agterberg2025estimating,ma2021determining}. 
In these problems, the settings are typically fully parametric. 
Our methodology and theory together offer a possible way of extending them to nonparametric settings. 
In a broad context, our model is also related to many statistical problems involving latent structure and nonparametric components, where the estimation difficulty depends jointly on structural complexity and function smoothness; see, e.g., graphon estimation \citep{gao2015rate}, semiparametric latent variable models \citep{liu2012high}, and smoothed tensor estimation \citep{lee2025statistical, han2024tensor}. 




In this paper, we assume the data are i.i.d. within each of the $n$ groups, and our analysis relies on properties of incomplete U-statistics and incomplete U-processes. When the data within each group are non-independent, we may continue to apply our estimator, but the analysis will require more sophisticated technical tools (e.g., \cite{wu2008empirical}).

\appendix

\section{Analysis of our estimator} \label{supp:UBproofs}


\subsection{General tools}
\begin{lemma}[Minkowski's integral inequality]\label{lem:Minkowski}
    Suppose that $(S_1,\mu_1)$ and $(S_2,\mu_2)$ are two sigma-finite measure spaces and $F:S_1\times S_2\rightarrow \mbb{R}$ is measurable. Then, it holds that
    \[
        \left[ \int_{S_2} \left| \int_{S_1} F(x, y) \, \mu_1(dx) \right|^p \mu_2(dy) \right]^{1/p}\le\int_{S_1} \left( \int_{S_2} |F(x, y)|^p \, \mu_2(dy) \right)^{1/p} \mu_1(dx),
    \]
    for $1\le p<\infty$.
\end{lemma}

\begin{lemma}[Lemma 7.26 in \cite{lafferty2008concentration}]\label{lem:mgfbound}
    Suppose that $|X|\le c$, $\mbb{E}X=0$ and $\Var(X)=\sigma^2$. Then, for any $t>0$,
    \begin{align*}
        \mbb{E}e^{tX}\le \text{exp}\left(t^2\sigma^2\left(\frac{e^{tc}-1-tc}{(tc)^2}\right)\right).
    \end{align*}
\end{lemma}

\begin{lemma}[Bernstein inequality] \label{lem:Bernstein}
Let $Z_1, Z_2, \ldots, Z_m$ be independent mean-zero variables. Suppose $|Z_i|\leq b$ and $\sum_{i=1}^m\mathrm{Var}(Z_i)\leq \sigma^2$. For any $\delta\in (0,1)$, with probability $1-\delta$, 
\begin{align*}
\Bigl|\sum_{i=1}^m Z_i\Bigr|\leq 2\sigma\sqrt{\log(2/\delta)} + (4b/3)\log(2/\delta). 
\end{align*}
\end{lemma}

\subsection{A useful  lemma and its proof}

In this subsection, we present a technical lemma. It contains a collection of statements that will be used repeatedly in the later proofs. 

Let $f_i^{\text{hist}}\in\mathbb{R}^M$ be defined such that $f_i^{\text{hist}}(m)=\int_{x\in {\cal B}_m}f_i(x)dx$. Then, we write $F^{\text{hist}}=[f_1^{\text{hist}}, \ldots, f_n^{\text{hist}}]\in\mathbb{R}^{M\times n}$. By our model, for $1\leq i\leq n,1\leq j\leq N,1\leq m\leq M$, 
\[
\mathbb{E}[U_{ijm}] = f_i^{\text{hist}}(m)= \sum_{k=1}^K\pi_i(k)g_k^{\text{hist}}(m). 
\]
We can prove the following lemma:
\begin{lemma} \label{lem:key}
Under Assumption \ref{assump2}, there exist positive constants $c_5$-$c_7$ with $c_5\ge  1$ and $c_6\ge 1$ such that the following statements are true:
\begin{itemize}
\item[(a)] $\mathbb{E}[S(x)] =\diag(N^{-1}\mbb{E}[K(x)]\bm{1}_N)\Pi G' $, and $\mathbb{E}[T] = G\Pi'\Pi G' $.
\item[(b)] There exists a set of bins $\{\mcal{B}_m\}_{m=1}^M$ such that for all $1\le m\le M$, if $M/K\rightarrow \infty$, $n^{-1}\sum_{i=1}^nf_i^{\text{hist}}(m) \leq c_5M^{-1}$. Additionally, if we assume $\min_{1\le k\le K}\sum_{i=1}^n\pi_{ik}\ge c_2^{-1}K^{-1}n$ and $\min_{1\le m\le M}\sum_{k=1}G_{mk}\ge c_4^{-1}M^{-1}K$ for some constants $c_2>1$ and $c_4>1$, then $n^{-1}\sum_{i=1}^nf_i^{\text{hist}}(m) \geq c_5^{-1}M^{-1}$.
\item [(c)] If $M/K\rightarrow\infty$, there exists a set of bins $\{\mcal{B}_m\}_{m=1}^M$ such that $\lambda_{\min}(\Sigma_G)\asymp \lambda_{\max}(\Sigma_G)\asymp 1$, $\|G\mbf{1}_K\|_{\infty}=O(M^{-1}K)$ and the smallest entry of $G\mbf{1}_K$ is of the order $M^{-1}K$. Consequently, $c_6^{-1}(K/M)^{1/2}\le \|G\|\le c_6(K/M)^{1/2}$, $ \|Q\|\le c_7(M/K)^{1/2}$.
\end{itemize}
\end{lemma}

\noindent
{\it Proof of Lemma~\ref{lem:key}:} For (a), we have: for $1\le i\le n$ and $1\le m\le M$,
\begin{align*}
    \mbb{E}[S_{mi}(x)] &= \frac{1}{N(N-1)}\sum_{1\leq j\neq j'\leq N} \mbb{E}[K_{ij}(x)] \mbb{E}[U_{ij'm}]\\
    &=\frac{1}{N}\sum_{1\leq j\leq N} \mbb{E}[K_{ij}(x)] \sum_{k=1}^K\pi_i(k)g_k^{\text{hist}}(m).
\end{align*}
This proves the first identity. Moreover, for $1\le m,m'\le M$, we also have:
\begin{align*}
    \mbb{E}[T_{mm'}]&=\frac{1}{N(N-1)}  \sum_{i=1}^n \sum_{1\leq j\neq j'\leq N} \mbb{E}[U_{ijm}]\mbb{E}[U_{ij'm'}]=\sum_{i=1}^n\mbb{E}[U_{i1m}]\mbb{E}[U_{i2m'}]\\
    &=\sum_{k,k=1}^Kg_k^{\text{hist}}(m)\left[\sum_{i=1}^n\pi_i(k)\pi_i(k')\right]g_{k'}^{\text{hist}}(m').
\end{align*}
This proves the second equality.

In the following, we first prove the last statement (c) as the proof of statement (b) relies on it. This is actually a discretized version of Assumption \ref{assump2}. Recall $\bar{g}(x)=K^{-1}\sum_{k=1}^Kg_k(x)$. Define a CDF $H(x):=\int_{-\infty}^x\bar{g}(x)dx$. Using the change of variable $u=H(x)$, we have:
\begin{align*}
    \Sigma_{\bm{g}}=\frac{1}{K}\int \bm{g}(x)\bm{g}(x)'dx=\frac{1}{K}\int_0^1 \tilde{\bm{g}}(u)\tilde{\bm{g}}(u)'\omega(u)du,\quad \omega(u):=\bar{g}(x(u)),
\end{align*}
where $\tilde{\bm{g}}(u)=(\tilde{g}_1(u),\ldots,\tilde{g}_K(u))'$ with $\tilde{g}_k(u)=g_k(x(u))/\omega(u)$ for $1\le k\le K$. Now, we have translated the integral from $\mbb{R}$ to $[0,1]$. Letting $I_m=\left(\frac{m-1}{M},\frac{m}{M}\right]$, we then construct $\mcal{B}_m=H^{-1}(I_m)=x(I_m)$. Hence, by construction, it holds that 
\begin{align*}
    \int_{\mcal{B}_m}\bar{g}(x)dx=|I_m|=\frac{1}{M}.
\end{align*}
This implies for any $1\le m\le M$, $(G\mbf{1}_K)_m=\sum_{k=1}^K\int_{\mcal{B}_m}g_k(x)dx=\frac{K}{M}$. Then, we focus on proving the eigenvalue bounds for $\Sigma_G$. Its proof mainly follows \cite[Proof of Lemma C.3]{austern2025poisson}, which we extend to the full domain $\mbb{R}$ instead of a compact set and  growing $K$ here. Using the same change of variable above, we have:
\begin{align*}
    G_{mk}=\int_{\mcal{B}_m}g_k(x)dx = \int_{I_m}\tilde{g}_k(u)du,
\end{align*}
and
\begin{align*}
    \Sigma_G = \frac{M}{K}\sum_{m=1}^M\left(\int_{I_m}\tilde{\bm{g}}(u)du\right)\left(\int_{I_m}\tilde{\bm{g}}(u)du\right)'.
\end{align*}
Now, we define
\begin{align*}
    \Sigma_{\tilde{\bm{g}}}=\frac{1}{K}\int_0^1 \tilde{\bm{g}}(u)\tilde{\bm{g}}(u)'du.
\end{align*}
Since both the integrals in $\Sigma_G$ and $\Sigma_{\tilde{\bm{g}}}$ are defined on $[0,1]$, following \cite[Proof of Lemma C.3]{austern2025poisson}, we immediately obtain $\|\Sigma_G-\Sigma_{\tilde{\bm{g}}}\|\le CKM^{-\beta\wedge 1}=o(1)$ for some constant $C>0$. Using Assumption \ref{assump2}, this implies $\lambda_{\min}(\Sigma_G)\ge \lambda_{\min}(\Sigma_{\tilde{\bm{g}}})-o(1)$ and $\lambda_{\max}(\Sigma_G)\le \lambda_{\max}(\Sigma_{\tilde{\bm{g}}})+o(1)$. It then remains to prove $\lambda_{\max}(\Sigma_{\tilde{\bm{g}}})\asymp \lambda_{\min}(\Sigma_{\tilde{\bm{g}}})\asymp 1$. Since $\max_x \bar{g}(x)=O(1)$, we have $C^{-1}\lambda_{\min}(\Sigma_{\bm{g}})\le \lambda_{\min}(\Sigma_{\tilde{\bm{g}}})$ for some constant $C>0$. This proves the lower bound for the smallest eigenvalue. As for the upper bound for the largest eigenvalue, fix $\epsilon>0$ and let $S_{\epsilon}=\{x:\bar{g}(x)\ge \epsilon\}$ and $U_{\epsilon}=H(S_{\epsilon})\subset[0,1]$. We then immediately have $|U_{\epsilon}^c|=\int_{S_\epsilon^c}\bar{g}(x)dx\rightarrow 0$ as $\epsilon\rightarrow 0$. We decompose:
\begin{align*}
    \Sigma_{\tilde{\bm{g}}}=\Sigma_{\tilde{\bm{g}}}^{(\epsilon)}+R_{\epsilon},
\end{align*}
where $\Sigma_{\tilde{\bm{g}}}^{(\epsilon)}=K^{-1}\int_{U_\epsilon}\bm{\tilde{g}}(u)\bm{\tilde{g}}(u)'du$ and $R_\epsilon=K^{-1}\int_{U_\epsilon^c}\bm{\tilde{g}}(u)\bm{\tilde{g}}(u)'du$. As a result, we have $\lambda_{\max}(\Sigma_{\tilde{\bm{g}}})\le \lambda_{\max}(\Sigma_{\tilde{\bm{g}}}^{(\epsilon)})+\lambda_{\max}(R_\epsilon)$. First, on $S_\epsilon$, it holds that
\begin{align*}
    \lambda_{\max}(\Sigma_{\tilde{\bm{g}}}^{(\epsilon)})\le \epsilon^{-1}\lambda_{\max}(\Sigma_{\bm{g}}).
\end{align*}
Moreover, on $S_\epsilon^c$, we can trivially bound $\lambda_{\max}(R_\epsilon)=\|R_\epsilon\|\le K^{-1}|U_\epsilon^c|\|\tilde{\bm{g}}(u)\|^2\le C|U_\epsilon^c|\rightarrow 0$ as $\epsilon\rightarrow 0$ for some constant $C>0$. Hence, by picking $\epsilon$ properly large, we obtain $\lambda_{\max}(\Sigma_{\tilde{\bm{g}}})=O(1)$. Putting all bounds together, we proved the first statement in (c). For the rest of the bounds for $\|G\|$ and $\|Q\|$, the bound on $\|G\|$ is a direct consequence of the bound on eigenvalues of $\Sigma_G$. As for $Q$, it is seen that $\|Q\|=\|G(G'G)^{-1}\|\le \|G\|\|(G'G)^{-1}\|=O((K/M)^{1/2})\cdot O(M/K)=O((M/K)^{1/2})$.

Finally, for (b), note that for all $1\le m\le M$,
\[
\sum_{i=1}^n f_i^{\text{hist}}(m)  = \sum_{k=1}^K g_k^{\text{hist}}(m)\Bigl[ \sum_{i=1}^n \pi_i(k)\Bigr]\leq \underbrace{\biggl[\sum_{k=1}^K g_k^{\text{hist}}(m)
\biggr]}_{\leq c_4M^{-1}K}\underbrace{\max_{k}\biggl[ \sum_{i=1}^n \pi_i(k)\biggr]}_{\leq c_2K^{-1}n} =O(M^{-1}n),
\]
where we have used Assumption \ref{assump2} and Lemma \ref{lem:key} (c) in the inequality. The lower bound is similar.

\subsection{Proof of Lemma~\ref{lem:error1}}
Let $B=\bm{\Sigma}^{-1} \Pi'$, where $\bm{\Sigma}=n^{-1}\Pi'\Pi$.
Combining these notations with 
\eqref{g-star-expression}, we have $\bg^*(x)=(Nn)^{-1}B\mathbb{E}[K(x)]{\bf 1}_N$. 
Therefore, for each $1\leq k\leq K$, 
\[
g^*_k(x)=\frac{1}{nN}\sum_{i=1}^n b_i(k)\sum_{j=1}^N\mbb{E}[\mcal{K}_h(x-X_{ij})]=\frac{1}{n}\sum_{i=1}^n b_i(k)\mbb{E}[\mcal{K}_h(x-X_{i1})],
\]
where the last equality is because $\{X_{ij}\}_{j=1}^N$ are i.i.d. variables. 
Recall that  $f_i(x)$ is the density of $X_{i1}$; and under our model, $f_i(x)=\sum_{\ell=1}^K \pi_i(\ell)g_\ell(x)$. It follows that 
\[
g^*_k(x) = \frac{1}{n}\sum_{i=1}^nb_i(k)\int \mcal{K}_h(x-y)f_i(y)dy
=\sum_{\ell=1}^K\left(\frac{1}{n}\sum_{i=1}^nb_i(k)\pi_i(\ell)\right)\int\mcal{K}_h(x-y)g_\ell(y)dy.
\]
Since $n^{-1}B\Pi=n^{-1}(n^{-1}\Pi'\Pi)^{-1}\Pi'\Pi=I$, we have $\frac{1}{n}\sum_{i=1}^nb_i(k)\pi_i(\ell)$ is equal to 1 if $k=\ell$ and 0 otherwise. 
It follows that 
\beq \label{proof-bias-1}
g_k^*(x)=\int \mcal{K}_h(x-y)g_k(y)dy. 
\eeq

We now use \eqref{proof-bias-1} to bound the difference between $g_k^*(x)$ and $g_k(x)$. Recall that ${\cal K}_h(x) = \frac{1}{h}{\cal K}(\frac{x}{h})$. 
For $1\le k\le K$, it is seen that
\begin{align*}
    e_k'({\boldsymbol g}^*(x)-{\boldsymbol g}(x))=g_k^*(x)-g_k(x)
\end{align*}
is actually the bias of the KDE of $g_k$ with bandwidth $h$. Therefore, it is bounding the mean integrated squared error (MISE). The proof is relatively standard. See \cite[Proposition 5.1]{tsybakov2009introduction}. For completeness, we include it in the following. Since we assume all $g_k$'s belong to the Nikol’ski class with smoothness $\beta>0$, using Taylor expansion (with integral form of the remainder, see \cite{apostol1991calculus}), it holds that
\begin{align*}
    g_k(uh+x)=g(x)+g'(x)uh+\ldots+\frac{(uh)^{\ell}}{(\ell-1)!}\int_{0}^{1}(1-\tau)^{\ell-1}g^{(\ell)}(x+\tau uh)d\tau,
\end{align*}
where $\ell=\lfloor \beta \rfloor$. Then, using the assumption that the kernel $\mcal{K}$ is of order $\ell$ (Assumption \ref{assump2} (c)) that $\int u^j\mcal{K}(u)du=0$ for $j=1,\ldots,\ell$, we have:
\begin{align*}
    g_k^*(x)-g_k(x)&=\int \mcal{K}(u)(g_k(uh+x)-g_k(x))du\\
    &=\int \mcal{K}(u)\frac{(uh)^{\ell}}{(\ell-1)!}\int_{0}^{1}(1-\tau)^{\ell-1}g^{(\ell)}(x+\tau uh)d\tau du\\
    &=\int \mcal{K}(u)\frac{(uh)^{\ell}}{(\ell-1)!}\int_{0}^{1}(1-\tau)^{\ell-1}(g^{(\ell)}(x+\tau uh)-g^{(\ell)}(x))d\tau du.
\end{align*}
Then, using Minkowski's integral inequality in Lemma \ref{lem:Minkowski} twice and the assumption that the kernel $\mcal{K}$ is of order $\ell$ (Assumption \ref{assump2} (c)) that $\int |u|^\beta|\mcal{K}(u)|du<\infty$, we can further obtain the bound:
\begin{align*}
    \int (g_k^*(x)-g_k(x))^2dx&\le \int \Bigg(\int |\mcal{K}(u)|\frac{(uh)^{\ell}}{(\ell-1)!}\times\\
    &\qquad\int_{0}^{1}(1-\tau)^{\ell-1}|g^{(\ell)}(x+\tau uh)-g^{(\ell)}(x)|d\tau du\Bigg)^2 dx\\
    &\le \Bigg(\int |\mcal{K}(u)|\frac{(uh)^{\ell}}{(\ell-1)!}\times\\
    &\qquad\left[\int \left(\int_{0}^{1}(1-\tau)^{\ell-1}|g^{(\ell)}(x+\tau uh)-g^{(\ell)}(x)|d\tau\right)^2dx\right]^{\frac{1}{2}}du\Bigg)^2\\
    &\le \left(\int |\mcal{K}(u)| \frac{(uh)^{\ell}}{(\ell-1)!}\left[\int_{0}^{1}(1-\tau)^{\ell-1}L|uh|^{\beta-\ell}d\tau\right]du\right)^2\\
    &\le \left(\frac{L}{\ell!}\int |u|^\beta|\mcal{K}(u)|du\right)^2h^{2\beta},
\end{align*}
where $\int |u|^\beta|\mcal{K}(u)|du<\infty$. Taking the sum over all $1\le k\le K$, it then implies the desired bound.

\subsection{Proof of Lemma~\ref{lem:error2-decompose}}
Using the definitions in \eqref{all-estimators}, we immediately have:
\begin{align} \label{proof-erro2decomp}
\widehat{\boldsymbol g}^*&(x)-{\boldsymbol g}^*(x) 
= G'G(G'TG)^{-1}G'S{\bf 1}_n - G'G(G'\mathbb{E}TG)^{-1}G'\mathbb{E}S{\bf 1}_n \cr
&= G'G(G'\mbb{E}TG)^{-1}G'(S-\mathbb{E}S){\bf 1}_n + G'G\bigl[(G'TG)^{-1}-(G'\mathbb{E}T G)^{-1}\bigr]G'S{\bf 1}_n\cr
&= G'G(G'\mbb{E}TG)^{-1}\Bigl[ G'(S-\mathbb{E}S){\bf 1}_n + (G'\mathbb{E}T G -G'TG)(G'TG)^{-1}G'S{\bf 1}_n\Bigr]\cr
&= (\Pi'\Pi)^{-1}\bigl[ Q'(S-\mathbb{E}S){\bf 1}_n + Q'(\mathbb{E}T-T)Q \, \widehat{\boldsymbol g}^*(x)\bigr]\cr
&= \Omega'(S-\mathbb{E}S){\bf 1}_n - \Omega'(T-\mathbb{E}T)Q\,\widehat{\bg}^*(x), 
\end{align}
where in the third line we have used the equality of $A^{-1}-B^{-1}=B^{-1}(B-A)A^{-1}$ (for any symmetric matrices $A$ and $B$), in the fourth line we have plugged in the expression of $\mathbb{E}[T]$ in Lemma \ref{lem:key} (a), in the fifth line we have used  $Q=G(G'G)^{-1}$ and the definition of $\widehat{\boldsymbol g}^*(x)$, and the last line is due to the definition of $\Omega=Q(\Pi'\Pi)^{-1}$. 

Note that \eqref{proof-erro2decomp} is a linear equation on $\widehat{\bg}^*(x)$. 
Solving this equation gives the claim. \qed

\subsection{Proof of Lemma \ref{lem:uprocess-bound}}\label{sec:pf-lem:uprocess-bound}
We decompose $H_{n,N}(x)$ as follows: 
\[
 H_{n,N}(x)=\frac{1}{n}\sum_{i=1}^n H_i(x), \qquad        H_i(x):=\frac{1}{N(N-1)}\sum_{1\le j\neq j'\le N}h(x;x_{ij},x_{ij'}), 
\]
where $H_i(x)$'s are independent across $i$. 
Meanwhile, $\theta(x)=\frac{1}{n}\sum_{i=1}^n \theta_i(x)$. 
Hence, using Fubini's theorem,
    \begin{align} \label{proof-intvarU-1}
\mbb{E}\biggl[&\int (H_{n,N}(x)-\theta(x))^2dx\biggr]=\int \mbb{E}\Bigl(\frac{1}{n}\sum_{i=1}^{n}(H_i(x)-\theta_i(x))\Bigr)^2dx\cr
&= \frac{1}{n^2}\sum_{1\leq i, i'\leq n}\int \mathbb{E}\bigl[(H_i(x)-\theta_i(x))(H_{i'}(x)-\theta_{i'}(x))\bigr]dx\cr
&= \frac{1}{n^2}\sum_{i=1}^n\int \mathbb{E}\bigl[ (H_i(x)-\theta_i(x))^2]dx,  
    \end{align}
where 
in the third line we have used the independence across $H_i(x)$ and the fact that $\mathbb{E}[H_i(x)]=\theta_i(x)$.

It remains to bound $\int\mathbb{E}[(H_i(x)-\theta_i(x))^2]dx$. 
Let $\bar{h}_{i,j,j'}(x):=h(x;x_{ij},x_{ij'})-\theta_i(x)$. Then, $H_i(x)-\theta_i(x)=\frac{1}{N(N-1)}\sum_{1\le j\neq j'\le N}\bar{h}_{i,j,j'}(x)$. 
We immediately have
\beq  \label{proof-intvarU-2}
  \int \mathbb{E}\bigl[ (H_i(x)-\theta_i(x))^2]dx =\frac{1}{N^2(N-1)^2}\sum_{j\neq j'}\sum_{l\neq l'}\int \mbb{E}\bigl[\bar{h}_{i,j,j'}(x)\bar{h}_{i,l,l'}(x)\bigr]dx.
\eeq
    Now, we divide the double sum into three parts.

    \textbf{Case 1.} The disjoint pairs are $\{j,j'\}\cap\{l,l'\}=\emptyset$. Due to independence, such terms are zero.
    
    \textbf{Case 2.} The diagonal terms correspond to the pairs $(j,j')=(l,l')$. There are $N(N-1)$ terms and each of them satisfies (using Fubini's theorem)
    \begin{align*}
        \int \mbb{E}[\bar{h}_{i,j,j'}(x)\bar{h}_{i,l,l'}(x)]dx=\mbb{E}\left[\int \bar{h}_{i,j,j'}(x)\bar{h}_{i,l,l'}(x)dx\right]=\sigma_i^2.
    \end{align*}
    
    \textbf{Case 3.} The overlapping pairs correspond to the case $|\{j,j'\}\cap\{l,l'\}|=1$, where $\bar{h}_{i,j,j'}(x)$ and $\bar{h}_{i,l,l'}(x)$ are still dependent. There are at most $4N(N-1)(N-2)$ such pairs with each satisfying
    \begin{align*}
        \left|\int \mbb{E}[\bar{h}_{i,j,j'}(x)\bar{h}_{i,l,l'}(x)]dx\right|&\le \left(\int \mbb{E}[\bar{h}^2_{i,j,j'}(x)]dx\right)^{\frac{1}{2}}\left(\int \mbb{E}[\bar{h}^2_{i,l,l'}(x)]dx\right)^{\frac{1}{2}}\le 4b_i^2,
    \end{align*}
    where we used Cauchy–Schwarz inequality and the assumption $\int \mbb{E}[h^2(x;x_{i1},x_{i2})]dx=\mbb{E}[\int h^2(x;x_{i1},x_{i2})dx]\le b_i^2$ for $1\le i\le n$ by Fubini's theorem.

    Finally, combing all cases, there exists a constant $C>0$ such that
    \begin{align*}
        \mbb{E}\left[\int (H_i(x)-\theta_i(x))^2dx\right]&\le \frac{1}{N^2(N-1)^2}\left(N(N-1)\sigma_i^2+4N(N-1)(N-2)b_i^2\right)\\
        &\le \frac{4(\sigma_i^2+b_i^2)}{N}.
    \end{align*}
    Plugging it in \eqref{proof-intvarU-1}, it yields that
    \begin{align*}
        \mbb{E}\left[\int (H_{n,N}(x)-\theta(x))^2dx\right]\le \frac{4(\sigma^2+b^2)}{nN}.
    \end{align*}
    This is the desired integrated variance bound.

\subsection{Proof of Lemma~\ref{lem:error21}}
Recall that in \eqref{error21}, it is seen that for each $1\le k\le K$, $n^{-1}e_k'\Omega'(S -\mathbb{E}[S]){\bf 1}_n$ can be written as a (centered) U-process in \eqref{incomplete-U-process} indexed by $x$ with
\begin{align*}
    h(x;X_{ij},X_{ij'})=\sum_{m=1}^M \Omega_{mk} {\cal K}_h(x-X_{ij})1\{X_{ij'}\in\mcal{B}_m\}.
\end{align*}
Hence, for each $1\le k\le K$, we shall aply Lemma \ref{lem:uprocess-bound}. To this end, in the following, we will bound the quantities $b^2$ and $\sigma^2$ given in Lemma \ref{lem:uprocess-bound}.

First, we focus on $b^2$. For each $1\le k\le K$ and $1\le i\le n$, 
    \begin{align*}
         b_i^2&=\mbb{E}\left[\int h^2(x;x_{i1},x_{i2})dx\right]=\mbb{E}\left[\int\mcal{K}^2_h(x-x_{i1})dx\right]\cdot \mbb{E}\left[\left(\sum_{m=1}^M\Omega_{mk}1\{x_{i2}\in\mcal{B}_m\}\right)^2\right]\\
         &=h^{-1}\int \mcal{K}^2(z)dz\cdot \sum_{m=1}^M f_i^{\text{hist}}(m)\Omega_{mk}^2\le Ch^{-1}\sum_{m=1}^M f_i^{\text{hist}}(m)\Omega_{mk}^2,
    \end{align*}
where we have used the condition \eqref{cond-kernel}. Using Lemma \ref{lem:key} (b), it then yields that
\begin{align*}
 b^2 = \frac{1}{n}\sum_{i=1}^n b_i^2\le Ch^{-1}M^{-1}\sum_{m=1}^M\Omega_{mk}^2.
\end{align*}

Moreover, for any $1\le k\le K$ and $1\le i\le n$, note that by Fubini's theorem,
    \begin{align*}
        \sigma_i^2&=\mbb{E}\left[\int (h(x;X_{i1},X_{i2})-\mbb{E}[h(x;X_{i1},X_{i2})])^2dx\right]=\int \Var[h(x;X_{i1},X_{i2})]dx\\
        &\le \int \mbb{E}[{\cal K}^2_h(x-X_{i1})]dx\cdot \mbb{E}\left[\sum_{m=1}^M\gamma_{km} 1\{X_{i2}\in {\cal B}_m\}\right]^2.
    \end{align*}
This is the exactly the same bound in $b_i^2$ above. Using the same argument, we obtain:
\begin{align*}
    \sigma^2=\frac{1}{n}\sum_{i=1}^n \sigma_i^2\le Ch^{-1}M^{-1}\sum_{m=1}^M\Omega_{mk}^2.
\end{align*}

Finally, using the bound in Lemma \ref{lem:uprocess-bound} and summing over $k=1,\ldots,K$, it yields that
\begin{align*}
    \mbb{E}\left[\int \bigl\|n^{-1}\Omega'[S(x)-\mathbb{E}S(x)]{\bf 1}_n\bigr\|^2dx\right]\le C(Nnh)^{-1}M^{-1}\sum_{k=1}^K\sum_{m=1}^M\Omega_{mk}^2.
\end{align*}
Note that the definition of the operator norm is $\|\Omega\|=\sup_{\|v\|=1}\|\Omega v\|$. Taking a special case $v=e_k$ for $1\le k\le K$, it holds that $\sum_{m=1}^M \Omega_{mk}^2\le \|\Omega\|^2$ for $1\le k\le K$. Thus, we conclude
\begin{align*}
    \mbb{E}\left[\int \bigl\|n^{-1}\Omega'[S(x)-\mathbb{E}S(x)]{\bf 1}_n\bigr\|^2dx\right]\le C(Nnh)^{-1}M^{-1}K\|\Omega\|^2.
\end{align*}
Furthermore, using Lemma \ref{lem:key} (c), we have:
\begin{align*}
    \|\Omega\|^2\le \|Q\|^2\|(\Pi'\Pi)^{-1}\|^2= n^{-2}\|Q\|^2\|\bm{\Sigma}^{-1}\|^2\le Cn^{-2}\frac{M}{K}K^2=Cn^{-2}KM.    
\end{align*}
Therefore, we obtain:
\begin{align*}
    \mbb{E}\left[\int \bigl\|\Omega'[S(x)-\mathbb{E}S(x)]{\bf 1}_n\bigr\|^2dx\right]\le C(Nnh)^{-1}K^2.
\end{align*}

\subsection{Proof of Lemma \ref{lem:ustats-bound}}
    Let $k=\lceil \frac{N}{2} \rceil$, the smallest integer not less than $\frac{N}{2}$. For each $1\le i\le n$, set
    \begin{align*}
        W_i(x_{i1},\ldots,x_{iN}):=\frac{h(x_{i1},x_{i2})+h(x_{i3},x_{i4})+\ldots+h(x_{i,2k-1},x_{i,2k})}{k},
    \end{align*}
    where we simply set $h(x_{i,2k-1},x_{i,2k})=0$ if $2k>N$. Here, for each $1\le i\le n$, we break our sample into $k$ non-overlapping blocks of size 2. We then have:
    \begin{align*}
        k\sum_{\sigma=(\sigma(1),\ldots,\sigma(N))\in S^N}W_i(x_{i\sigma(1)},\ldots,x_{i\sigma(N)})=k(N-2)!\sum_{1\le j\neq j'\le N}h(x_{ij},x_{ij'}),
    \end{align*}
    where $S^N$ is the permutation group on $\{1,\ldots,N\}$. This implies:
    \begin{align*}
        H_{n,N}=\frac{1}{nN!}\sum_{i=1}^{n}\sum_{\sigma=(\sigma_1,\ldots,\sigma_n)\in S^N}W_i(x_{i\sigma(1)},\ldots,x_{i\sigma(N)}).
    \end{align*}
    Let $T_{i,\sigma}:=W_i(x_{i\sigma(1)},\ldots,x_{i\sigma(N)})-\theta_i$. Note that for each $i$, $W_i(x_{i\sigma(1)},\ldots,x_{i\sigma(N)})$ is an average of $k$ i.i.d. random variables with mean $\theta_i$ such that $T_{i,\sigma}$ is mean zero. We then obtain:
    \begin{align*}
        H_{n,N}-\theta=:\frac{1}{n}\sum_{i=1}^{n}H_{i},
    \end{align*}
    where $H_i$ is given by
    \begin{align*}
        H_i=\frac{1}{N!}\sum_{\sigma=(\sigma_1,\ldots,\sigma_n)\in S^N}T_{i,\sigma}.
    \end{align*}
    Now, for $t>0$, we apply the Chernoff technique: for any $\lambda>0$,
    \begin{align}\label{chernoff}
        \mbb{P}(H_{n,N}-\theta>t)&=\mbb{P}\left(\sum_{i=1}^{n}H_i>nt\right)\le e^{-\lambda nt}\mbb{E}e^{\lambda\sum_{i=1}^{n}H_i}\nonumber\\
        &=e^{-\lambda nt}\prod_{i=1}^{n}\mbb{E}e^{\lambda H_i}\le e^{-\lambda nt}\prod_{i=1}^{n}\left(\frac{1}{N!}\sum_{\sigma\in S^N}\mbb{E}e^{\lambda T_{i,\sigma}}\right),
    \end{align}
    where the last step is from Jensen's inequality. Here, note that according to the definition of $T_{i,\sigma}$, it is an average of $k$ i.i.d. random variables. By independence and Lemma \ref{lem:mgfbound}, we have:
    \begin{align*}
        \mbb{E}e^{\lambda T_{i,\sigma}}=\prod_{j=1}^{k}\mbb{E}e^{\lambda \frac{1}{k}(h(x_{i,2j-1},x_{i,2j})-\theta_i)}\le\text{exp}\left(\lambda^2\tilde{\sigma}_i^2\frac{e^{2\tilde{b}\lambda}-1-2\tilde{b}\lambda}{(2\tilde{b}\lambda)^2}\right),
    \end{align*}
    where $\tilde{b}=k^{-1}b$ and $\tilde{\sigma}_i^2=k^{-1}\sigma_i^2$. Plugging this back in \eqref{chernoff}, it yields that
    \begin{align*}
        \mbb{P}(H_{n,N}-\theta>t)&\le e^{-\lambda nt}\prod_{i=1}^{n}\text{exp}\left(\lambda^2\tilde{\sigma}_i^2\frac{e^{2\tilde{b}\lambda}-1-2\tilde{b}\lambda}{(2\tilde{b}\lambda)^2}\right)\\
        &=e^{-\lambda nt}\text{exp}\left(\lambda^2n\tilde{\sigma}^2\frac{e^{2\tilde{b}\lambda}-1-2\tilde{b}\lambda}{(2\tilde{b}\lambda)^2}\right),
    \end{align*}
    where $\tilde{\sigma}^2=n^{-1}\sum_{i=1}^{n}\tilde{\sigma}_i^2=n^{-1}\sum_{i=1}^{n}k^{-1}\sigma_i^2$. Take $\lambda=(2\tilde{b})^{-1}\log\left(1+2\tilde{b}t\tilde{\sigma}^{-2}\right)$ to obtain
    \begin{align*}
         \mbb{P}(H_{n,N}-\theta>t)\le\text{exp}\left(-\frac{n\tilde{\sigma}^2}{(2\tilde{b})^2}\phi\left(\frac{2\tilde{b}t}{\tilde{\sigma}^2}\right)\right),
    \end{align*}
    where $\phi(x)=(1+x)\log(1+x)-x$ and it holds that $\phi(x)\ge x^2/(2+2x/3)$ for all $x\ge 0$. Then, it implies
    \begin{align*}
        \mbb{P}(H_{n,N}-\theta>t)\le\text{exp}\left(-\frac{nt^2}{2\tilde{\sigma}^2+4\tilde{b}t/3}\right)\le \text{exp}\left(-\frac{1}{2}\min\left\{\frac{nt^2}{2\tilde{\sigma}^2},\frac{3nt^2}{4\tilde{b}t}\right\}\right).
    \end{align*}
    Repeating for the other side gives us:
    \begin{align}\label{twosidebound}
        \mbb{P}(|H_{n,N}-\theta|>t)\le 2\text{exp}\left(-\frac{nt^2}{2\tilde{\sigma}^2+4\tilde{b}t/3}\right)\le 2\text{exp}\left(-\frac{1}{2}\min\left\{\frac{nt^2}{2\tilde{\sigma}^2},\frac{3nt}{4\tilde{b}}\right\}\right).
    \end{align}
    Setting $$t=\max\left\{2\tilde{\sigma}\sqrt{\frac{\log(2/\delta)}{n}},\frac{8\tilde{b}\log(2/\delta)}{3n}\right\}=\max\left\{2\sigma\sqrt{\frac{\log(2/\delta)}{nk}},\frac{8b\log(2/\delta)}{3nk}\right\},$$ 
    and noting that $k\ge N/2$, we obtain the first concentration inequality.

\subsection{Proof of Lemma~\ref{lem:error22}}\label{sec:pf-lem:error22}

We present a technical lemma as follows: 
\begin{lemma} \label{lem:spectralnorms}
Under the conditions of Theorem~\ref{thm:main}, 
there exists a constant $C_3>0$ such that with probability $1-o((Nn)^{-5})$ simultaneously, 
\begin{align} 
   &\|T-\mbb{E}[T]\|\le C_3\left(\sqrt{\frac{nK(\log (Nn)+M)}{NM^2}}+\frac{\log (Nn)+M}{N}\right), \label{spec-norms-1}\\
    &\|Q'(T-\mbb{E}[T])Q\|\le C_3(NK)^{-1/2}\sqrt{n(\log(Nn)+K)}. \label{spec-norms-2}
\end{align}
Furthermore, it holds that $\mbb{P}(\mcal{E})\ge 1-o((Nn)^{-5})$.
\end{lemma}

Given Lemma~\ref{lem:spectralnorms}, Lemma~\ref{lem:error22} follows immediately by inserting \eqref{spec-norms-2} into \eqref{error22-temp1}.

\bigskip

We now prove Lemma~\ref{lem:spectralnorms}. Recall that
\begin{align*}
    T-\mathbb{E}[T] 
= \frac{1}{N(N-1)}\sum_{i=1}^n \sum_{1\leq j\neq j'\leq N}\bigl(U_{ij}U_{ij'}'  - \mbb{E}[U_{ij}]\mbb{E}[U_{ij}']\bigr).
\end{align*}
We use the tool in \cite{vershynin2010introduction} to bound the spectral norm. Fix an $\epsilon$-net, ${\cal N}_\epsilon$, of the unit sphere of $\mathbb{R}^M$. By \cite[Lemma 5.4]{vershynin2010introduction}, $\|T-\mathbb{E}[T]\|\leq (1-2\epsilon)^{-1}\max_{v\in{\cal N}_\epsilon}|v'(T-\mathbb{E}T)v|=2\max_{v\in{\cal N}_\epsilon}|v'(T-\mathbb{E}T)v|$ if we set $\epsilon=1/4$. Then, note that
\begin{align*}
    n^{-1}v'(T-\mathbb{E}T)v=\frac{1}{nN(N-1)}\sum_{i=1}^n \sum_{1\leq j\neq j'\leq N}\left[(v'U_{ij})(v'U_{ij'})-(v'\mbb{E}[U_{ij}])(v'\mbb{E}[U_{ij'}])\right]
\end{align*}
can be viewed as a U-statistic in \eqref{eq:u-stat} with
\begin{align*}
    h(X_{ij},X_{ij'})=(v'U_{ij})(v'U_{ij'}).
\end{align*}
In order to apply Lemma \ref{lem:ustats-bound}, we will bound $b$ and $\sigma$ in Lemma \ref{lem:ustats-bound} in the following.

First, it is seen that $|v'U_{ij}|\le \|v\|\|U_{ij}\|$. Here, $\|v\|=1$ and $\|U_{ij}\|^2=\sum_{m=1}^M U_{ijm}^2=1$. It yields that
\begin{align*}
    |v'U_{ij}||v'U_{ij'}|\le \|v\|^2\|U_{ij}\|\|U_{ij'}\|\le 1=:b
\end{align*}

Next, for $1\le i\le n$, due to the independence of $U_{ij}$ and $U_{ij'}$ for $j\ne j'$, we have:
\begin{align*}
    \Var\left[(v'U_{ij})(v'U_{ij'})\right]\le \mbb{E}\left[(v'U_{ij})^2\right]\mbb{E}\left[(v'U_{ij'})^2\right].
\end{align*}
Then, it is seen that
\begin{align*}
    \mbb{E}\left[(v'U_{ij})^2\right]&=\mbb{E}\left[\sum_{m=1}^M v_mU_{ijm}\right]^2=\sum_{m=1}^M v_{m}^2f_i^{\text{hist}}(m).
\end{align*}
The same bound also holds for $j'$ similarly. Thus, we obtain:
\begin{align*}
    \sum_{i=1}^n\Var\left[(v'U_{ij})(v'U_{ij'})\right]\le  \sum_{m=1}^M \sum_{m'=1}^Mv_{m}^2v_{m'}^2\sum_{i=1}^nf_i^{\text{hist}}(m) f_i^{\text{hist}}(m').
\end{align*}
We'd like to bound the entrywise maximum norm of the $M\times M$ matrix $F^{\text{hist}}(F^{\text{hist}})'$. Lemma \ref{lem:key} (c) implies that $\|G'e_m\|_1=O(K/M)$ for each $1\leq m\leq M$. As a result, for any $1\leq m,s\leq M$, 
\begin{align*}
e_m'F^{\text{hist}}(F^{\text{hist}})'e_{s} = e_m'G(\Pi'\Pi) G'e_{s}& \leq \|\Pi'\Pi\|_{\max} \|G' e_m\|_1 \|G'e_s\|_1 \cr
&= O(n/K)\cdot O(K^2/M^2)=  O(nKM^{-2}).  
\end{align*}
Since $\|Q\|^2=O(M/K)$ by Lemma \ref{lem:key} (c), we immediately have 
\[
\sum_{i=1}^n\Var\left[(v'U_{ij})(v'U_{ij'})\right]\leq \|v\|^4\cdot \|(F^{\text{hist}})'F^{\text{hist}}\|_{\max}= O(nKM^{-2}). 
\]
This then implies $\sigma^2=n^{-1}\sum_{i=1}^n\Var\left[(v'U_{ij})(v'U_{ij'})\right]\le CKM^{-2}$. Therefore, using Lemma \ref{lem:ustats-bound}, it holds that 
\begin{align*}
    \mbb{P}\left(|n^{-1}v'(T-\mathbb{E}T)v|\ge 2\sigma\sqrt{\frac{2\log(2/\delta)}{nN}}+\frac{16b\log(2/\delta)}{3nN}\right)< \delta.
\end{align*}
Taking the union bound over all vectors in $\mcal{N}_{\epsilon}$, we have:
\begin{align*}
    \mbb{P}\left(\max_{v\in\mcal{N}_\epsilon}|n^{-1}v'(T-\mathbb{E}T)v|\ge 2\sigma\sqrt{\frac{2\log(2/\delta)}{nN}}+\frac{16b\log(2/\delta)}{3nN}\right)< |\mcal{N}_{\epsilon}|\delta.
\end{align*}
Here, since we set $\epsilon=1/4$, according to \cite[Lemma 5.2]{vershynin2010introduction}, $|\mcal{N}_{1/4}|\le 9^M$. It then yields that
\begin{align*}
    \mbb{P}\left(\max_{v\in\mcal{N}_\epsilon}|n^{-1}v'(T-\mathbb{E}T)v|\ge 2\sigma\sqrt{\frac{2\log(2/\delta)}{nN}}+\frac{16b\log(2/\delta)}{3nN}\right)< 9^M\delta.
\end{align*}
Setting $9^M\delta=o((Nn)^{-5})$ and plugging in the bounds of $b$ and $\sigma^2$ above, we then have: with probability $1-o((Nn)^{-5})$,
\begin{align*}
    &\|n^{-1}(T-\mbb{E}[T])\|\le 2\max_{v\in{\cal N}_\epsilon}|n^{-1}v'(T-\mathbb{E}T)v|\le 4\sigma\sqrt{\frac{2\log(2/\delta)}{nN}}+\frac{32b\log(2/\delta)}{3nN}\\
    &\le C\left(\sqrt{\frac{K(\log (Nn)+M)}{NnM^2}}+\frac{\log (Nn)+M}{nN}\right).
\end{align*}
This shows that
\begin{align*}
    \|T-\mbb{E}[T]\|\le C\left(\sqrt{\frac{nK(\log (Nn)+M)}{NM^2}}+\frac{\log (Nn)+M}{N}\right).
\end{align*}
Moreover, under the assumption that $K\leq M\leq [Nn/\log^2(Nn)]^{1/2}$, it implies:
\begin{align}\label{eventE1}
    \|T-\mbb{E}[T]\|\le CM^{-1}n.
\end{align}

The proof for the second bound on $\|\Omega'(T-\mbb{E}[T])Q\|$ is similar. Let $V_{ij}=Q'U_{ij}$. Recall that
\begin{align*}
    Q'(T-\mathbb{E}[T])Q 
= \frac{1}{N(N-1)}\sum_{i=1}^n \sum_{1\leq j\neq j'\leq N}\bigl(V_{ij}V_{ij'}'  - \pi_i\pi_i'\bigr).
\end{align*}
Fix an $\epsilon$-net, ${\cal N}_\epsilon$, of the unit sphere of $\mathbb{R}^K$. By \cite[Lemma 5.4]{vershynin2010introduction}, $\|Q'(T-\mathbb{E}[T])Q\|\leq (1-2\epsilon)^{-1}\max_{v\in{\cal N}_\epsilon}|v'Q'(T-\mathbb{E}T)Qv|=2\max_{v\in{\cal N}_\epsilon}|v'Q'(T-\mathbb{E}T)Qv|$ if we set $\epsilon=1/4$. Then, note that
\begin{align*}
    n^{-1}v'Q'(T-\mathbb{E}T)Qv=\frac{1}{nN(N-1)}\sum_{i=1}^n \sum_{1\leq j\neq j'\leq N}\left[(v'V_{ij})(v'V_{ij'})-(v'\pi_i)(v'\pi_i)\right]
\end{align*}
can be viewed as a U-statistic in \eqref{eq:u-stat} with
\begin{align*}
    h(X_{ij},X_{ij'})=(v'V_{ij})(v'V_{ij'}).
\end{align*}
In order to apply Lemma \ref{lem:ustats-bound}, we will bound $b$ and $\sigma$ in Lemma \ref{lem:ustats-bound} in the following.

First, it is seen that $|v'V_{ij}|\le \|v\|\|V_{ij}\|\le \|v\|\|\|(G'G)^{-1}\|\|G'U_{ij}\|$. Here, $\|v\|=1$ and according to Lemma \ref{lem:key} (c), $\|(G'G)^{-1}\| = O(MK^{-1})$. Moreover, according to Lemma \ref{lem:key} (c), we have: 
\begin{align*}
    \|G'U_{ij}\|\le \|G'U_{ij}\|_1\le \max_{1\le m\le M}\|G'e_m\|_1\le CM^{-1}K.
\end{align*}
Then, it yields that $|v'V_{ij}|\le C $ for some constant $C>0$. Similarly, we also have $|v'\pi_{i}|\le C$. Then, it is seen that we can take $b=C$ for some constant $C>0$.

Next, for $1\le i\le n$, due to the independence of $V_{ij}$ and $V_{ij'}$ for $j\ne j'$, we have:
\begin{align*}
    \Var\left[(v'V_{ij})(v'V_{ij'})\right]\le \mbb{E}\left[(v'V_{ij})^2\right]\mbb{E}\left[(v'V_{ij'})^2\right].
\end{align*}
Let $\tilde{v}=Qv$ such that $v'V_{ij}=\tilde{v}'U_{ij}$. Then, it is seen that
\begin{align*}
    \mbb{E}\left[(v'V_{ij})^2\right]&=\mbb{E}\left[\sum_{m=1}^M\tilde{v}_mU_{ijm}\right]^2=\sum_{m=1}^M\tilde{v}_{m}^2f_i^{\text{hist}}(m).
\end{align*}
The same bound also holds for $j'$ similarly. Thus, we obtain:
\begin{align*}
    \sum_{i=1}^n\Var\left[(v'V_{ij})(v'V_{ij'})\right]\le \sum_{i=1}^n\left(\sum_{m=1}^M\tilde{v}_{m}^2f_i^{\text{hist}}(m)\right)^2=\sum_{m,m'=1}^M\tilde{v}_{m}^2\tilde{v}_{m'}^2\sum_{i=1}^n f_i^{\text{hist}}(m)f_i^{\text{hist}}(m').
\end{align*}
Since $\|Q\|^2=O(M/K)$ by Lemma \ref{lem:key} (c) and we have shown above that $\|(F^{\text{hist}})'F^{\text{hist}}\|_{\max}=O(nKM^{-2})$, we immediately have 
\begin{align*}
\sum_{i=1}^n\Var\left[(v'V_{ij})(v'V_{ij'})\right]&\leq \|Qv\|^4\cdot \|(F^{\text{hist}})'F^{\text{hist}}\|_{\max}\leq O(M^2/K^2)\cdot O(nKM^{-2}) \\
&= O(nK^{-1}). 
\end{align*}
This implies $\sigma^2=n^{-1}\sum_{i=1}^n\Var\left[(v'V_{ij})(v'V_{ij'})\right]\le CK^{-1}$. Therefore, using Lemma \ref{lem:ustats-bound}, it holds that 
\begin{align*}
    \mbb{P}\left(|n^{-1}v'Q'(T-\mathbb{E}T)Qv|\ge 2\sigma\sqrt{\frac{2\log(2/\delta)}{nN}}+\frac{16b\log(2/\delta)}{3nN}\right)< \delta.
\end{align*}
Taking the union bound over all vectors in $\mcal{N}_{\epsilon}$, we have:
\begin{align*}
    \mbb{P}\left(\max_{v\in\mcal{N}_\epsilon}|n^{-1}v'Q'(T-\mathbb{E}T)Qv|\ge 2\sigma\sqrt{\frac{2\log(2/\delta)}{nN}}+\frac{16b\log(2/\delta)}{3nN}\right)< |\mcal{N}_{\epsilon}|\delta.
\end{align*}
Here, since we set $\epsilon=1/4$, according to \cite[Lemma 5.2]{vershynin2010introduction}, $|\mcal{N}_{1/4}|\le 9^K$. It then yields that
\begin{align*}
    \mbb{P}\left(\max_{v\in\mcal{N}_\epsilon}|n^{-1}v'Q'(T-\mathbb{E}T)Qv|\ge 2\sigma\sqrt{\frac{2\log(2/\delta)}{nN}}+\frac{16b\log(2/\delta)}{3nN}\right)< 9^K\delta.
\end{align*}
Setting $9^K\delta=o((Nn)^{-5})$ and plugging in the bounds of $b$ and $\sigma^2$ above, we then have: with probability $1-o((Nn)^{-5})$,
\begin{align*}
    &\|n^{-1}Q'(T-\mbb{E}[T])Q\|\le 2\max_{v\in{\cal N}_\epsilon}|n^{-1}v'Q'(T-\mathbb{E}T)Qv|\le 4\sigma\sqrt{\frac{2\log(2/\delta)}{nN}}+\frac{32b\log(2/\delta)}{3nN}\\
    &\le C\left(\sqrt{\frac{\log (Nn)+K}{KNn}}+\frac{\log (Nn)+K}{nN}\right)\le C\sqrt{\frac{\log (Nn)+K}{NnK}},
\end{align*}
where we have used the assumptions $K\leq M\leq [Nn/\log^2(Nn)]^{1/2}$ and $ (Nn)^{-1}K\ll h\ll \log^{-1}(Nn)$ in the last step. Then we have:
\begin{align}\label{bound-QTQ}
    \|Q'(T-\mbb{E}[T])Q\|\le C\sqrt{\frac{n(\log n+K)}{NK}}.
\end{align}
Finally, since $\Omega'=(\Pi'\Pi)^{-1}Q'$ with $\|(\Pi'\Pi)^{-1}\|=O(n^{-1}K)$ by Assumption \ref{assump2}, we have: $\|\Omega'(T-\mathbb{E}T)Q\|\leq O(n^{-1}K)\cdot \|Q'(T-\mathbb{E}T)Q\|$. This leads to the desired bound with \eqref{bound-QTQ}. Moreover, under the assumption that $K\leq M\leq [Nn/\log^2(Nn)]^{1/2}$, we further have:
\begin{align}\label{eventE2}
    \|Q'(T-\mbb{E}[T])Q\|\leq n/(K\log^{1/2}(Nn)).
\end{align}
Combing \eqref{eventE1} and \eqref{eventE2}, it shows $\mbb{P}(\mcal{E})\ge 1-o((Nn)^{-5})$.

\subsection{Proof of Lemma \ref{lem:error23}}
We divide the error into two parts as follows:
\begin{align*}
    &\mbb{E}\left[\int \|\bm{\widehat{g}}^*(x)-\bm{g}^*(x)\|^2dx\right]\\
    &=\mbb{E}\left[1\{\Delta\}\int \|\bm{\widehat{g}}_k^*(x)-\bm{g}^*(x)\|^2dx\right]+\mbb{E}\left[1\{\Delta^c\}\int \|\bm{\widehat{g}}_k^*(x)-\bm{g}^*(x)\|^2dx\right],
\end{align*}
where $\Delta$ is the event on which Lemma \ref{lem:error22} holds with $\mbb{P}(\Delta)\ge 1-o((Nn)^{-5})$.

We first bound the error on $\Delta$. Recall the expression of $\widehat{\bm{g}}^*(x)$ in \eqref{error2-decompose}:
\begin{align}\label{pf-hatg*-decompose} 
\widehat{\boldsymbol g}^*(x) = \bigl[I_K +\Omega'(T-\mathbb{E}T)Q\bigr]^{-1}\bigl[\bg^*(x) + \Omega'(S-\mathbb{E}S){\bf 1}_n\bigr]. 
\end{align}
We claim the following auxiliary lemma.
\begin{lemma}\label{lem:lem-lem-error23}
    Let $u$ and $\nu$ be two vectors in $\mbb{R}^K$ for $K\ge 1$. Moreover, let $E\in\mbb{R}^{K\times K}$ be a matrix such that $\|E\|<1$ and $e\in\mbb{R}^K$ be a vector. They satisfy the following identity:
    \begin{align*}
        u=(I_K+E)^{-1}(\nu+e).
    \end{align*}
    Then, it holds that
    \begin{align*}
        \|u-\nu\|\le \frac{\|E\|}{1-\|E\|}(\|\nu\|+\|e\|)+\|e\|.
    \end{align*}
\end{lemma}
\noindent
{\it Proof of Lemma~\ref{lem:lem-lem-error23}:} To prove it, using the fact that $A^{-1}-B^{-1}=B^{-1}(B-A)A^{-1}$, we have:
\begin{align*}
    u-\nu &=(I_K+E)^{-1}(\nu+e)-(\nu+e)+e\\
    &=((I_K+E)^{-1}-I_K)(\nu+e)+e\\
    &=E(I_K+E)^{-1}(\nu+e)+e.
\end{align*}
Here, since $\|E\|<1$, the Neumann series $(I_K+E)^{-1}=\sum_{k=0}^\infty(-E)^k$ converges absolutely in operator norm, and it yields that
\begin{align*}
    \|(I_K+E)^{-1}\|\le \sum_{k=0}^\infty\|E\|^k=\frac{1}{1-\|E\|}.
\end{align*}
Thus, we obtain:
\begin{align*}
    \|u-\nu \|\le \frac{\|E\|}{1-\|E\|}(\|\nu\|+\|e\|)+\|e\|.
\end{align*}

Recall \eqref{pf-hatg*-decompose}. We can set $u=u(x)=\bm{\widehat{g}}^*(x)$, $\nu=\nu(x)=\bm{g}^*(x)$, $E=\Omega'(T-\mathbb{E}T)Q$ and $e=e(x)=\Omega'(S-\mathbb{E}S){\bf 1}_n$. According to Lemma \ref{lem:error22}, there exists a constant $C_3>0$ such that
\begin{align*}
    \|E\|\le C_3(Nn)^{-\frac{1}{2}}\sqrt{K(\log (Nn) +K)}=o(1).
\end{align*}
Moreover, Lemma \ref{lem:error1} implies that $\int \|\nu(x)\|^2dx=\int \|\bm{g}^*(x)\|^2dx\le \int \|\bm{g}(x)\|^2dx+o(K)=O(K)$ and it is seen in Lemma \ref{lem:error21} that $$\int\mbb{E}[\|e(x)\|^2]dx\le C(Nnh)^{-1}K^2=O(K).$$

Consequently, with the bounds about $E,\nu(x)$ and $e(x)$ above, using Lemma \ref{lem:lem-lem-error23} and the fact that $(a+b)^2\le 2(a^2+b^2)$, we have: on the event $\Delta$, there exists a constant $C>0$ such that
\begin{align*}
    &\mbb{E}\left[1\{\Delta\}\int \|\bm{\widehat{g}}^*(x)-\bm{g}^*(x)\|^2dx\right]\\
    &\le 2\mbb{E}\left[\left(\frac{\|E\|}{1-\|E\|}\right)^2\int (\|\nu(x)\|+\|e(x)\|)^2dx\right]+2\mbb{E}\left[\int \|e(x)\|^2dx\right]\\
    &\le C\left\{\left[(Nn)^{-\frac{1}{2}}\sqrt{K(\log (Nn) +K)}\right]^2\cdot K+(Nnh)^{-1}K^2\right\}\\
    &\le CK^2\left(\frac{\log (Nn)+K}{Nn}+\frac{1}{Nnh}\right).
\end{align*}

On the event $\Delta^c$, we will directly bound 
\begin{align*}
    \int \|\bm{\widehat{g}}^*(x)-\bm{g}^*(x)\|^2dx\le \int \|\bm{g}^*(x)\|^2dx+\int \|\bm{\widehat{g}}^*(x)\|^2dx.
\end{align*}
Lemma \ref{lem:error1} implies that $\int \|\bm{g}^*(x)\|^2dx\le \int \|\bm{g}(x)\|^2dx+o(K)=O(K)$. Next, note that
$\widehat{\bg}^*(x)=G'G(G'TG)^{-1}G'S{\bf 1}_n$. 

Since we add a small-order perturbation in \eqref{our-estimator2}
, it then implies $\|(G'TG)^{-1}\|\le CK^{-1}M^2n^{-1}\log (Nn)$. Then, we have: almost surely,
\begin{align*}
    &\int \|\bm{\widehat{g}}_k^*(x)\|^2dx\le C\|G'G\|^2\|(G'TG)^{-1}\|^2\|G\|^2\int \|S(x)\mbf{1}_n\|^2dx\\
    &\le C\frac{K^2}{M^2}\frac{M^4\log^2(Nn)}{K^2n^2}\frac{K}{M}\int \|S(x)\mbf{1}_n\|^2dx\le C\frac{KM\log^2(Nn)}{n^2}\int \|S(x)\mbf{1}_n\|^2dx,
\end{align*}
where we have used Lemma \ref{lem:key} (c) in the above bound. It then remains to bound $\int \|S(x)\mbf{1}_n\|^2dx$. We have:
\begin{align*}
    \int \|S(x)\mbf{1}_n\|^2dx&=\int\sum_{m=1}^M \left(\frac{1}{N(N-1)}\sum_{i=1}^n \sum_{j\neq j'}\mcal{K}_h(x-X_{ij})U_{ij'm}\right)^2dx\\
    &\le h^{-1}n\|\mcal{K}\|_{\infty}\frac{1}{N(N-1)}\sum_{i=1}^n \sum_{j\neq j'}\sum_{m=1}^MU_{ij'm}\int|\mcal{K}_h(x-X_{ij})|dx.
\end{align*}
Since $\sum_{m=1}^MU_{ij'm}=1$ and $\int|\mcal{K}_h(x-X_{ij})|dx=\int|\mcal{K}(u)|du<\infty$, we have: almost surely,
\begin{align}\label{intS1n-crude-bound}
    \int \|S(x)\bm{1}_n\|^2dx\le Ch^{-1}n^2.
\end{align}
Then, we obtain: almost surely,
\begin{align}\label{hatgstar-crude-bound}
    \int\|\bm{\widehat{g}}^*(x)\|^2dx\le CKMh^{-1}\log^2 (Nn).
\end{align}
Together, it yields that almost surely,
\begin{align*}
    \int \|\bm{\widehat{g}}^*(x)-\bm{g}^*(x)\|^2dx\le CKMh^{-1}\log^2(Nn).
\end{align*}
Consequently, on the event $\Delta^c$, we have:
\begin{align*}
    \mbb{E}\left[1\{\Delta^c\}\int \|\bm{\widehat{g}}^*(x)-\bm{g}^*(x)\|^2dx\right]\le C\frac{KM\log^2(Nn)}{(Nn)^5h}.
\end{align*}
Finally, combining all bounds above, we conclude there exists a constant $C>0$ such that
\begin{align*}
    \mbb{E}\left[\int \|\bm{\widehat{g}}^*(x)-\bm{g}^*(x)\|^2dx\right]\le CK\left[\frac{K(\log (Nn)+K)}{Nn}+\frac{K}{Nnh}+\frac{M\log^2(Nn)}{(Nn)^5h}\right].
\end{align*}
Using the assumption $K\leq M\leq [Nn/\log^2(Nn)]^{1/2}$ and $ (Nn)^{-1}K\ll h\ll \log^{-1}(Nn)$, we obtain the final bound.

\subsection{Proof of Lemma~\ref{lem:error3-decompose}}
By definitions, $\widehat{\bg}(x)=\widehat{G}'\widehat{G}(\widehat{G}'T\widehat{G})^{-1}\widehat{G}'S{\bf 1}_n$, and $\widehat{\bg}^*(x)=G'G(G'TG)^{-1}G'S{\bf 1}_n$, 
Therefore, we can write $\widehat{\boldsymbol g}(x) - \widehat{\bg}^*(x) =I_1 + I_2 + I_3$, where
\begin{align*}
I_1 & = G'G(G'TG)^{-1}(\widehat{G}-G)'S{\bf 1}_n,\cr
I_2 &= G'G\bigl[(\widehat{G}'T\widehat{G})^{-1} - (G'TG)^{-1}\bigr]\widehat{G}'S{\bf 1}_n,\cr
I_3 &= (\widehat{G}'\widehat{G}- G'G) (\widehat{G}'T\widehat{G})^{-1}\widehat{G}'S{\bf 1}_n. 
\end{align*}
For $I_1$, the definition of $Q=G(G'G)^{-1}$ implies that 
\beq \label{proof-error3decomp-I1}
I_1 =(Q'TQ)^{-1}(G'G)^{-1}(\widehat{G}-G)'S{\bf 1}_n = R(\widehat{G}-G)'S{\bf 1}_n.  
\eeq
For $I_3$, using $\widehat{\bg}(x) = \widehat{G}'\widehat{G}(\widehat{G}'T\widehat{G})^{-1}\widehat{G}'S{\bf 1}_n$ again, we deduce that 
\beq \label{proof-error3decomp-I3}
I_3 = (\widehat{G}'\widehat{G}- G'G)(\widehat{G}'\widehat{G})^{-1} \widehat{\bg}(x)=\Delta_1(G'G+\Delta_1)^{-1}\,\widehat{\bg}(x). 
\eeq
For $I_2$, using the equality that $A^{-1}-B^{-1}=B^{-1}(B-A)A^{-1}$, we have:
\begin{align} \label{proof-error3decomp-I2}
I_2 &= G'G(G'TG)^{-1}(G'TG-\widehat{G}'T\widehat{G})(\widehat{G}'T\widehat{G})^{-1}\widehat{G}'S{\bf 1}_n \cr
& =  G'G(G'TG)^{-1}(G'TG-\widehat{G}'T\widehat{G})(\widehat{G}'\widehat{G})^{-1}\widehat{\bg}(x),\cr
&= (Q'TQ)^{-1}(G'G)^{-1}(G'TG-\widehat{G}'T\widehat{G})(\widehat{G}'\widehat{G})^{-1} \widehat{\bg}(x)\cr
&= - R \Delta_2 (G'G+\Delta_1)^{-1}\,  \widehat{\bg}(x). 
\end{align}
Combining \eqref{proof-error3decomp-I1}-\eqref{proof-error3decomp-I3} gives 
\[
\widehat{\boldsymbol g}(x) - \widehat{\bg}^*(x)  = R(\widehat{G}-G)'S{\bf 1}_n + (\Delta_1-R\Delta_2  )(G'G+\Delta_1)^{-1}\,\widehat{\bg}(x). 
\]
This gives a linear equation on $\widehat{\bg}(x)$. The claim follows from solving this equation.

\subsection{Proof of Lemma \ref{lem:error3-prep}}
We first present an auxiliary lemma as follows.
\begin{lemma} \label{lem:error21*}
Under Assumption~\ref{assump2}, 
there exists a constant $C_2>0$ such that 
\[
\mbb{E}\left[\int \bigl\|[S(x)-\mathbb{E}S(x)]{\bf 1}_n\bigr\|^2dx\right] \le C_2n(Nh)^{-1}. 
\]
\end{lemma}
\noindent
{\it Proof of Lemma~\ref{lem:error21*}:} 
It is seen that for each $1\le m\le M$, $n^{-1}e_m'(S -\mathbb{E}[S]){\bf 1}_n$ can be written as a (centered) U-process in \eqref{incomplete-U-process} indexed by $x$ with
\begin{align*}
    h(x;X_{ij},X_{ij'})={\cal K}_h(x-X_{ij})1\{X_{ij'}\in\mcal{B}_m\}.
\end{align*}
Hence, for each $1\le m\le M$, we shall aply Lemma \ref{lem:uprocess-bound}. To this end, in the following, we will bound the quantities $b^2$ and $\sigma^2$ given in Lemma \ref{lem:uprocess-bound}.

First, we focus on $b^2$. For each $1\le m\le M$ and $1\le i\le n$, 
    \begin{align*}
         b_i^2&=\mbb{E}\left[\int h^2(x;x_{i1},x_{i2})dx\right]=\mbb{E}\left[\int\mcal{K}^2_h(x-x_{i1})dx\right]\cdot \mbb{E}\left[1\{x_{i2}\in\mcal{B}_m\}\right]\\
         &=h^{-1}\int \mcal{K}^2(z)dz\cdot f_i^{\text{hist}}(m)\le Ch^{-1}f_i^{\text{hist}}(m),
    \end{align*}
where we have used the condition of the kernel \eqref{cond-kernel} that $\int \mcal{K}^2(z)dz<\infty$ . With Lemma \ref{lem:key} (b), it then implies:
\begin{align*}
 b^2 = \frac{1}{n}\sum_{i=1}^n b_i^2\le Ch^{-1}\frac{1}{n}\sum_{i=1}^n f_i^{\text{hist}}(m)\le  Ch^{-1}M^{-1}.
\end{align*}

Moreover, for any $1\le m\le M$ and $1\le i\le n$, note that by Fubini's theorem,
    \begin{align*}
        \sigma_i^2&=\mbb{E}\left[\int (h(x;X_{i1},X_{i2})-\mbb{E}[h(x;X_{i1},X_{i2})])^2dx\right]=\int \Var[h(x;X_{i1},X_{i2})]dx\\
        &\le \int \mbb{E}[{\cal K}^2_h(x-X_{i1})]dx\cdot \mbb{E}\left[1\{X_{i2}\in {\cal B}_m\}\right].
    \end{align*}
This is the exactly the same bound in $b_i^2$ above. Using the same argument, we obtain:
\begin{align*}
    \sigma^2=\frac{1}{n}\sum_{i=1}^n \sigma_i^2\le Ch^{-1}M^{-1}.
\end{align*}

Finally, using the bound in Lemma \ref{lem:uprocess-bound} and summing over $m=1,\ldots,M$, it yields that
\begin{align*}
    \mbb{E}\left[\int \bigl\|n^{-1}[S(x)-\mathbb{E}S(x)]{\bf 1}_n\bigr\|^2dx\right]\le C(Nnh)^{-1}.
\end{align*}
Therefore, we obtain:
\begin{align*}
    \mbb{E}\left[\int \bigl\|[S(x)-\mathbb{E}S(x)]{\bf 1}_n\bigr\|^2dx\right]\le Cn(Nh)^{-1}.
\end{align*}

We now focus on the main bounds. For simplicity, let $\Delta G=\widehat{G}-G$. For the first bound, note that
\begin{align*}
    \widehat{G}'\widehat{G}=(G+\Delta G)'(G+\Delta G)=G'G+G'\Delta G+(\Delta G)'G+(\Delta G)'\Delta G.
\end{align*}
Then, we have:
\begin{align*}
    \|\Delta_1\|&=\|\widehat{G}'\widehat{G}-G'G\|\le 2\|G\|\|\Delta G\|+\|\Delta G\|^2\\
    &\le C\|\Delta G\|\left(\sqrt{\frac{K}{M}}+\|\Delta G\|\right)\le C\sqrt{\frac{K}{M}}\|\Delta G\|,
\end{align*}
where we have used Lemma \ref{lem:key} (c) and the assumption $M\delta_n^2=o(K)$.

For the second bound, letting $\Delta T=T-\mbb{E}[T]$, note that
\begin{align*}
    R\Delta_2&=(Q'TQ)^{-1}(G'G)^{-1}(\widehat{G}'T\widehat{G}-G'TG)\\
    &=(Q'\mbb{E}[T]Q+Q'\Delta T Q)^{-1}(G'G)^{-1}\left[(G+\Delta G)'T(G+\Delta G)-G'TG\right]\\
    &=(Q'\mbb{E}[T]Q+Q'\Delta T Q)^{-1}(G'G)^{-1}(G'T\Delta G+(\Delta G)'TG+(\Delta G)'T\Delta G)\\
    &=:(Q'\mbb{E}[T]Q+Q'\Delta T Q)^{-1}(G'G)^{-1}(J_1+J_2+J_3).
\end{align*}
In the following, we analyze $J_1,J_2,J_3$ and $(Q'\mbb{E}[T]Q+Q'\Delta T Q)^{-1}$ respectively.

First, for $J_1$, it is seen that
\begin{align*}
    J_1=G'T\Delta G=G'(\Delta T+\mbb{E}[T])\Delta G = G'\mbb{E}[T]\Delta G+G'\Delta T \Delta G.
\end{align*}
According to Assumption \ref{assump2}, Lemma \ref{lem:key} (a) and (c) as well as the assumption that $\|T-\mbb{E}[T]\|\leq M^{-1}n$, we further have:
\begin{align*}
   \|J_1\|&\le \|G'G\Pi'\Pi G'\|\|\Delta G\|+\|G\|\|\Delta T\|\|\Delta G\|\\
   &\le Cn\left(\sqrt{\frac{K}{M^3}}+\sqrt{\frac{K}{M^3}}\right)\|\Delta G\|=2Cn\sqrt{\frac{K}{M^3}}\|\Delta G\|.
\end{align*}
Similarly, the same bound also holds for $J_2$. Also, it holds that
\begin{align*}
    J_3=(\Delta G)'T\Delta G = (\Delta G)'(\Delta T+\mbb{E}[T])\Delta G = (\Delta G)'\mbb{E}[T]\Delta G+(\Delta G)'\Delta T\Delta G.
\end{align*}
Then, we have:
\begin{align*}
    \|J_3\|&\le \|\Delta G\|^2\left(\|G\Pi'\Pi G'\|+\|\Delta T\|\right)\\
    &\le Cn\left(\frac{1}{M}+\frac{1}{M}\right)\|\Delta G\|^2=2Cn\frac{1}{M}\|\Delta G\|^2.
\end{align*}
Thus, combining all the bounds above, we have:
\begin{align*}
    \|J_1+J_2+J_3\|
    &\le Cn\Bigg(\sqrt{\frac{K}{M^3}}\|\Delta G\|+\frac{1}{M}\|\Delta G\|^2\Bigg)\\
    &\le Cn\sqrt{\frac{K}{M^3}}\|\Delta G\|\left(1+\sqrt{\frac{M}{K}}\|\Delta G\|\right)\le Cn\sqrt{\frac{K}{M^3}}\|\Delta G\|,
\end{align*}
where in the last step, we have used the assumption $M\delta_n^2=o(K)$.

Next, we focus on $(Q'\mbb{E}[T]Q+Q'\Delta T Q)^{-1}$. Note that
\begin{align*}
    (Q'\mbb{E}[T]Q+Q'\Delta T Q)^{-1}=(Q'\mbb{E}[T]Q)^{-1}(I_K+Q'\Delta T Q(Q\mbb{E}[T]Q)^{-1})^{-1}.
\end{align*}
Here, according to Assumption \ref{assump2}, Lemma \ref{lem:key} (a) and (c), and the assumption that $\|Q'(T-\mbb{E}[T])Q\|\leq n/(K\log^{1/2}(Nn))$, it holds that 
\begin{align*}
    \|Q'\Delta T Q(Q'\mbb{E}[T]Q)^{-1}\|&\le n(K\log^{1/2}(Nn))^{-1}\|(Q'G\Pi'\Pi G'Q)^{-1}\|\\
    &\le C\log^{-1/2}(Nn)=o(1).
\end{align*}
Thus, letting $F=Q'\Delta T Q(Q'\mbb{E}Q)^{-1}$, the Neumann series $(I_K+F)^{-1}=\sum_{k=0}^\infty(-F)^k$ converges absolutely in operator norm, and it yields that
\begin{align*}
    \|(I_K+F)^{-1}\|\le \sum_{k=0}^\infty\|F\|^k=\frac{1}{1-\|F\|}.
\end{align*}
Hence, with Assumption \ref{assump2}, Lemma \ref{lem:key} (a) and (c), it holds that
\begin{align*}
    \|(Q'\mbb{E}[T]Q+Q'\Delta T Q)^{-1}\|&=\|(Q'\mbb{E}[T]Q)^{-1}(I_K+Q'\Delta T Q(Q'\mbb{E}[T]Q)^{-1})\|\\
    &\le C\|(Q'G\Pi'\Pi G'Q)^{-1}\|\frac{1}{1-\|F\|}\le C\frac{K}{n}.
\end{align*}

Eventually, putting all bounds above together, we obtain:
\begin{align*}
    \|R\Delta_2\|&\le C\frac{K}{n}\frac{M}{K}n\sqrt{\frac{K}{M^3}}\|\Delta G\|\le C\sqrt{\frac{K}{M}}\|\Delta G\|.
\end{align*}

For the third bound, note that
\begin{align*}
    \|R(\widehat{G}-G)'S{\bf 1}_n\|\le \|R\|\|\Delta G\|\|S\mbf{1}_n\|.
\end{align*}
We have already studied $\|R\|$ above. It is seen that
\begin{align*}
    \|R\|=\|(Q'\mbb{E}[T]Q+Q'\Delta T Q)^{-1}(G'G)^{-1}\|\le C\frac{K}{n}\frac{M}{K}=C\frac{M}{n}.
\end{align*}
It remains to study $\|S\mbf{1}_n\|$. Note that
\begin{align*}
    \mbb{E}\left[\int \|S(x)\mbf{1}_n\|^2dx\right]\le \int \|\mbb{E}[S(x)]\mbf{1}_n\|^2dx+\mbb{E}\left[\int \|(S(x)-\mbb{E}[S(x)])\mbf{1}_n\|^2dx\right].
\end{align*}
According to Lemma \ref{lem:error21*}, we have:
\begin{align*}
    \mbb{E}\left[\int \|(S(x)-\mbb{E}[S(x)])\mbf{1}_n\|^2dx\right]\le Cn(Nh)^{-1}.
\end{align*}
Moreover, note that
\begin{align*}
    \|\mbb{E}[S(x)]\mbf{1}_n\|^2&=\sum_{m=1}^M \left(\frac{1}{N(N-1)}\sum_{i=1}^n \sum_{j\neq j'}\mbb{E}[\mcal{K}_h(x-X_{ij})]\mbb{E}[U_{ij'm}]\right)^2\\
    &=\sum_{m=1}^M \left(\frac{1}{N(N-1)}\sum_{i=1}^n \sum_{j\neq j'}\mbb{E}[\mcal{K}_h(x-X_{i1})]f_i^{\text{hist}}(m)\right)^2\\
    &=\sum_{m=1}^M\left(\sum_{i=1}^n \mbb{E}[\mcal{K}_h(x-X_{i1})]f_i^{\text{hist}}(m)\right)^2.
\end{align*}
Here, it is seen that
\begin{align*}
    \mbb{E}[\mcal{K}_h(x-X_{i1})]=\int h^{-1}\mcal{K}\left(\frac{x-z}{h}\right)f_i(z)dz=\int \mcal{K}(u)f_i(uh+x)du.
\end{align*}
We further have:
\begin{align*}
    &\int \|\mbb{E}[S(x)]\mbf{1}_n\|^2dx\\&= \sum_{m=1}^M\int \left(\int \mcal{K}(u)\sum_{i=1}^nf_i(uh+x)f_i^{\text{hist}}(m)du\right)^2dx\\
    &= \sum_{m=1}^M\int \int \int\mcal{K}(u)\mcal{K}(u')\sum_{i=1}^nf_i(uh+x)f_i^{\text{hist}}(m) \sum_{i'=1}^nf_{i'}(u'h+x)f_{i'}^{\text{hist}}(m)dudu'dx.
\end{align*}
We first take integral with respect to $x$. Note that
\begin{align*}
    \int f_i(uh+x)f_{i'}(uh+x)dx\le \left(\int f_i^2(uh+x)dx\right)^{\frac{1}{2}}\left(\int f_{i'}^2(uh+x)dx\right)^{\frac{1}{2}}=O(1),
\end{align*}
where it holds that $\int f_i^2(uh+x)dx=\int f_i^2(z)dz=O(1)$ since we assume $\int g_k^2(z)dz=O(1)$ and $f_i(z)=\sum_{k=1}^K \pi_i(k)g_k(z)$. Then, using the kernel condition \eqref{cond-kernel} and Lemma \ref{lem:key} (b), we have:
\begin{align*}
    \int \|\mbb{E}[S(x)]\mbf{1}_n\|^2dx
    &\le \sum_{m=1}^M\int \mcal{K}(u)du\int \mcal{K}(u')du'\sum_{i=1}^n f_i^{\text{hist}}(m)\sum_{i'=1}^n f_{i'}^{\text{hist}}(m)\\
    &\le Cn^2M^{-1}.
\end{align*}
Consequently, it holds that
\begin{align*}
    \mbb{E}\left[\int \|S(x)\mbf{1}_n\|^2dx\right]\le Cn^2\left((Nnh)^{-1}+M^{-1}\right)\le Cn^2M^{-1}.
\end{align*}
Putting the above bounds together, we have:
\begin{align*}
    \mbb{E}\left[\int \|R(\widehat{G}-G)'S(x){\bf 1}_n\|^2dx\right]\le C\frac{M^2}{n^2}\|\Delta G\|^2n^2M^{-1}\le CM\|\Delta G\|^2.
\end{align*}

\subsection{Proof of Lemma \ref{lem:error3}}
Recall that in Lemma \ref{lem:error3-prep}, we let $\mcal{E}$ be the event on which $\|T-\mbb{E}[T]\|\leq M^{-1}n$ and $\|Q'(T-\mbb{E}[T])Q\|\leq n/(K\log^{1/2}(Nn))$. Moreover, in Lemma \ref{lem:spectralnorms}, it is seen that $\mbb{P}(\mcal{E})\ge 1-o((Nn)^{-5})$. In addition, let $\Delta^*$ be the event on which the last assumption of Theorem \ref{thm:main} holds, i.e., on $\Delta^*$, $\|\widehat{G}-G\|\le \delta_n$ with $\mbb{P}(\Delta^*)\ge 1-o((Nn)^{-5})$. Our proof strategy is to divide the error into two parts as follows:
\begin{align*}
    &\mbb{E}\left[\int \|\bm{\widehat{g}}^*(x)-\widehat{\bm{g}}(x)\|^2dx\right]\\
    &=\mbb{E}\left[1\{\mcal{E}\cap \Delta^*\}\int \|\bm{\widehat{g}}^*(x)-\widehat{\bm{g}}(x)\|^2dx\right]+\mbb{E}\left[1\{(\mcal{E}\cap\Delta^*)^c\}\int \|\bm{\widehat{g}}^*(x)-\widehat{\bm{g}}(x)\|^2dx\right].
\end{align*}

For the first term, recall Lemma \ref{lem:error3-decompose} that
\[
\widehat{\boldsymbol g}(x) = \bigl[ I_K + (R\Delta_2-\Delta_1)(G'G+\Delta_1)^{-1}\bigr]^{-1}\bigl[ \widehat{\bg}^*(x) + R(\widehat{G}-G)'S{\bf 1}_n\bigr]. 
\]
Thus, it enables us to apply Lemma \ref{lem:lem-lem-error23} with $u=u(x)=\widehat{\boldsymbol g}(x)$, $\nu=\nu(x)=\widehat{\bg}^*(x)$, $E=(R\Delta_2-\Delta_1)(G'G+\Delta_1)^{-1}$ and $e=e(x)=R(\widehat{G}-G)'S(x){\bf 1}_n$. In the following, we will then study $E$ and $e(x)$.

We focus on $E$ first. According to Lemma \ref{lem:error3-prep}, on $\mcal{E}\cap \Delta^*$, $\|R\Delta_2-\Delta_1\|\le \|R\Delta_2\|+\|\Delta_1\|=o(1)$. Moreover, we have:
\begin{align*}
    (G'G+\Delta_1)^{-1}=(G'G)^{-1}(I_K+\Delta_1(G'G)^{-1})^{-1},
\end{align*}
where $\|\Delta_1(G'G)^{-1}\|\le CMK^{-1}\|\Delta_1\|=o(1)$ with Lemma \ref{lem:key} (c) and the assumption $M\delta_n^2=o(K)$. Then, letting $H=\Delta_1(G'G)^{-1}$, the Neumann series $(I_K+H)^{-1}=\sum_{k=0}^\infty(-H)^k$ converges absolutely in operator norm, and it yields that
\begin{align*}
    \|(I_K+H)^{-1}\|\le \sum_{k=0}^\infty\|H\|^k=\frac{1}{1-\|H\|}.
\end{align*}
It implies $\|(G'G+\Delta_1)^{-1}\|\le CMK^{-1}\frac{1}{1-\|H\|}\le CMK^{-1}$. With the above bounds and Lemma \ref{lem:error3-prep}, we obtain: on the event $\mcal{E}\cap \Delta^*$, 
\begin{align*}
    \|E\| =\|(R\Delta_2-\Delta_1)(G'G+\Delta_1)^{-1}\|\le C\sqrt{\frac{M}{K}}\delta_n. 
\end{align*}
We then apply Lemma \ref{lem:lem-lem-error23} to obtain: on the event $\mcal{E}\cap \Delta^*$,
\begin{align*}
    &\|\widehat{\bm{g}}^*(x)-\widehat{\bm{g}}(x)\|\\&\le C \|(R\Delta_2-\Delta_1)(G'G+\Delta_1)^{-1}\|\|\widehat{\bm{g}}^*(x)+R(\widehat{G}-G)'S(x){\bf 1}_n\|+\|R(\widehat{G}-G)'S(x){\bf 1}_n\|\\
    &\le C\sqrt{\frac{M}{K}}\delta_n\|\widehat{\bm{g}}^*(x)+R(\widehat{G}-G)'S(x){\bf 1}_n\|+\|R(\widehat{G}-G)'S(x){\bf 1}_n\|.
\end{align*}
Using Lemma \ref{lem:error3-prep} and the fact that $(a+b)^2\le 2(a^2+b^2)$, it yields that on the event $\mcal{E}\cap \Delta^*$,
\begin{align*}
    &\mbb{E}\left[\int \|\widehat{\bm{g}}^*(x)-\widehat{\bm{g}}(x)\|^2dx\right]\\
    &\le C\frac{M}{K}\delta_n^2\left\{\mbb{E}\left[\int \|\widehat{\bm{g}}^*(x)\|^2dx\right]+\mbb{E}\left[\int \|R(\widehat{G}-G)'S(x){\bf 1}_n\|^2dx\right]\right\}\\
    &\qquad+\mbb{E}\left[\int \|R(\widehat{G}-G)'S(x){\bf 1}_n\|^2dx\right].
\end{align*}
According to Lemma \ref{lem:error1} and Lemma \ref{lem:error23}, $\mbb{E}\left[\int \|\widehat{\bm{g}}^*(x)\|^2dx\right]\le O(K)+\int \|\bm{g}(x)\|^2dx=O(K)$. Furthermore, it is seen in Lemma \ref{lem:error3-prep} that on $\mcal{E}\cap \Delta^*$, $$\mbb{E}\left[\int \|R(\widehat{G}-G)'S(x){\bf 1}_n\|^2dx\right]\le CM\delta_n^2.$$ It then holds that on the event $\mcal{E}\cap \Delta^*$,
\begin{align*}
    &\mbb{E}\left[1\{\mcal{E}\cap \Delta^*\}\int \|\widehat{\bm{g}}^*(x)-\widehat{\bm{g}}(x)\|^2dx\right]\\
    &\le C\left[\frac{M}{K}\delta_n^2\left(K+M\delta_n^2\right)+M\delta
    _n^2\right]\le  CM\delta_n^2,
\end{align*}
where we have used the assumption $M\delta_n^2=o(K)$.

Next, we consider the bound on $(\mcal{E}\cap\Delta^*)^c=\mcal{E}^c\cup(\Delta^*)^c$. Since $\mbb{P}((\mcal{E}\cap\Delta^*)^c)\le \mbb{P}(\mcal{E}^c)+\mbb{P}(\Delta^*)=o((Nn)^{-5})$, it suffices to directly bound 
\begin{align*}
    \int \|\bm{\widehat{g}}^*(x)-\widehat{\bm{g}}(x)\|^2dx\le \int \|\bm{\widehat{g}}^*(x)\|^2dx+\int \|\bm{\widehat{g}}(x)\|^2dx.
\end{align*}
It is seen in \eqref{hatgstar-crude-bound} that almost surely
\begin{align*}
    \int\|\bm{\widehat{g}}^*(x)\|^2dx\le CKMh^{-1}\log^2 (Nn),
\end{align*}
for some constant $C>0$. It then suffices to bound the second term. Note that by definition in \eqref{our-estimator},
\begin{align*}
    \widehat{\bg}(x)=\widehat{G}'\widehat{G}(\widehat{G}'T\widehat{G})^{-1}\widehat{G}'S{\bf 1}_n.
\end{align*}
Since $\widehat{G}$ is the estimated topic matrices in the induced topic model, by regularity, $\|\widehat{G}\|_1=1$. Then, we have $\|\widehat{G}\|\le \sqrt{K}\|\widehat{G}\|_1=\sqrt{K}$. Moreover, since we add the perturbation in \eqref{our-estimator2}
, it then implies $\|(\widehat{G}'T\widehat{G})^{-1}\|\le C(Kn)^{-1}M^2\log^2(Nn)$. Moreover, it is seen in \eqref{intS1n-crude-bound} that almost surely,
\begin{align*}
    \int \|S(x)\bm{1}_n\|^2dx\le Ch^{-1}n^2,
\end{align*}
for some constant $C>0$. Together, we have: almost surely,
\begin{align*}
    \int \|\bm{\widehat{g}}(x)\|^2dx\le CK^3(Kn)^{-2}M^4\log^6(Nn)h^{-1}n^2 = CKM^4h^{-1}\log^4(Nn).
\end{align*}
Thus, combining the above two bounds, it holds almost surely that
\begin{align*}
    \int \|\bm{\widehat{g}}^*(x)-\widehat{\bm{g}}(x)\|^2dx\le C\left(\frac{KM\log^2(Nn)}{h}+\frac{KM^4\log^4(Nn)}{h}\right)\le C\frac{KM^4\log^4(Nn)}{h}.
\end{align*}
Therefore, on the event $(\mcal{E}\cap \Delta^*)^c$,
\begin{align*}
    \mbb{E}\left[1\{(\mcal{E}\cap\Delta^*)^c\}\int \|\bm{\widehat{g}}^*(x)-\widehat{\bm{g}}(x)\|^2dx\right]\le C\frac{KM^4\log^4(Nn)}{(Nn)^5h}.
\end{align*}
Together, we obtain:
\begin{align*}
    \mbb{E}\left[\int \|\bm{\widehat{g}}^*(x)-\widehat{\bm{g}}(x)\|^2dx\right]\le C\left[M\delta_n^2+\frac{KM^4\log^4(Nn)}{(Nn)^5h}\right].
\end{align*}

\subsection{Proof of Theorem \ref{thm:TM_errorbound}}
In this section, we will study the topic modeling $\|\widehat{G}-G\|$ in Theorem \ref{thm:TM_errorbound}. In \cite{ke2024entry}, the authors obtained the minimax optimal rates for the algorithm called Topic-SCORE of constructing the estimate $\widehat{G}$. However, their results only considered a fixed $K$ setting with all quantities related to $K$ hidden in the constants. In view of this, we will extend their results by giving a finer analysis via keeping track of $K$ in all the bounds thus proving Theorem \ref{thm:TM_errorbound}. For presentational convenience, restricted to this section only, we will adopt the notations in \cite{ke2024entry} while making clear connections of their notations to ours when needed.

In equation (2) of \cite{ke2024entry}, they considered the following mode: Let $X\in\mathbb{R}^{p\times n}$ be the word-count matrix.
Introduce the empirical frequency matrix $D=[d_1,d_2,\ldots,d_n]\in \mathbb{R}^{p\times n}$, defined by:
\begin{equation*}
d_i(j)= N_i^{-1}X_i(j), 
\quad 1\leq i\leq n, 1\leq j \leq p\, ,
\end{equation*}
where $\mathbb{E}[d_i]=d_i^0=\sum_{k=1}^Kw_i(k)A_k$. Write $D_0=[d_1^0, d_2^0,\ldots, d_n^0]\in\mathbb{R}^{p\times n}$. It follows that:
\[
\mathbb E D = D_0 = AW.
\]
Here, $p$ is our $M$, the number of bins; $N_1=N_1=\ldots=N_n=N$; $A\in\mbb{R}^{p\times K}$ is our topic matrix $G^{\text{hist}}=G$; $W\in\mbb{R}^{K\times n}$ is our mixed-membership matrix $\Pi'$; $D_0\in\mbb{R}^{p\times n}$ is our matrix $F^{\text{hist}}=F$ with $d_i^0 = f_i^{\text{hist}}=f_i$ and $d_i$ is our $(f_i^{\text{hist}}(1),\ldots,f_i^{\text{hist}}(M))'$ for $1\le i\le n$.

Now, define a matrix $M$ (Equation (5) in \cite{ke2024entry}) as 
\begin{align*}
    M = \diag\left(\frac{1}{n}\sum_{i=1}^nd_i\right).
\end{align*}
For each $1\leq k\leq K$, let $\hat{\xi}_k\in\mathbb{R}^p$ denote the $k$th left singular vector of $M^{-1/2}D$. Recall that $D_0=\mathbb{E}D$. In addition, define:
\begin{equation*} 
M_0: = \mathbb {E} M  = \mathrm{diag}\biggl(\frac{1}{n}\sum_{i=1}^nd_i^0 \biggr). 
\end{equation*}
Moreover, define
\begin{equation*}
\xi_k: \mbox{the $k$th eigenvector of } M_0^{-1/2}\mathbb{E}[DD']M_0^{-1/2}, \qquad 1\leq k\leq K. 
\end{equation*}
Write $\hat \Xi : = [\hat \xi_1, \cdots, \hat \xi_K]$ and $\Xi: = [\xi_1, \cdots,  \xi_K]$. Define $\hat{R}\in\mathbb{R}^{p\times (K-1)}$ by:
\begin{equation} \label{SCORE}
\hat{R}(j,k) = \hat{\xi}_{k+1}(j)/\hat{\xi}_1(j), \qquad 1\leq j\leq p, 1\leq k\leq K-1. 
\end{equation}
Let $\hat{r}_1',\hat{r}_2',\ldots,\hat{r}_p'$ denote the rows of $\hat{R}$. Then, they pointed out that $\hat{r}_j$ is a $(K-1)$-dimensional embedding of the $j$th word in the vocabulary and that there is a simplex structure associated with these word embeddings. Specifically, define the population counterpart of $\hat{R}$ as $R$, where:
\begin{equation*}
R(j,k) = \xi_{k+1}(j)/\xi_1(j), \qquad 1\leq j\leq p, 1\leq k\leq K-1. 
\end{equation*}
Let $r_1', r_2',\ldots, r_p'$ denote the rows of $R$. All these $r_j$ are contained in a simplex ${\cal S}\subset\mathbb{R}^{K-1}$ that has $K$ vertices $v_1, v_2, \ldots, v_K$.  
If the $j$th word is an anchor word (an anchor word of topic $k$ satisfies that $A_k(j)\neq 0$ and $A_\ell(j)=0$ for all other $\ell\neq k$), then $r_j$ is located at one of the vertices. Therefore, as long as each topic has at least one anchor word, they can apply a vertex hunting algorithm to recover the $K$ vertices of ${\cal S}$. As a result, applying a vertex hunting algorithm on $\{\hat{r}_j\}_{j=1}^p$ to obtain the estimates $\{\hat{v}_k\}_{k=1}^K$ and using the fact that in the oracle case, if $j$th word is an anchor word of topic $k$, $r_j = v_k$, they were able to further construct their estimate $\widehat{G}$ in \cite{ke2024entry}. Hence, the rest of this section is devoted to proving the following theorem.

\begin{thm}[Estimation of $A$]\label{thm:A_k}
Under the setting of Theorem \ref{thm:TM_errorbound}, with probability $1- o((Nn)^{-5})$,  simultaneously for $1\leq j \leq p$:
\[
\Vert \hat a_j - a_j  \Vert_1 \leq \|a_j\|_1\cdot C \sqrt{\frac{Kp\log (Nn)}{nN}}\,. 
\]
Furthermore, with probability $1- o((Nn)^{-5})$,
\begin{align*}
\|\widehat{A}-A\|&\leq C  K\sqrt{\frac{K\log (Nn)}{nN}} \\
\mathcal L (\widehat A, A)  
&\leq C  K\sqrt{\frac{pK\log (Nn)}{nN}}.
\end{align*}
\end{thm}

To prove Theorem \ref{thm:A_k}, we first present the following key results. 
\begin{thm}[Entry-wise singular vector analysis]\label{thm:row_eigenv}
Under the setting of Theorem \ref{thm:A_k}, there exists a constant $C>0$ such that with probability $1- o((Nn)^{-5})$, there is an orthogonal matrix $O\in \mathbb{R}^{K\times K}$ satisfying that simultaneously for $1\leq j \leq p$:
\[
\Vert e_j' (\hat \Xi- \Xi O')   \Vert \leq C \sqrt{\frac{h_j Kp\log (Nn)}{nN}}, 
\]
where $h_j = \sum_{k=1}^K A_k(j)$.
\end{thm}

\begin{thm}[Word embeddings]\label{thm:word_embeddings}
Under the setting of Theorem \ref{thm:A_k}, with probability $1-o((Nn)^{-5})$, 
there exist an orthogonal matrix $\Omega \in \mathbb{R}^{(K-1)\times (K-1)}$ and a constant $C>0$ such that simultaneously for $1\leq j \leq p$:
\[
\Vert \hat r_j - \Omega r_j   \Vert \leq C \sqrt{\frac{Kp\log(Nn)}{nN}}. 
\]
\end{thm}

\begin{thm}[Vertex hunting errors]\label{thm:VH_error}
Under the setting of Theorem \ref{thm:A_k}, if we apply the Successive Projection Algorithm (SPA) to do vertex hunting, with probability $1-o((Nn)^{-5})$, up to a permutation of the $K$ estimated vertices, there exists a constant $C>0$ such that simultaneously for $1\le k\le K$,
\begin{align*}
    \|\hat{v}_k-v_k\|\le C \sqrt{\frac{Kp\log(n)}{nN}},
\end{align*}
where $V=[v_1,\ldots,v_K]$ and $\widehat{V}=[\hat{v}_1,\ldots,\hat{v}_K]$.
\end{thm}

\begin{lemma}[Lemmas D.1 and E.1 in \cite{ke2024using}]\label{lem:MM0}
Suppose the conditions in Theorem \ref{thm:A_k} hold. Then,
\begin{align*}
KM_0(j,j) \asymp h_j;\qquad \text{ and} \quad   | M(j,j) - M_0(j,j) | \leq C\sqrt{\frac{\log (Nn)}{pNn}},
\end{align*}
for some constant $C>0$, with probability $1- o((Nn)^{-5})$, simultaneously for all $1\leq j \leq p$. Furthermore, with probability $1- o((Nn)^{-5})$, 
\begin{align} \label{est:MM01/2}
\Big \Vert M^{-1/2} M_0^{1/2} - I_p \Big \Vert  \leq C\sqrt{\frac{p\log (Nn)}{Nn}}.
\end{align}
\end{lemma}

Note that all above results have been proven in \cite{ke2024entry} and \cite{ke2024using} under the fixed $K$ setting. As mentioned in the beginning of this section, we will extend their proofs by keeping track of $K$ explicitly in all constants.

Now, with the above results, we prove Theorem \ref{thm:A_k}.
\begin{proof}[Proof of Theorem \ref{thm:A_k}]
    We now provide a modified proof based on \cite{ke2024entry} by clearly writing out all $K$ terms. We refer readers to the proof of Theorem 3 in \cite{ke2024entry} for more details.
    
    Recall the Topic-SCORE algorithm. Let $\widehat V = (\hat v_1, \hat v_2, \ldots, \hat v_K )$ and denote its population counterpart by $V$. We write:
\begin{align*}
\hat Q = \left(
\begin{array}{ccc}
1& \ldots  &1 \\
\hat v_1 & \ldots & \hat v_K
\end{array}
\right), 
\qquad 
Q = \left(
\begin{array}{ccc}
1& \ldots  &1 \\
 v_1 & \ldots & v_K
\end{array}
\right)
\end{align*}
Theorem \ref{thm:VH_error} shows
\[
\| \hat V   -  V \| \leq C\sqrt{\frac{Kp\log(Nn)}{Nn}},
\]
where we omit the permutation for simplicity here and throughout this proof. 
As a result:
\begin{align*}
\| \hat \pi_j^*  - \pi_j^*\| & =\left\| \hat Q^{-1} \left(\begin{array}{c}1\\ \hat r_j\end{array} \right) -
 Q^{-1} \left(\begin{array}{c}1\\  \Omega r_j\end{array} \right)  \right\| \notag\\
 & \leq \| \hat {Q}^{-1}   -  Q^{-1} \| \cdot \|r_j\|  + \|\hat{Q}^{-1} \| \|\hat r_j - \Omega r_j\|\notag\\
 & \leq C\sqrt{\frac{Kp\log(Nn)}{Nn}} = o(1)
\end{align*}
where we used the fact that $\|Q^{-1}-\hat{Q}^{-1}\|=O(\max_{k}\|v_k-\hat{v}_k\|)$ and $\|r_j\|=O(1)$ for all $1\le j\le p$, whose details can be found in the proof of Lemma~G.1 and Proof of Lemma~D.3 in \cite{ke2024using} respectively (note that both $Q,\hat{Q}$ and $r_j$ are well-normalized such that the bound of them remain unchanged), and Theorem \ref{thm:VH_error}. Considering the truncation at 0, it is not hard to see that:
\[
\|\tilde \pi_j^* - \pi_j^*\| \leq C \| \hat \pi_j^*  - \pi_j^*\|\leq C\sqrt{\frac{Kp\log(Nn)}{Nn}} =o(1);
\]
and furthermore:
\begin{align}\label{2024022601}
\|\hat \pi_j - \pi_j \|_1 & \leq \frac{\|\tilde \pi_j^* - \pi_j^*\|_1}{\|\tilde \pi_j^*\|_1} + \frac{\|\pi_j^*\|_1 \big|\|\tilde \pi_j^*\|_1 -\| \pi_j^*\|_1\big|}{\|\tilde \pi_j^*\|_1\| \pi_j^*\|_1}\notag\\
& \leq C \|\tilde \pi_j^* - \pi_j^*\|_1 \leq C\sqrt{\frac{Kp\log(Nn)}{Nn}}. 
\end{align}
by noticing that $\pi_j = \pi_j^*$ in the oracle case. 

Recall $\tilde A = M^{1/2} {\rm diag}(\hat \xi_1)\hat \Pi =:(\tilde a_1, \ldots, \tilde a_p)'$. Let $A^* = M_0^{1/2}{\rm diag} (\xi_1) \Pi = (a^*_1, \ldots, a^*_p)'$. Note that  $A = A^* [{\rm diag}({\bf 1}_p A^*)]^{-1}$. We can derive:
\begin{align}\label{2024022610}
&\|\tilde a_j - a^*_j \|_1  \leq  \Big\|  \sqrt{M(j,j)}  \, \hat \xi_1(j)\hat \pi_j  -  \sqrt{M_0(j,j)}  \, \xi_1(j) \pi_j \Big\|_1 \notag\\
& \leq C \left|\sqrt{M(j,j)}  - \sqrt{M_0(j,j)} \right| \cdot | \xi_1(j)| \cdot \|\pi_j\|_1 + C\sqrt{M_0(j,j)}\, \cdot  | \hat \xi_1(j) - \xi_1(j)|\cdot \|\pi_j\|_1  \notag\\
&\qquad + C\sqrt{M_0(j,j)}\, \cdot  | \xi_1(j)|\cdot \|\hat \pi_j - \pi_j\|_1  \notag\\
& \leq C \sqrt{\frac{Kh_j \log (n) }{Nn}},
\end{align}
where we used \eqref{2024022601}, Theorem \ref{thm:row_eigenv} and also Lemma~\ref{lem:MM0}.  Write $\tilde A = (\tilde A_1, \ldots, \tilde A_K)$ and $A^* = (A_1^*, \ldots, A_K^*)$.   We crudely bound: 
\begin{align}\label{2024022611}
\Big| \|\tilde A_k\|_1  - \|A^*_k\|_1\Big|\leq  \sum_{j=1}^p \|\tilde a_j - a^*_j \|_1 \leq
C K\sqrt{\frac{p \log (n) }{Nn}} =o(K^{-1/2}) 
\end{align}
simultaneously for all $1\leq k \leq K$,
since $h_j\asymp K/p$. By the study of oracle case in Lemma D.2 in \cite{ke2024using} and Proof of Lemma G.1 in \cite{ke2024using}, it can be deduced that $\|A^*_k\|_1\asymp K^{-1/2}$ under the growing $K$ setting. It then follows that $\|\tilde A_k\|_1\asymp K^{-1/2}$ and
\begin{align*}
\|\hat a_j - a_j \|_1 &=\Big\|{\rm diag} (1/\|\tilde A_1\|_1, \ldots, 1/\|\tilde A_K\|_1 )  \tilde a_j  - {\rm diag} (1/\| A^*_1\|_1, \ldots, 1/\|A^*_K\|_1 )  a^*_j \Big\|_1 \notag\\
& = \sum_{k=1}^K \bigg|  \frac{\tilde a_j(k) }{\|\tilde A_k\|_1} -  \frac{a^*_j(k)}{\|A^*_k\|_1} \bigg|\notag\\
& \leq  \sum_{k=1}^K \bigg|  \frac{\tilde a_j(k) - a^*_j(k) }{\|A^*_k\|_1} \bigg| + |a^*_j(k)| \frac{\big|\|\hat A_k\|_1 - \|A^*_k\|_1 \big|}{\|A^*_k\|_1\|\tilde A_k\|_1} \notag\\
& \leq C K^{1/2} \|\tilde a_j - a^*_j \|_1   + \|a_j^*\|_1 \max_{k}\Big| \|\tilde A_k\|_1  - \|A^*_k\|_1\Big| \notag\\
&  \leq
CK  \sqrt{\frac{h_j \log (n) }{Nn}}  = C\|a_j\|_1 \sqrt{\frac{Kp \log (n) }{Nn}} \,. 
\end{align*}
Here, we used (\ref{2024022610}), (\ref{2024022611}), $h_j\asymp K/p$ and the following  estimate:
\[
\|a_j^*\|_1 = \sqrt{M_0(j,j)}\,   |\xi_1(j)| \|\pi^*_j\|_1\asymp M^{-1/2}\sqrt{h_j}\asymp K^{1/2}p^{-1}.
\]
Combining all $j$ together, we immediately have the result for $\mathcal L (\hat A, A)$. Using the inequality $\|\widehat{A}-A\|\le \sqrt{M}\|\widehat{A}'-A'\|_1$ with the above bound, we obtain the result for the operator norm error.
\end{proof}

\subsubsection{Proof of Theorem \ref{thm:row_eigenv}}
To prove Theorem \ref{thm:row_eigenv}, we introduce the following lemmas. 

Recall that $\hat{\xi}_k\in\mathbb{R}^p$ is the $k$th left singular vector of $M^{-1/2}D$. Define:
\begin{equation*}
G : = M^{-1/2} DD' M^{-1/2} - \frac{n}{N} I_p, \qquad G_0:=  n\cdot M_0^{-1/2} A\Sigma_W A'M_0^{-1/2}, 
\end{equation*}
where $\Sigma_W = \frac{1-N^{-1}}{n}WW'$.

Since the identify matrix in $G$ does not affect the eigenvectors, $\hat{\xi}_k$ is the $k$th eigenvector of $G$. Additionally, \cite{ke2024entry} also showed that $\xi_k$ is the $k$th eigenvector of $G_0$ and 
\begin{equation*}
G - G_0 = M^{-1/2}DD'M^{-1/2}-M_0^{-1/2}\mathbb{E}[DD']M_0^{-1/2}. 
\end{equation*}

\begin{lemma}[Lemmas F.2, F.3, and D.3 in \cite{ke2024using}]\label{lem:eigen}
Suppose the conditions in Theorem \ref{thm:A_k} hold. Denote by $\lambda_1\geq \lambda_1\geq \ldots \geq \lambda_K$ the non-zero eigenvalues of $G_0$. There exists a constant $C$ such that:
\[
Cn \beta_n \leq \lambda_k\leq C n, \;\; \text{ for $2\leq k\leq K$},\qquad \text{and} \quad \lambda_1 \geq C^{-1} n + \max_{2\leq k \leq K}\lambda_K\,. 
\]
Furthermore, recall that $\xi_1,\xi_2, \ldots, \xi_K $ are the associated eigenvectors of $G_0$. Then:
\begin{align*}
C^{-1} \sqrt{h_j} \leq \xi_1(j)\leq C \sqrt{h_j}\, , \qquad \Vert e_j'\Xi\Vert \leq C\sqrt{h_j} \, . 
\end{align*}
\end{lemma}
Here, we point out the difference that under our setting, $\beta_n = 1$.
. Moreover, for $1\le j\le p$, $M_0(j,j)\asymp h_j\asymp K/p$ in \cite{ke2024entry} while under our setting, $h_j\asymp K/p$ but $M_0(j,j)\asymp p^{-1}$. However, under our setting, all eigenvalues of $\Sigma_W$ is of order $K^{-1}$ while it is assumed of constant order in \cite{ke2024entry}. Hence, both $G$ and $G_0$ are still well normalized under our setting such that the above bounds remain unchanged. We now prove Theorem \ref{thm:row_eigenv}.

\begin{lemma} \label{lem:G-G0}
Under the setting of Theorem \ref{thm:row_eigenv}. With probability $1- o((Nn)^{-5})$:
\begin{align}
&\Vert G- G_0 \Vert  \leq  C\sqrt{\frac{Kpn\log(Nn)}{N}} \ll n\beta_n;   \label{est:G-G0}
  \\ 
&
\Vert e_j'(G- G_0) \hat \Xi \Vert/n \leq C \sqrt{\frac{h_j Kp\log (Nn)}{nN}}  \bigg(1+ \Vert H^{-\frac 12}  (\hat \Xi-\Xi O') \Vert_{2\to \infty} \bigg)+ o(\beta_n) \cdot \Vert e_j' (\hat \Xi-\Xi O')\Vert \, , \label{est:jG-G0}
\end{align}
simultaneously for all $1\leq j \leq p$. 
\end{lemma}

Next, we use Lemma \ref{lem:eigen} and Lemma~\ref{lem:G-G0} to prove Theorem~\ref{thm:row_eigenv}. 
Let $(\hat \lambda_k, \hat \xi_k)$ and $(\hat \lambda_k, \hat \xi_k)$ be the $k$-th eigen-pairs of $G$ and $G_0$, respectively. Let $\hat \Lambda = {\rm diag} (\hat \lambda_1, \hat \lambda_2, \ldots, \hat \lambda_K)$ and $\Lambda = {\rm diag} ( \lambda_1,  \lambda_2, \ldots,  \lambda_K)$. Following Equation (A18) in \cite{ke2024entry}, we have:
\begin{align}\label{2022041901}
\Vert e_j' (\hat \Xi- \Xi O')   \Vert  \leq \Vert  e_j' \Xi (\Xi' \hat \Xi - O') \Vert  + \Vert e_j' \Xi \Xi' (G_0- G) \hat \Xi  \hat \Lambda^{-1}\Vert +  \Vert e_j'(G- G_0) \hat \Xi \hat \Lambda^{-1} \Vert.
\end{align}

In the sequel, we bound the three terms on the RHS above  one-by-one. 

First, by sine-theta theorem:
\begin{align*}
\Vert  e_j' \Xi (\Xi' \hat \Xi - O') \Vert \leq C \Vert e_j' \Xi \Vert \frac{\Vert G-G_0\Vert^2 }{|\hat \lambda_K- \lambda_{K+1}|^2}.
\end{align*}
For $1\leq k \leq  p $, by Weyl's inequality:
\begin{align}\label{ineq:weyl_lambda}
| \hat \lambda_k - \lambda_k| \leq \Vert G - G_0\Vert  \ll n \beta_n
\end{align}
with probability $1- o((Nn)^{-5})$, by employing (\ref{est:G-G0}) in Lemma~\ref{lem:G-G0}. In
particular,  $\lambda_1\asymp n $ and $Cn \beta_n<\lambda_k\leq Cn$ for $2\leq k\leq K$ and $\lambda_k=0$ otherwise (see Lemma~\ref{lem:eigen}). Thereby,  $|\hat \lambda_K- \lambda_{K+1}|\geq C n \beta_n$.  Further using  $\Vert e_j' \Xi\Vert \leq C\sqrt{h_j}$ (see Lemma~\ref{lem:eigen}),
with the aid of Lemma~\ref{lem:G-G0}, we obtain that with probability $1- o((Nn)^{-5})$:
\begin{align}\label{est:Term1}
\Vert  e_j' \Xi (\Xi'\hat \Xi - O') \Vert \leq C\sqrt{h_j} \,\cdot  \frac{Kp\log (Nn)}{Nn\beta_n^2} 
\end{align}
simultaneously for all $1\leq j \leq p$.

Next, we similarly bound the second term:
\begin{align}\label{est:Term2}
\Vert e_j' \Xi \Xi' (G_0- G) \hat \Xi \hat  \Lambda^{-1}\Vert \leq \frac{C}{n\beta_n} \Vert e_j' \Xi \Vert \Vert G- G_0\Vert  \leq C\sqrt{\frac{h_j Kp \log (Nn)}{Nn \beta_n^2}}\, .
\end{align} 
Here we used the fact that $\hat \lambda_K \geq C n\beta_n$ following from (\ref{ineq:weyl_lambda}) and Lemma~\ref{lem:eigen}.

For the last term, we simply  bound:
\begin{align}\label{est:Term3}
\Vert e_j'(G- G_0) \hat \Xi \hat \Lambda^{-1} \Vert \leq C \Vert e_j'(G- G_0) \hat \Xi \Vert/(n\beta_n)\,. 
\end{align}

Combining (\ref{est:Term1}), (\ref{est:Term2}), and (\ref{est:Term3}) into (\ref{2022041901}), by (\ref{est:jG-G0}) in Lemma~\ref{lem:G-G0},  we arrive at: 
\begin{align*}
\Vert e_j' (\hat \Xi - \Xi O')   \Vert &\leq C  \sqrt{\frac{h_j Kp\log (Nn)}{Nn\beta_n^2}} \bigg(1+ \Vert H^{-\frac 12}  (\hat \Xi-\Xi O') \Vert_{2\to \infty} \bigg)\\
&\quad+ o(1) \cdot \Vert e_j' ( \hat \Xi-\Xi O')\Vert  \, .
\end{align*}
Rearranging both sides above gives: 
\begin{align}\label{eq:almost}
\Vert e_j' (\hat \Xi -  \Xi O')   \Vert \leq C  \sqrt{\frac{h_j Kp\log (Nn)}{Nn\beta_n^2}} \bigg(1+ \Vert H^{-\frac 12}  (\hat \Xi-\Xi O') \Vert_{2\to \infty} \bigg)   ,
\end{align}
with probability $ 1- o((Nn)^{-3})$, simultaneously for all $1\leq j \leq p$. 

To proceed, we multiply both sides in (\ref{eq:almost}) by $h_j^{-1/2} $ and take the maximum. It follows that:
\begin{align*}
\Vert H^{-\frac 12}  (\hat \Xi -\Xi O') \Vert_{2\to \infty}  \leq C   \sqrt{\frac{Kp\log (Nn)}{Nn\beta_n^2}} \bigg(1 +  \Vert H_0^{-\frac 12}  (\hat \Xi -\Xi O') \Vert_{2\to \infty} \bigg)\, .
\end{align*}
Note that $\sqrt{Kp\log( Nn)}/\sqrt{Nn\beta_n^2} = o(1)$. We further rearrange both sides above and get:
\[
\Vert H^{-\frac 12}  (\hat \Xi - \Xi O') \Vert_{2\to \infty}  \leq \sqrt{\frac{Kp\log (Nn)}{Nn\beta_n^2}} = o(1)\, .
\]
Plugging the above estimate into (\ref{eq:almost}), we finally conclude 
the proof of Theorem \ref{thm:row_eigenv}.

Hence, it remains to prove Lemma \ref{lem:G-G0}. This lemma has been proven in Lemma A3 in \cite{ke2024entry} with a fixed $K$ setting, where Equation (A24) in \cite{ke2024entry} showed
\begin{align*}
    G-G_0 = E_1+E_2+E_3+E_4,
\end{align*}
with $E_i$'s defined in Equation (A25) in \cite{ke2024entry}. As a result, to prove Lemma \ref{lem:G-G0}, we present the following lemmas.
\begin{lemma}\label{lem:tech1}
Suppose the conditions in Theorem~\ref{thm:row_eigenv} hold. There exists a constant $C>0$, such that with probability $1- o((Nn)^{-5})$:
\begin{align}
&\| E_s\| \leq C\sqrt{K}\sqrt{\frac{pn\log (Nn)}{N}}, 
\qquad \text{for $s=1, 2,3$}
\label{eq:tech11}\\
& \| E_4 \|  = \| M_0^{-\frac12} (ZZ' - \mathbb{E} ZZ') M_0^{-\frac 12}\| \leq C\sqrt{K} \max\Big\{ \sqrt{\frac{pn\log (Nn)}{N^2}}, \frac{p\log (Nn)}{N}\Big\}\, . \label{eq:tech2_E4}
\end{align}
\end{lemma}

\begin{lemma}\label{lem:tech2}
Suppose the conditions in Theorem~\ref{thm:row_eigenv} hold. There exists a constant $C>0$, such that with probability $1- o(n^{-3})$, simultaneously for all $1\leq j \leq p$:
\begin{align}
 & \| e_j' E_s\hat \Xi\|/n \leq C\sqrt{K}\sqrt{\frac{h_jp\log (Nn)}{Nn}}, 
\qquad \text{for $s=2,3$} \label{eq:entryE23} \\
&  \Vert e_j' E_4 \hat \Xi  \Vert /n  \leq  C\sqrt{K}\sqrt{\frac{h_j p\log (Nn)}{Nn}} \,  \Big(1 +  \Vert H_0^{-\frac 12} (\hat \Xi- \Xi O')\Vert_{2\to \infty} \Big) \,, \label{eq:tech2_entryE4}
\end{align}
with $O = {\rm sgn} (\hat \Xi'\Xi )$. 
\end{lemma}

\begin{lemma}\label{lem:tech3}
Suppose the conditions in Theorem~\ref{thm:row_eigenv} hold. There exists a constant $C>0$, such that with probability $1- o((Nn)^{-3})$, simultaneously for all $1\leq j \leq p$:
\begin{align}
&  \Vert e_j' E_4(M_0^{1/2}M^{-1/2}-I_p) \hat \Xi  \Vert/n \leq C\sqrt{K}\sqrt{h_j} \cdot \frac{p\log (Nn)}{nN} \Big( 1 +   \Vert H^{-\frac 12}  (\hat \Xi - \Xi O') \Vert_{2\to \infty}\Big) , 
\label{eq:tech2_add}\\
&\Big\Vert e_j'  \big(M^{1/2}M_0^{-1/2} - I_p\big)\hat\Xi  \Big \Vert \leq C\sqrt{K}\sqrt{\frac{\log (Nn)}{Nn} } + o(\beta_n) \cdot \Vert e_j'(\hat \Xi - \Xi O')\Vert; \label{eq:tech3_add}
\end{align}
and furthermore:
\begin{align}
 \| e_j' E_1\hat \Xi\|/n &\leq C\sqrt{K}\sqrt{\frac{h_jp\log (Nn)}{Nn}} \, \Big(1 +  \Vert H_0^{-\frac 12} (\hat \Xi- \Xi O')\Vert_{2\to \infty} \Big)\\
 &\quad+ o(\beta_n) \cdot \Vert e_j'(\hat \Xi - \Xi O')\Vert\,.  \label{eq:tech3_entryE1}
\end{align}
\end{lemma}
Using Lemma \ref{lem:tech1}, Lemma \ref{lem:tech2} and Lemma \ref{lem:tech3} along with the fact that $G-G_0 = E_1+E_2+E_3+E_4$, we immediately obtain Lemma \ref{lem:G-G0}. In the following, it then suffices to prove the above three technical lemmas. Note that they have been proven in \cite{ke2024entry} as Lemmas A.4-4.6 without the extra $\sqrt{K}$ in all the bounds under the fixed $K$ setting. Here, to avoid redundant proof arguments in \cite{ke2024entry}, we only highlight the changes under our growing $K$ settings.

\begin{proof}[Proof of Lemma \ref{lem:tech1}]
    For $\|E_2\|$, according to Equation (A34) in \cite{ke2024entry}, it holds that
    \begin{align*}
        E_2 = M_0^{-1/2}ZW'A'M_0^{-1/2}.
    \end{align*}
    It is seen in Lemma \ref{lem:MM0} that $\|M_0^{-1/2}\|\asymp p^{1/2}$. Moreover, we have $\|W\|\asymp \sqrt{n/K}$ and $\|A\|\asymp \sqrt{K/p}$. Here, $Z$ is a $p\times n$ matrix such that $Z_i(j)=N^{-1}\sum_{m=1}^N(T_{im}(j)-d_i^0(j))$ with $T_{im}(j)\sim\text{Bernoulli}(d_i^0(j))$ and $d_i^0(j)\asymp p^{-1}$. We now bound $\|ZW'\|^2$. We apply the $\epsilon$-net approach. Fix an $1/4$-net, $\mcal{N}_{1/4}$, of the unit sphere $\mbb{R}^K$. By \cite[Lemma 5.2 and Lemma 5.4]{vershynin2010introduction}, $|\mcal{N}_{1/4}|\le 9^K$ and $\|ZW'\|=\sup_{v\in\mbb{R}^K:\|v\|=1}\|ZW'v_k\|\le 2\max_{v\in\mcal{N}_{1/4}}\|ZW'v_k\|$. For each $v\in\mbb{R}^K$ in $\mcal{N}_{1/4}$, we have:
    \begin{align*}
        \|ZW'v\|^2 = \sum_{j=1}^p\left(\frac{1}{N}\sum_{i=1}^n\sum_{m=1}^N\left(\sum_{k=1}^KW_{ki}v_k\right)(T_{im}(j)-d_i^0(j))\right)^2.
    \end{align*}
\end{proof}
We then Bernstein inequality in Lemma \ref{lem:Bernstein} to the sum in the square. Note that
\begin{align*}
    \frac{1}{N^2}\sum_{i=1}^n\sum_{m=1}^N\left(\sum_{k=1}^KW_{ki}v_k\right)^2\Var[T_{im}(j)]\le \frac{1}{N}\sum_{i=1}^n\left(\sum_{k=1}^KW_{ki}v_k\right)^2d_i^0(j).
\end{align*}
Hence, for any $\delta>0$, it holds that with probability $1-\delta$, for some constant $C>0$,
\begin{align*}
    &\left|\frac{1}{N}\sum_{i=1}^n\sum_{m=1}^N\left(\sum_{k=1}^KW_{ki}v_k\right)(T_{im}(j)-d_i^0(j))\right|\\
    &\le C\left(\sqrt{\frac{\log(2/\delta)}{N}\sum_{i=1}^n\left(\sum_{k=1}^KW_{ki}v_k\right)^2d_i^0(j)}+\frac{\log(2/\delta)}{N}\right).
\end{align*}
Applying the union bound for all $v\in \mcal{N}_{1/4}$ and all $1\le j\le p$, the above bound then holds with probability $1-9^Kp\delta$. Letting $9^Kp\delta=o((Nn)^{-5})$ with $p\le Nn$, we have:
\begin{align*}
    &\left|\frac{1}{N}\sum_{i=1}^n\sum_{m=1}^N\left(\sum_{k=1}^KW_{ki}v_k\right)(T_{im}(j)-d_i^0(j))\right|\\
    &\le C\left(\sqrt{\frac{K+\log(Nn)}{N}\sum_{i=1}^n\left(\sum_{k=1}^KW_{ki}v_k\right)^2d_i^0(j)}+\frac{K+\log(Nn)}{N}\right)\\
    &\le C\left(\sqrt{\frac{K\log(Nn)}{N}\sum_{i=1}^n\left(\sum_{k=1}^KW_{ki}v_k\right)^2d_i^0(j)}+\frac{K\log(Nn)}{N}\right).
\end{align*}
As a result, with probability $1-o((Nn)^{-5})$, it holds that
\begin{align*}
    \|ZW'\|^2&\le C\left(\frac{K\log(Nn)}{N}\sum_{j=1}^p\sum_{i=1}^n\left(\sum_{k=1}^KW_{ki}v_k\right)^2d_i^0(j)+\frac{K^2p\log^2(Nn)}{N^2}\right)\\
    &\le C\left(\frac{K\log(Nn)}{N}\|W\|^2+\frac{K^2p\log^2(Nn)}{N^2}\right)\le \frac{Cn\log(Nn)}{N}.
\end{align*}
Then, we have:
\begin{align*}
    \|E_2\|\le Cp^{1/2}\sqrt{\frac{n\log(Nn)}{N}}\sqrt{\frac{K}{p}}p^{1/2} = C\sqrt{\frac{Kpn\log(Nn)}{N}}.
\end{align*}

Moreover, since $E_3'=E_2'$ from Equation (A25) in \cite{ke2024entry}, the same bound holds.

Next, for $E_4$, in Equation (A39) in \cite{ke2024entry}, they used $M_0(j,j)\asymp h_j\asymp K/p$, which should be $M_0(j,j)\asymp p^{-1}$ for our growing $K$ setting. As a result, in Equation (A39), it suffices to prove
\begin{align}\label{bd:new2}
 \Vert H^{-\frac 12} (ZZ'- \mathbb{E} ZZ') H^{-\frac 12} \Vert
  \leq CK^{-1/2} \max\Big\{ \sqrt{\frac{pn\log (n)}{N^2}}, \frac{p\log (n)}{N}\Big\}. 
\end{align}
In the following, we will point out where this additional $K^{-1/2}$ comes from under our growing $K$ setting.

In Equation (A42) in \cite{ke2024entry}, they showed
\begin{align*}
 \mathbb{E} ( \tilde{z}_i' \tilde{z}_i)= \mathbb{E} z_i' H^{-1} z_i &=\frac{1}{N_i^2} \sum_{m=1}^{N_i} \mathbb{E} (T_{im}- \mathbb{E} T_{im})' H^{-1} (T_{im}- \mathbb{E} T_{im}) \notag\\
 &= \frac{1}{N_i^2} \sum_{m=1}^{N_i}\sum_{t=1}^p \mathbb{E}(T_{im}(t)- d_i^0(t))^2 h_t^{-1}\notag\\
 &=  \frac{1}{N_i^2} \sum_{m=1}^{N_i} \sum_{t=1}^p d_i^0(t)\big(1 - d_i^0(t)\big)h_t^{-1}\leq \frac{p}{N_i}.
 \end{align*}
 Here, again, they used $h_t\asymp K/M$ but hid this $K$ inside the constant. The finer bound should be $\frac{p}{KN}$ (recall in our setting, all $N_i$ are of the same order of $N$). Thus, whenever using Bernstein inequality, this bound appears in the variance bound yielding one more $K^{-1/2}$.

 Last, consider $E_1$. According to Equation (A52) in \cite{ke2024entry}, bounding $E_1$ involves bounding $M_0^{-1/2}DD'M_0^{-1/2}=G_0+\frac{n}{N}I_p+E_2+E_3+E_4$. Therefore, $E_1$ inherits the additional $\sqrt{K}$ term from $E_2,E_3$ and $E_4$.

 \begin{proof}[Proof of Lemma \ref{lem:tech2}]
     Due to the additional $\sqrt{K}$ term under our growing $K$ setting in all $E_i$ terms of Lemma \ref{lem:tech1} compared to Lemma A.4 in \cite{ke2024entry}, Lemma \ref{lem:tech2} regarding all $E_i$ terms inherits this additional $\sqrt{K}$ term compared to Lemma A.5 in \cite{ke2024entry}.
 \end{proof}

  \begin{proof}[Proof of Lemma \ref{lem:tech3}]
     Again, due to the additional $\sqrt{K}$ term under our growing $K$ setting in all $E_i$ terms of Lemma \ref{lem:tech1} compared to Lemma A.4 in \cite{ke2024entry}, Lemma \ref{lem:tech3} regarding all $E_i$ terms inherits this additional $\sqrt{K}$ term compared to Lemma A.6 in \cite{ke2024entry}. However, as for the last higher order $o(\beta_n)=o(1)$ term, in Proof of Lemma A.6 in \cite{ke2024entry}, they showed
\begin{align*}
\Big\Vert e_j'  \big(M^{1/2}M_0^{-1/2} - I_p\big)(\hat\Xi - \Xi O') \Big \Vert  
&\leq  \left|\sqrt{\frac{M(j,j)}{M_0(j,j)}} - 1\right| \cdot \|e_j'(\hat \Xi - \Xi O')\| \notag\\
& \leq \sqrt{\frac{p\log(n)}{Nn}} \cdot  \|e_j'(\hat \Xi - \Xi O')\|\notag\\
& = o(\beta_n) \cdot \|e_j'(\hat \Xi - \Xi O')\| \, ,
\end{align*}
using $\sqrt{\frac{p\log(n)}{Nn}}=o(\beta_n)$. When inheriting the additional $\sqrt{K}$, it becomes $\sqrt{\frac{Kp\log(n)}{Nn}}=o(\beta_n)$, which still remains unchanged under our assumption. 
 \end{proof}

\subsubsection{Proof of Theorem \ref{thm:word_embeddings}}
\begin{proof}[Proof of Theorem \ref{thm:word_embeddings}]
    According to Theorem \ref{thm:row_eigenv}, we can let $O' = \diag(\omega,\Omega')$, where $\omega\in \{-1,1\}$ and $ \Omega'$ is an orthogonal matrix in $\mathbb R^{K-1, K-1}$. Let us write $\hat \Xi_1 : = (\hat \xi_2, \ldots, \hat \xi_K)$ and similarly for $\Xi_1$. Without loss of generality, we assume $\omega =1$. Therefore:
\begin{align} \label{2024022501}
\big| \xi_1(j) - \hat \xi_1(j)  \big|  \leq C\sqrt{\frac{Kh_j p\log(n)}{Nn}}, \qquad \big\| e_j'(\hat \Xi_1 - \Xi_1) \Omega'  \big\|  \leq C\sqrt{\frac{Kh_j p\log(n)}{Nn}} \,. 
\end{align}
We rewrite:
\begin{align*}
\hat r_j' - r_j' \Omega' =\hat \Xi_1(j) \cdot  \frac{\xi_1(j) - \hat \xi_1(j) }{\hat \xi_1(j)\xi_1(j)}  -  \frac{e_j' (\hat \Xi_1 - \Xi_1 \Omega' ) }{ \xi_1(j)}. 
\end{align*}
Using Lemma~\ref{lem:eigen} together with (\ref{2024022501}), we conclude the proof. 
\end{proof}

\subsubsection{Proof of Theorem \ref{thm:VH_error}}
\begin{proof}[Proof of Theorem \ref{thm:VH_error}]
    Since the successive projection algorithm or any other vertex hunting algorithms we consider here is shown to be efficient, i.e., $\max_{1\le k\le K} \|v_k-\hat{v}_k\|= O(\max_{1\le j\le p}\|r_j-\hat{r}_j\|)$. We then obtain the bound using Theorem \ref{thm:word_embeddings}.
\end{proof}

\subsection{Proof of Theorem \ref{thm:upperbound}}
Note that under the assumptions of Theorem \ref{thm:upperbound}, both Theorem \ref{thm:main} and Theorem \ref{thm:TM_errorbound} hold. Plugging $\delta_n$ of Theorem \ref{thm:TM_errorbound} in Theorem \ref{thm:main}, with $h\asymp [K/(Nn)]^{\frac{1}{2\beta+1}}$, we obtain the desired bound.

\section{Proof of the lower bound} \label{supp:LBproofs}

\subsection{Proof of Theorem \ref{thm:lowerbound}}
For the minimax lower bound, it suffices to prove the lower bound by considering a specific case of $\mcal{G}$ and $\Pi$ among the class. We fix a $\Pi$ such that each $\pi_i$ with only one nonzero entry leading to $\Sigma_{\Pi}=K^{-1}n$ and Assumption \ref{assump2} (b). For $\mcal{G}$, we will apply the following minimax lower bound result in \cite[Theorem 2.7]{tsybakov2009introduction}.
\begin{lemma}\label{lem:minimaxlowerboundgeneral}
    Suppose there exists a family of models $\{\mbb{P}^{(s)}\}_{s=1}^J$ for $J\ge 2$, where for each $s$, the corresponding density vector $\bm{g}^{(s)}=(g_1^{(s)}.\ldots,g_K^{(s)})'\in\mcal{G}$, and the following hold:
    \begin{itemize}
        \item [(a).] For all $s\neq s'$, $\left(\int \|\bm{g}^{(s)}(x)-\bm{g}^{(s')}(x)\|^2 dx\right)^{1/2}\ge 2\varepsilon>0$.
        \vspace{0.1in}
        \item [(b).] $\frac{1}{J}\sum_{s=1}^J \text{KL}(\mbb{P}_s,\mbb{P}_0)\le \alpha\log J$, for some $\alpha\in (0,1/8)$, where $\mbb{P}^{(s)}$ denotes the probability measure associated with $\bm{g}^{(s)}\in\mcal{G}$ and $\text{KL}(\cdot,\cdot)$ denotes the Kullback–Leibler divergence.
    \end{itemize}
    Then, there exists a positive constant $c(\alpha)$ only depending on $\alpha$ such that
    \begin{align*}
        \inf_{\widehat{\bm{g}}}\sup_{\bm{g}\in\mcal{G}}\mbb{E}\left[\int \|\bm{g}(x)-\bm{\widehat{g}}(x)\|^2dx\right]\ge c(\alpha)\varepsilon^2.
    \end{align*}
\end{lemma}

It then suffices to verify the conditions (a) and (b) above by constructing such a class $\{\bm{g}^{(s)}\}_{s=1}^J$ with $\varepsilon^2 = K\left(\frac{K}{Nn}\right)^{\frac{2\beta}{2\beta+1}}$. We divide the interval $[-K/2,K/2]$ evenly into $K$ blocks $\{S^{(k)}\}_{k=1}^K$ with $S^{(k)}:=\left[a_k,a_k+1\right]$. For each $1\le k\le K$, we first need to construct a base density $g_{0,k}$ such that (a) it is in the Nikol’ski class (i.e., satisfying Assumption \ref{assump1}); (b) there exists a region such that only density $g_{0,k}$ is nonzero while the rest densities are zero (for the existence of anchor region); (c) there exists a region where all densities $g_{0,k}$ are lower bounded away from zero (where we can perturb them to construct $\{g_k^{(s)}\}$). To achieve this, we will utilize the following three types of bump functions that are both infinitely-order smooth:
\begin{align*}
    \phi_0(z;a) &= \text{exp}\left(\frac{1}{z^2-a^2}\right)1\{-a<z<a\},\\
    \phi_1(z;a)&=\text{exp}\left((z-a)^2+\frac{1}{(z-a)^2}\right)1\{z>a\},\\ \phi_2(z;a)&=\text{exp}\left((z-a)^2+\frac{1}{(z-a)^2}\right)1\{z<a\}.
\end{align*}
Fix $T>0$ and an interval $[-2T,2T]$. To construct $g_{0,k}(x)$, we first put a bump $\phi_0(x;2T)$ in the middle symmetric around $x=0$ such that there exists a constant $c_0>0$ (depending on $T$) such that $\inf_{x\in [-T,T]}\phi_0(x,2T)\ge c_0>0$. Next, we put a bump $\phi_0(x-(2T+2k-1);1)$ to the right, which centers at $x=2T+2k-1$ over the interval $[2T+2k-2,2T+2k]$. Further, for the right tail, we add a bump $\phi_1(x;2T+2K)$, which is zero at $x=2T+2K$ with an exponential tail decaying to zero. Then, we set all other values of $g_{0,k}(x)$ on $[0,+\infty)$ except these bumps to 0, and $g_{0,k}(x)=g_{0,k}(-x)$ for $x<0$ to make it symmetric around $x=0$. In order to make it a density, we properly normalize it and define it as:
\begin{align}\label{eq:g0kconstruction}
    g_{0,k}(x) &= \frac{1}{5Z_1}\phi_0(x;2T)+\frac{1}{5Z_2}\left[\phi_0(x-(2T+2k-1);1)+\phi_0(x+(2T+2k-1);1)\right]\\\nonumber
    &\quad+\frac{1}{5Z_3}\left[\phi_1(x;2T+2K)+\phi_2(x;-(2T+2K))\right],
\end{align}
where 
\begin{align*}
    Z_1 = \int \phi_0(z;2T)dz,\ Z_2=\int \phi_0(z;1)dz,\ Z_3 = \int \phi_1(z;2T+2K)dz.
\end{align*}
Here, it is seen that $\mcal{A}_k=$ is the anchor region for density $g_{0,k}$. In the following, we will make perturbations to this base density $g_{0,k}$ on the interval $[-T,T]$.

Let $\psi$ be a bump function belonging to the Nikol’ski class (i.e., satisfying Assumption \ref{assump1}) such that it is supported on $[-1,1]$, $\int_0^1 \psi(x)dx=0$, $\int_0^1 \psi^2(x)dx=1$ and $\|\psi\|_{\infty}<\infty$. Picking an integer $B\asymp (Nn/K)^{\frac{1}{2\beta+1}}$, for each $1\le b\le B$, we let 
\begin{align*}
    \psi_b(x):=\psi\left(Bx-b\right).
\end{align*}
Since $B^{-1}=o(1)$, by picking a properly large $T$, it is seen that $\{\psi_b\}_{b=1}^B$ have disjoint support within $[-T,T]$ such that $\int \psi_b(u)du=0$ and $\int \psi^2_b(u)du=B^{-1}$. We will use $\psi_b$ to perturb the base density $g_{0,k}$ for each $1\le k\le K$. The following lemma is a well-known result (e.g., see \cite[Lemma 2.9]{tsybakov2009introduction} for a proof).
\begin{lemma}[Varshamov–Gilbert]
    For any integer $A\ge 8$, there exists an integer $J\ge 2^{A/8}$, and vectors $\nu^{(0)},\ldots,\nu^{(J)}\in \{0,1\}^A$ such that $\nu^{(0)}$ is a zero vector and that $\|\nu^{(s)}-\nu^{(t)}\|_1\ge A/8$ for all $0\le s\neq t\le J$.
\end{lemma}
Recall that for each $s$, we have $K$ densities $\{g_k^{(s)}\}_{k=1}^K$ with $B$ bumps $\{\psi_b\}$ for each $k$. We then need a vector of length $KB$ from the Varshamov–Gilbert lemma. Set $A=KB$ then there exists an integer $J\ge 2^{KB/8}$ and a vector $\omega^{(s)}=((\omega^{(s,1)})',\ldots,(\omega^{(s,K)})')'\in \{0,1\}^{KB}$, where each $\omega^{(s,k)}=(\omega_1^{(s,k)},\ldots,\omega_B^{(s,k)})'\in \{0,1\}^B$ such that $\|\omega^{(s)}-\omega^{(s')}\|_1\ge KB/8$ for any $s\neq s'$. Now, our construction is given by
\begin{align*}
    g_k^{(s)}(x)=g_{0,k}(x)+\delta\sum_{b=1}^B\omega_b^{(s,k)}\psi_b(x),\ \text{for some}\ \delta\asymp32\left(\frac{K}{Nn}\right)^{\frac{\beta}{2\beta+1}}=o(1).
\end{align*}
In the following, it suffices to verify the condition (a) and (b) in Lemma \ref{lem:minimaxlowerboundgeneral}.

Before this, we first check all $\{g_k^{(s)}\}$ are densities in the Nikol’ski class (i.e., satisfying Assumption \ref{assump1}). For each $1\le k\le K$, it is clear that $\int g_k^{(s)}(x)dx=\int g_{0,k}(x)dx=1$ due to $\int \psi_b(x)dx=0$. Outside $[-T,T]$, $g_k^{(s)}(x)=g_{0,k}(x)\ge 0$ and inside $[-T,T]$, since $\{\psi_b\}$ have disjoint supports, we then have $g_k^{(s)}(x)\ge c_0-\delta\|\psi\|_{\infty}=c_0-o(1)$. Thus, all $g_k^{(s)}$'s are nonnegative and they are densities. Meanwhile, they all have constant integrals over anchor regions, i.e., $[2T+2k-2,2T+2k]\cup[-2T-2k+2,-2T-2k]$, where only density $g_{0,k}$ is nonzero for $1\le k\le K$. Moreover, for any derivative order $r>0$, it holds that
\begin{align*}
    \psi^{(r)}_b(x)=B^r\psi^{(r)}(Bx-b).
\end{align*}
Hence, for each $1\le k\le K$, note that for each bump $\psi_b$, we have:
\begin{align*}
    &\int \left| \psi_b^{(\lfloor\beta\rfloor)}(x + t) - \psi_b^{(\lfloor\beta\rfloor)}(x) \right|^2 dx\\
    &= B^{2\lfloor\beta\rfloor}\int \left| \psi^{(\lfloor\beta\rfloor)}(Bx-b + Bt) - \psi^{(\lfloor\beta\rfloor)}(Bx-b) \right|^2 dx\\
    &\le B^{2\lfloor\beta\rfloor-1}\int \left| \psi^{(\lfloor\beta\rfloor)}(u + Bt) - \psi^{(\lfloor\beta\rfloor)}(u) \right|^2 du \\
    &\le C B^{2\lfloor\beta\rfloor-1}|Bt|^{2\beta-2\lfloor\beta\rfloor}=CB^{2\beta-1}|t|^{2\beta-2\lfloor\beta\rfloor}.
\end{align*}
Thus, we obtain:
\begin{align*}
    \frac{\int \left| \psi_b^{(\lfloor\beta\rfloor)}(x + t) - \psi_b^{(\lfloor\beta\rfloor)}(x) \right|^2dx}{|t|^{2\beta-2\lfloor\beta\rfloor}}\le CB^{2\beta-1}.
\end{align*}
Since we have $B$ bumps in total, we have 
\begin{align*}
    \sum_{b=1}^B\frac{\int \left| \psi_b^{(\lfloor\beta\rfloor)}(x + t) - \psi_b^{(\lfloor\beta\rfloor)}(x) \right|^2dx}{|t|^{2\beta-2\lfloor\beta\rfloor}}\le CB^{2\beta}.
\end{align*}
Since $\delta^2B^{2\beta}=O(1)$, this ensures each $g_k^{(s)}$ belongs to the Nikol’ski class (i.e., satisfying Assumption \ref{assump1}) given $g_{0,k}$ is already in the Nikol’ski class.

Next, for the condition (a) in Lemma \ref{lem:minimaxlowerboundgeneral}, due to the construction that all $\{\psi_b\}_{b=1}^B$ have disjoint supports, it holds that for any $s\neq s'$,
\begin{align*}
    \sum_{k=1}^K\|g_k^{(s)}-g^{(s')}_k\|_{L^2}^2=\delta^2\sum_{k=1}^K \sum_{b=1}^B(\omega_b^{(s,k)}-\omega_b^{(s',k)})^2\|\psi_b(x)\|_{L^2}^2,
\end{align*}
where we use the notation $\|\cdot\|_{L^2}$ to denote the $L_2$ norm of a function. Since $\|\psi_b(x)\|_{L^2}^2=B^{-1}$ and $\|\omega^{(s)}-\omega^{(s')}\|_1=\|\omega^{(s)}-\omega^{(s')}\|_2^2\ge KB/8$ (they are in $\{0,1\}^{KB}$) for any $s\neq s'$, we then have:
\begin{align*}
    \sum_{k=1}^K\|g_k^{(s)}-g^{(s')}_k\|_{L^2}^2\ge \delta^2B^{-1}KB/8=K\delta^2/8\ge 4K\delta^2= 4\varepsilon^2.
\end{align*}
The condition (a) is satisfied.

As for the condition (b) in Lemma \ref{lem:minimaxlowerboundgeneral}, for all $1\le i\le n$ and each $s$, according to our model, we have:
\begin{align*}
    f_i^{(s)}(x)=\sum_{\ell\neq k}\pi_i(\ell)g_0(x)+\pi_i(k)g_k^{(s)}(x),\quad
    f_i^{(0)}(x)=g_0(x).
\end{align*}
It implies:
\begin{align*}
    f_i^{(s)}(x)-f_i^{(0)}(x)=\sum_{k=1}^K\pi_i(k)(g_k^{(s)}(x)-g_0(x))=\delta\sum_{k=1}^K\pi_i(k)\sum_{b=1}^B\omega_b^{(s,k)}\psi_b(x)=:\Delta_i^{(s)}(x).
\end{align*}
Note that the KL divergence between $\mathbb{P}^{(s)}$ and $\mathbb{P}^{(0)}$ is:
\[
\mathrm{KL}(\mathbb{P}^{(s)}, \mathbb{P}^{(0)}) = \sum_{i=1}^n \sum_{j=1}^N \int f_i^{(s)}(x) \log \left( \frac{f_i^{(s)}(x)}{f_i^{(0)}(x)} \right) dx.
\]
Since $\{\psi_b(x)\}$ have disjoint supports, we have for all $x\in [-T,T]$, $|\Delta_i^{(s)}(x)/f_i^{(0)}(x)|\le \delta \|\psi\|_{\infty}/c_0=o(1)$. Then using the inequality $\log(1 + u) \le u - \frac{u^2}{2(1 + |u|)}$ for $|u| < 1$, and noting that $\int \Delta_i^{(s)}(x) dx = 0$ such that the first term order vanishes, there exists a constant $C>0$ such that for all $1\le i\le n$ and each  $s$,
\[
\int f_i^{(s)}(x) \log \left( \frac{f_i^{(s)}(x)}{f_i^{(0)}(x)} \right) dx \le C \int \frac{ (\Delta_i^{(s)}(x))^2 }{ f_i^{(0)}(x) } dx\le Cc_0^{-1} \| \Delta_i^{(s)} \|_{L^2}^2.
\]
It then remains to bound $\| \Delta_i^{(s)} \|_{L^2}^2$. Due to the disjoint support of the bumps $\{\psi_b\}$, we have:
\begin{align*}
    \| \Delta_i^{(s)} \|_{L^2}^2 &= \left\|\delta\sum_{k=1}^K\pi_i(k)\sum_{b=1}^B\omega_b^{(s,k)}\psi_b(x)\right\|_{L^2}^2\\
    &=\delta^2\sum_{b=1}^B\left(\sum_{k=1}^K\pi_i(k)\omega_b^{(s,k)}\right)^2\frac{1}{B}\\
    &\le \delta^2.
\end{align*}
Therefore, the KL divergence between $\mathbb{P}^{(s)}$ and $\mathbb{P}^{(0)}$ satisfies:
\[
\mathrm{KL}(\mathbb{P}^{(s)}, \mathbb{P}^{(0)}) = O(Nn\delta^2).
\]
Recall that $\delta\asymp\left(\frac{K}{Nn}\right)^{\frac{\beta}{2\beta+1}}$ and $B\asymp (Nn/K)^{\frac{1}{2\beta+1}}$. Then, it is seen that 
\begin{align*}
    \frac{1}{J}\sum_{s=1}^J \text{KL}(\mbb{P}_s,\mbb{P}_0)=O(Nn\delta^2)=O(KB)=O(\log J).
\end{align*}
Therefore, the condition (b) is also satisfied by picking the constant in $B$ properly small.

Finally, we need to show that the above constructed $\{g_k^{(s)}\}_{k=1}^K$ satisfy Assumption \ref{assump2} and \eqref{tscore-op-norm} required in the upper bound Theorem \ref{thm:upperbound}  so that the lower bound $\varepsilon^2$ in Lemma \ref{lem:minimaxlowerboundgeneral} indeed provides a minimax lower bound for the class considered in the upper bound in Theorem \ref{thm:upperbound}. 

To show Assumption \ref{assump2} is satisfied, we aim to utilize the proof of Lemma \ref{lem:key} (c). Recall that
\begin{align}\label{perturb}
    g_k^{(s)}(x)=g_{0,k}(x)+\delta\sum_{b=1}^B\omega_b^{(s,k)}\psi_b(x).
\end{align}
Consequently, letting $G^{(s)}=(G_{mk}^{(s)})$, $G_0=(G_{0,mk})$ and $E^{(s)}=(E_{mk}^{(s)})$ be three matrices in $\mbb{R}^{M\times K}$, we have:
\begin{align}\label{perturbationonG}
    G_{mk}^{(s)}=G_{0,mk}+\delta E_{mk}^{(s)},
\end{align}
where
\begin{align*}
    G_{mk}^{(s)} = \int_{\mcal{C}_m}g_k^{(s)}(x)dx,\ G_{0,mk} = \int_{\mcal{C}_m}g_{0,k}(x)dx,\ E_{mk}^{(s)}=\sum_{b=1}^B\omega_b^{(s,k)}\int_{\mcal{C}_m}\psi_b(x)dx.
\end{align*}
Following the proof of Lemma \ref{lem:key} (c), our plan to show Assumption \ref{assump2} is satisfied is as follows: we construct a set of bins $\{\mcal{C}_m\}_{m=1}^M$, and first show the main term $G_{0}$ satisfying Lemma \ref{lem:key} (c), and then show the perturbation is small enough. Thus, the eigenvalue bounds of $\Sigma_{G^{(s)}}$ can lead to the eigenvalue bounds of $\Sigma_{\bm{g}^{(s)}}$.

Let $h_K(x)=\sum_{k=1}^Kg_{0,k}(x)$. Since $h_K$ has a total mass $\int h_K(x)dx=K$, we then evenly divide the total mass $K$ into $M$ bins over $\mbb{R}$ such that each bin contains mass equal to $K/M$. Formally, let $q_1,q_2,\ldots,q_{M-1}$ be the $K/M, 2K/M,\ldots,(M-1)K/M$-th quantile of $h_K(x)$ respectively. Then, we set $\mcal{C}_1 = (-\infty,q_1)$, $\mcal{C}_M=[q_{M-1},+\infty)$ and for $2\le m\le M-1$, $\mcal{C}_m = [q_{m-1},q_m)$. Hence, by construction, for all $1\le k\le K$,
\begin{align*}
    \sum_{k=1}^M\int_{\mcal{C}_m}g_{0,k}(x)dx = K/M,
\end{align*}
implying $\|G_{0}\mbf{1}_{K}\|_{\infty}=K/M$. As for $\Sigma_{G_0}$, let $G_{0,m:}\in\mbb{R}^K$ be its $m$-th row vector. Then, we have $\Sigma_G = \frac{M}{K}\sum_{m=1}^MG_{0,m:}G_{0,m}'$ as a sum of rank one matrices. For each $1\le k\le K$, let $\mcal{A}_k$ denote the anchor region of the density $g_{0,k}$, where only $g_{0,k}$ is nonzero. We can then divide all bins $\mcal{C}_m$ into three cases: (a) There exists $1\le k\le K$ such that $\mcal{C}_m\subset \mcal{A}_k$; (b) $\mcal{C}_m\cap\left(\cup_{k=1}^K\mcal{A}_k\right)=\emptyset$, i.e., $\mcal{C}_m$ does not overlap with any of the anchor region $\mcal{A}_k$; (c) there exists $1\le k\le K$ such that $\mcal{C}_m\cap\mcal{A}_k\neq \emptyset$ but there does not exist $1\le k\le K$ such that $\mcal{C}_m\subset\mcal{A}_k$, i.e., $\mcal{C}_m$ overlaps with some or more $\mcal{A}_k$ but it does not contain in any $\mcal{A}_k$. We thus have:
\begin{align*}
    \Sigma_G &= \frac{M}{K}\sum_{m=1}^MG_{0,m:}G_{0,m}'=\frac{M}{K}\left(\sum_{m\in\text{Case\ a}}+\sum_{m\in\text{Case\ b}}+\sum_{m\in\text{Case\ c}}\right)G_{0,m:}G_{0,m}'\\
    &=\Sigma_{G,a}+\left(\Sigma_{G,b}+\Sigma_{G,c}\right).
\end{align*}
In the following, we analyze them separately.

For $\Sigma_{G,a}$, since $\mcal{C}_m\subset \mcal{A}_k$, $G_{0,mk}=\int_{\mcal{C}_m}g_{0,k}(x)dx= \frac{K}{M}$ and $G_{0,mk'}=0$ for any $k'\neq k$. It then implies:
\begin{align*}
    \Sigma_{G,a}=\frac{M}{K}\sum_{k=1}^K\sum_{m:\mcal{C}_m\subset\mcal{A}_k}\frac{K^2}{M^2}e_ke_k'=\sum_{k=1}^K\left(\frac{K}{M}|\{m:\mcal{C}_m\subset\mcal{A}_k\}|\right)e_ke_k'.
\end{align*}
Let $\eta_k = \int_{\mcal{A}_k}g_{0,k}(x)dx$. Note that $\eta_k\asymp1$ due to our construction \eqref{eq:g0kconstruction}. Hence, the number of such bins $\mcal{C}_m$ contained in $\mcal{A}_k$ $|\{m:\mcal{C}_m\subset\mcal{A}_k\}|\asymp \eta_k\frac{M}{K}$. Therefore, there exist constant $C\ge c>0$ such that
\begin{align*}
     \Sigma_{G,a}=\diag(\sigma_{a,1},\sigma_{a,2},\ldots,\sigma_{a,K}),
\end{align*}
where $c\le \sigma_{a,k}\le C$ for all $k$. Hence, $\lambda_{\min}(\Sigma_G)\ge \lambda_{\min}(\Sigma_{G,a})\ge c$. As for the largest eigenvalue, we will then bound $\|\Sigma_{G,b}+\Sigma_{G,c}\|$.

Consider $\Sigma_{G,b}$. Note that outside $\cup_{k=1}^K\mcal{A}_k$, all densities $g_{0,k}(x)$ are equal by construction \eqref{eq:g0kconstruction}. Since $\int_{\mcal{C}_m}\sum_{k=1}^Kg_{0,k}(x)dx=\frac{K}{M}$, it then holds that $\int_{\mcal{C}_m}g_{0,k}(x)dx=\frac{1}{M}$ for all $1\le k\le K$. As a result, we obtain $G_{0,mk}=\frac{1}{M}$ for all $1\le k\le K$, and
\begin{align*}
    \Sigma_{G,b} = \frac{M}{K}\sum_{m\in \text{Case b}}\frac{1}{M^2}1_K1_K'.
\end{align*}
It then yields that
\begin{align*}
    \|\Sigma_{G,b}\|\le \frac{M}{K}M\frac{1}{M^2}K=1
\end{align*}

Finally, for $\Sigma_{G,c}$, recall that Case (c) corresponds to $m$, where $\mcal{C}_m$ partially overlaps with some or more $\mcal{A}_k$ while not strictly contained in any $\mcal{A}_k$. Recall the construction of $g_{0,k}$ in \eqref{eq:g0kconstruction}. We further divide Case (c) into three subcases. Case (c1): $\mcal{C}_m\cap \mcal{A}_1\neq \emptyset$ and $\mcal{C}_m\cap[-2T,2T]\neq\emptyset$; Case (c2): for $1\le k\le K-1$, $\mcal{C}_m\cap\mcal{A}_{k}\neq\emptyset$ and $\mcal{C}_m\cap\mcal{A}_{k+1}\neq\emptyset$; Case (c3): $\mcal{C}_m\cap \mcal{A}_K\neq\emptyset$ and $\mcal{C}_m\cap \left((-\infty,-2T-2K)\cup(2T+2K,+\infty)\right)\neq\emptyset$. We first analyze Case (c1) as Case (c3) is similar to it. 

For Case (c1), according to the definition of $g_{0,k}$ in \eqref{eq:g0kconstruction}, the function $\phi_0(x;2T)$ from all $g_{0,k}(x)$'s are nonzero on $\mcal{C}_m\cap[-2T,2T]$ and only $g_{0,1}(x)$ is nonzero on $\mcal{C}_m\cap \mcal{A}_1$. Hence, we obtain:
\begin{align*}
    G_{0,m:}=c_m\mbf{1}_K+a_1e_1,
\end{align*}
where $c_m = \frac{1}{5Z_1  }\int_{\mcal{C}_m\cap[-2T,2T]}\phi_0(x;2T)$ and $a_1 = \int_{\mcal{C}_m\cap \mcal{A}_1}g_{0,1}(x)dx$. Then, we have:
\begin{align*}
    \frac{M}{K}\sum_{m\in \text{Case c1}}G_{0,m:}G_{0,m:}'&=\frac{M}{K}(c_m\mbf{1}_K+a_1e_1)(c_m\mbf{1}_K+a_1e_1)'\\
    &=\frac{M}{K}(c_m^21_K1_K'+a_1c_m1_Ke_1'+a_1c_me_11_K'+a_1^2e_1e_1').
\end{align*}
Since $c_m\le 1/M, a_1\le K/M$, $\|1_K1_K'\|=K$, $\|1_Ke_1'\|=\sqrt{K}$ and $\|e_1e_1'\|=1$, we then obtain:
\begin{align*}
    \left\|\frac{M}{K}\sum_{m\in \text{Case c1}}G_{0,m:}G_{0,m:}\right\|\le \frac{M}{K}\left(\frac{K}{M^2}+2\frac{K\sqrt{K}}{M^2}+\frac{K^2}{M^2}\right)=O(1),
\end{align*}
noting that $M\ge K$. Similarly, we also obtain the same bound for Case (c3).
\begin{align*}
    \left\|\frac{M}{K}\sum_{m\in \text{Case c3}}G_{0,m:}G_{0,m:}\right\|=O(1).
\end{align*}
Last, focus on Case (c2). Since only density $g_{0,k}$ is nonzero on $\mcal{A}_k$, for $\mcal{C}_m\cap\mcal{A}_{k}\neq\emptyset$ and $\mcal{C}_m\cap\mcal{A}_{k+1}\neq\emptyset$, $G_{0,mk}=\int_{\mcal{C}_m\cap\mcal{A}_k}g_{0,k}(x)dx$ and $G_{0,m(k+1)}=\int_{\mcal{C}_m\cap\mcal{A}_{k+1}}g_{0,k+1}(x)dx$ while $G_{0,mk'}=0$ for $k'\notin\{k,k+1\}$. Hence, it yields that 
\begin{align*}
    G_{0,m:}=a_ke_k+a_{k+1}e_{k+1},
\end{align*}
where $a_k = \int_{\mcal{C}_m\cap \mcal{A}_k}g_{0,k}(x)dx\le K/M$ for all $1\le k\le K-1$. Then, we obtain:
\begin{align*}
    \frac{M}{K}\sum_{m\in \text{Case c2}}G_{0,m:}G_{0,m:} = \frac{M}{K}\sum_{k=1}^{K-1}(a_ke_k+a_{k+1}e_{k+1})(a_ke_k+a_{k+1}e_{k+1})'.
\end{align*}
Let $S = \sum_{k=1}^{K-1}(a_ke_k+a_{k+1}e_{k+1})(a_ke_k+a_{k+1}e_{k+1})'$. It is seen that $S$ is a symmetric tridiagonal matrix. The diagonal is $S_{11}=a_1^2$, $S_{KK}=a_K^2$ and $S_{kk}=2a_{kk}^2$ for $2\le k\le K-1$. The off-diagonal is $S_{k(k+1)}=a_ka_{k+1}$. Hence, using the bound
\begin{align*}
    \|S\|\le \max_{1\le i\le k}\sum_{j=1}^k|S_{ij}|=4\left(\frac{K}{M}\right)^2,
\end{align*}
we obtain:
\begin{align*}
    \left\|\frac{M}{K}\sum_{m\in \text{Case c2}}G_{0,m:}G_{0,m:}\right\|\le 4\frac{K}{M}=O(1).
\end{align*}
Combining all subcases (c1)-(c3) together, it yields that $\|\Sigma_{G,c}\|=O(1)$. Since it is seen that $\|\Sigma_{G,b}\|\le 1$, we can conclude $\|\Sigma_{G,b}+\Sigma_{G,c}\|=O(1)$, implying $\lambda_{\max}(\Sigma_G)\le \lambda_{\max}(\Sigma_{G,a})+O(1)=O(1)$. Therefore, we have shown with bins $\{\mcal{C}_m\}_{m=1}^M$, the main term matrix $G_0$ satisfies Lemma \ref{lem:key} (c).

Now, we consider the perturbation term $\delta E^{(s)}$. Recall \eqref{perturbationonG} that 
\begin{align*}
    E_{mk}^{(s)}=\sum_{b=1}^B\omega_b^{(s,k)}\int_{\mcal{C}_m}\psi_b(x)dx.
\end{align*}
Note that all bumps used for perturbation are supported within $[-T,T]$ for some fixed $T>0$. Hence, only $\mcal{C}_m$ within $[-T,T]$ makes $\int_{\mcal{C}_m}\psi_b(x)dx$ nonzero. Since all densities $g_{0,k}$ on $[-T,T]$ are bounded from above and away from zero. By the construction $\int_{\mcal{C}_m}\sum_{k=1}^Kg_{0,k}(x)dx=K/M$, it implies $|\mcal{C}_m|\asymp 1/M$. Hence, due to the disjoint support of the bumps, we have:
\begin{align*}
    |\delta E_{mk}^{(s)}| &= \left|\delta\sum_{b=1}^B\omega_b^{(s,k)}\int_{\mcal{C}_m}\psi_b(x)dx\right|\le \int_{\mcal{C}_m}\left|\delta\sum_{b=1}^B\omega_b^{(s,k)}\psi_b(x)\right|dx\\
    &=O\left( \frac{\delta\|\psi\|_{\infty}}{M}\right)=o\left(\frac{1}{M}\right).
\end{align*}
Then, it yields that
\begin{align*}
    \|\delta E^{(s)}1_K\|_{\infty}&=o\left(\frac{K}{M}\right),\\
    \|\delta E^{(s)}\|=o\left(\sqrt{KM}\frac{1}{M}\right)&=o\left(\sqrt{\frac{K}{M}}\right)=o(1).
\end{align*}
Since $G^{(s)}= G_0+\delta E^{(s)}$ and we already showed that $G_0$ satisfies Lemma \ref{lem:key} (c), $G^{(s)}$ also satisfies Lemma \ref{lem:key} (c) under the bins $\{\mcal{C}_m\}_{m=1}^M$. Hence, following the proof of Lemma \ref{lem:key} (c), the eigenvalue bounds of $\Sigma_{G^{(s)}}$ imply the eigenvalue bounds of $\Sigma_{\bm{g}^{(s)}}$ such that Assumption \ref{assump2} is satisfied. Moreover, according to Theorem \ref{thm:TM_errorbound}, there exists an estimator $\widehat{G}^{(s)}$ such that \eqref{tscore-op-norm} holds. Consequently, our constructed $\{g_k^{(s)}\}_{k=1}^{K}$ satisfy the conditions of Theorem \ref{thm:upperbound} in the upper bound, implying a valid minimax lower bound.

\section{Extension to the case of a general $d$} \label{supp:extension}
We now present the general version of our main result, Theorem \ref{thm:main}, in dimension $d>1$. 
\begin{thm} \label{thm:mainind}
Fix $d>1$. Consider the model \eqref{model-1}-\eqref{model-2} in $\mbb{R}^d$ and Assumption~\ref{assump1prime}. Suppose the assumptions (i)-(iii) and (v) in Theorem \ref{thm:main} hold. In addition, suppose $\prod_{j=1}^d h_j\rightarrow 0$ and $Nn\prod_{j=1}^d h_j\rightarrow\infty$. Now, consider the estimator $\widehat{\boldsymbol g}^+(\bm{x})$ in \eqref{our-estimator2} using the above multivariate product kernel $\mcal{K}_h^{\text{prod}}$. Suppose $K\leq M\leq [Nn/\log^2(Nn)]^{1/2}$ and $ (Nn)^{-1}K\ll \prod_{j=1}^d h_j\ll \log^{-1}(Nn)$. 
    Then, there exists a constant $C_0>0$ such that 
\begin{align*}
\mbb{E}\left[\int_{\mbb{R}^d} \|\widehat{\bm{g}}^+(\bm{x})-\bm{g}(\bm{x})\|^2d\bm{x}\right]\le C_0 \cdot K\left( \sum_{j=1}^d h_j^{2\beta_j}  + \frac{K}{Nn\prod_{j=1}^d h_j}+\frac{M}{K}\delta_n^2  +  \frac{K^2}{Nn} \right). 
\end{align*}
\end{thm}

\subsection{Proof of Theorem \ref{thm:mainind}}
Note that as mentioned in Section \ref{subsec:d>1}. The last two terms of the bound in Theorem \ref{thm:mainind} are irrelevant of the dimension. It thus suffices to prove the first two terms. The second term is obtained simply by replacing the bandwidth $h$ by the product of bandwidths $\prod_{j=1}^d h_j$ due to the use of the product kernel in our method. Then, in the following, we will only prove the first term.

The first term relates to the bias of KDE for a multivariate density. It is from a generalized version of Lemma \ref{lem:error1}, which we present as follows.

\begin{lemma}[Multivariate analogue of Lemma \ref{lem:error1}]
\label{lem:multivariate-lemma4.3}
Let $d>1$.  For each $1\le k\le K$ let $g_k:\mathbb{R}^d\to\mathbb{R}$ be a probability density belonging to an \emph{anisotropic Nikol'ski class} with smoothness vector $\boldsymbol{\beta}=(\beta_1,\dots,\beta_d)$ and constant $L>0$ in the sense that for each multi-index $r=(r_1,\dots,r_d)$ with $0\le r_j\le\lfloor\beta_j\rfloor$ the mixed partial $\partial^r g_k$ exists in $L^2(\mathbb{R}^d)$ and for all $t=(t_1,\dots,t_d)\in\mathbb{R}^d$
\[
\Big\| \partial^r g_k(\cdot + t) - \partial^r g_k(\cdot)\Big\|_{L^2(\mathbb{R}^d)}
\le L \sum_{j=1}^d |t_j|^{\beta_j - r_j} .
\]
Recall that the multivariate kernel is the product kernel
\[
\mcal{K}_{\mathbf{h}}(u) = \prod_{j=1}^d \frac{1}{h_j} \mcal{K}_j\!\left(\frac{u_j}{h_j}\right),\qquad \mathbf{h}=(h_1,\dots,h_d),
\]
where the univariate kernel $\mcal{K}_j$ has order $\ell_j=\lfloor\beta_j\rfloor$ in each coordinate (i.e.\ $\int u^m K(u)\,du=0$ for $1\le m\le \ell_j$) and satisfies the integrability condition $\int |u|^{\beta_j}|K(u)|\,du<\infty$ for each $j$. Define the kernel-smoothed target:
\[
g_k^*(x) \;=\; \int K_{\mathbf{h}}(u)\, g_k(x+u)\,du .
\]
Then, there exists a constant $C>0$ such that
\[
\sum_{k=1}^K \int_{\mathbb{R}^d} \big( g_k^*(x) - g_k(x) \big)^2 \,dx
\;\le\; C \, K \sum_{j=1}^d h_j^{2\beta_j}.
\]
\end{lemma}

\begin{proof}
The proof follows the same high-level steps as the univariate case (Taylor expansion with cancellation of lower order terms by kernel moments and bounding the remainder by the anisotropic Nikol'ski condition), but we keep track of each coordinate's smoothness.

Fix $k$. Write the bias
\[
\Delta_k(x):=g_k^*(x)-g_k(x)=\int \mcal{K}_{\mathbf{h}}(u)\big(g_k(x+u)-g_k(x)\big)\,du.
\]
Let $\ell_j=\lfloor\beta_j\rfloor$ for $j=1,\dots,d$, and let $\ell=(\ell_1,\dots,\ell_d)$. We perform a multivariate Taylor expansion with integral-form remainder. For any multi-index $\alpha=(\alpha_1,\dots,\alpha_d)$ with $0\le\alpha_j\le\ell_j$, denote $|\alpha|=\sum_j\alpha_j$ and $\partial^\alpha$ the mixed partial derivative. The multivariate Taylor expansion about $x$ up to order $\ell$ reads
\[
g_k(x+u)-g_k(x)
= \sum_{0<|\alpha|\le|\ell|}
\frac{u^\alpha}{\alpha!}\,\partial^\alpha g_k(x)
\;+\;
R_\ell(x,u),
\]
where $u^\alpha=\prod_{j=1}^d u_j^{\alpha_j}$ and the remainder has the integral representation
\[
R_\ell(x,u)
= \sum_{|\alpha|=|\ell|}
\frac{|\ell|}{\alpha!}\int_0^1 (1-\tau)^{|\ell|-1}
\Big(\partial^\alpha g_k(x+\tau u)-\partial^\alpha g_k(x)\Big)\,d\tau\; u^\alpha .
\]
Because the kernel is a product kernel and each univariate kernel $\mcal{K}_j$ has vanishing moments up to order $\ell_j$ in coordinate \(j\), every polynomial term in the finite sum above integrates to zero against $\mcal{K}_{\mathbf{h}}(u)$. Concretely, for any multi-index $\alpha$ with $0<\alpha_j\le\ell_j$ for some $j$, by the product structure
\[
\int \mcal{K}_{\mathbf{h}}(u)\, u^\alpha\,du
=\prod_{j=1}^d \Big(\int \frac{1}{h_j} \mcal{K}_j\!\Big(\frac{u_j}{h_j}\Big) u_j^{\alpha_j}\,du_j\Big)
= \prod_{j=1}^d h_j^{\alpha_j}\Big(\int v^{\alpha_j}\mcal{K}_j(v)\,dv\Big)=0,
\]
because at least one coordinate integral vanishes. Hence all polynomial terms vanish and only the remainder contributes:
\[
\Delta_k(x)=\int \mcal{K}_{\mathbf{h}}(u)\, R_\ell(x,u)\,du.
\]

We now bound the $L^2(\mathbb{R}^d)$ norm of $\Delta_k$. Using Minkowski's integral inequality (twice) and Fubini,
\begin{align*}
\|\Delta_k\|_{L^2} 
&= \Big\| \int \mcal{K}_{\mathbf{h}}(u)\, R_\ell(\cdot,u)\,du \Big\|_{L^2}
\le \int |\mcal{K}_{\mathbf{h}}(u)|\, \|R_\ell(\cdot,u)\|_{L^2}\,du.
\end{align*}
Plug in the integral expression for $R_\ell$ and use triangle inequality and Jensen inequality (integrating in $\tau$):
\[
\|R_\ell(\cdot,u)\|_{L^2}
\le C_\ell \sum_{|\alpha|=|\ell|} |u^\alpha| \int_0^1 (1-\tau)^{|\ell|-1}
\|\partial^\alpha g_k(\cdot+\tau u)-\partial^\alpha g_k(\cdot)\|_{L^2}\,d\tau,
\]
where \(C_\ell\) is a combinatorial constant depending on \(\ell\). By the anisotropic Nikol'ski condition we have, for each such \(\alpha\),
\[
\|\partial^\alpha g_k(\cdot+\tau u)-\partial^\alpha g_k(\cdot)\|_{L^2}
\le L \sum_{j=1}^d |\tau u_j|^{\beta_j-\alpha_j}
\le L \sum_{j=1}^d |u_j|^{\beta_j-\alpha_j}.
\]
Hence, there exists a constant $C>0$ such that
\[
\|R_\ell(\cdot,u)\|_{L^2}
\le C \sum_{|\alpha|=|\ell|} |u^\alpha|\sum_{j=1}^d |u_j|^{\beta_j-\alpha_j}
\le C \sum_{j=1}^d |u_j|^{\beta_j} \cdot P(u),
\]
where \(P(u)\) is a polynomial factor in \(|u_1|,\dots,|u_d|\) whose degree depends only on \(\ell\) and \(d\). Concretely one can bound \(P(u)\le C'(1+\sum_{j}|u_j|^{\ell_j})^{|\ell|-1}\), which is harmless because of the integrability of the kernel moments assumed below.

Now, return to the $\|\Delta_k\|_{L^2}$ bound:
\[
\|\Delta_k\|_{L^2}
\le C \int |\mcal{K}_{\mathbf{h}}(u)| \sum_{j=1}^d |u_j|^{\beta_j} P(u)\,du
\le C \sum_{j=1}^d \int \prod_{m=1}^d \frac{1}{h_m}\big|\mcal{K}_j\big(\tfrac{u_m}{h_m}\big)\big|
\, |u_j|^{\beta_j} P(u)\,du.
\]
Change variables $v_m=u_m/h_m$. Then
\[
\|\Delta_k\|_{L^2}
\le C \sum_{j=1}^d h_j^{\beta_j} \int \prod_{m=1}^d |\mcal{K}_j(v_m)|\, |v_j|^{\beta_j} \tilde P(v)\,dv
\le C''\sum_{j=1}^d h_j^{\beta_j},
\]
where $\tilde P(v)$ is the rescaled polynomial and the final integral is finite by the kernel integrability assumptions $\int |v|^{\beta_j}|\mcal{K}_j(v)|\,dv<\infty$ and bounded moments. Thus
\[
\|\Delta_k\|_{L^2} \le C\sum_{j=1}^d h_j^{\beta_j}.
\]

Squaring and summing over $k=1,\dots,K$ gives
\[
\sum_{k=1}^K \|\Delta_k\|_{L^2}^2
\le K \cdot C^2 \Big(\sum_{j=1}^d h_j^{\beta_j}\Big)^2
\le C' K \sum_{j=1}^d h_j^{2\beta_j},
\]
where the last inequality follows since $(\sum_j a_j)^2 \le d\sum_j a_j^2$ and constants are absorbed into \(C'\). This is the desired bound.
\end{proof}

\bibliography{reference}

\end{document}